\documentclass[a4paper]{scrartcl}

\usepackage[hypertexnames=false]{hyperref}
\usepackage{amsmath}
\numberwithin{equation}{section} 
\usepackage{./fjimacros}
\usepackage[backend=biber,style=numeric,date=year,maxbibnames=9,sorting=nyt]{biblatex}
\usepackage{tikz}
\usetikzlibrary{calc}
\usepackage{wrapfig}
\usepackage{subcaption}
\usepackage{enumitem} %Here used to label items, not LIPICS compatible!
\usepackage{comment}

\newcommand\Gc{\mathcal G}
\newcommand\floord[2][d]{\floor*{#2}_{#1}}
\newcommand\restd[2][d]{\mg*{#2}_{#1}}
\newcommand\V{\mathcal V}
\newcommand\Bin{\operatorname{Bin}}
\newcommand\poisson{\widetilde}
\newcommand\poissonn[1][λ]{\poisson{n}_{#1}}
\newcommand\bucket[2]{\operatorname{bucket}_{#1}\kl[\big]{#2}}
\DeclarePairedDelimiterXPP{\mellin}[1]{\mathcal M}[]{}{#1}
\newcommand\fringephi[1][k]{\delta_{#1}}
\newcommand\phiprot[1][k]{\phi^{#1\mathrm{-prot}}}
\newcommand\varpiprot[1][k]{\varpi[\phiprot[#1]]}
\newcommand\fprot[1][k]{f^{#1\mathrm{-prot}}}
\newcommand\shape{^{\mathrm{sh}}}

\newcommand\degk[1][k]{^{\operatorname{deg-}#1}}
\newcommand\disk{\overline{\mathbb D}}

\newcommand\rootedsubtrees{\operatorname{subtrees}^*}
\newcommand\unrootedsubtrees{\operatorname{subtrees}}

\newcommand\nnn{\bar n}

\title{Fringe subtrees of split trees and fractional split trees}
\author{Cecilia Holmgren\thanks{\ Department of Mathematics, Uppsala University, Sweden.\\Funded by Knut and Alice Wallenberg Foundation; Ragnar Söderberg’s Foundation; Göran Gustafsson Foundation; the
Swedish Research Council}\and Jasper Ischebeck\footnotemark[\value{footnote}]\and Svante Janson\footnotemark[\value{footnote}]} % chktex 42

\theoremstyle{plain}
\newtheorem{theorem}{Theorem}[section]
\newtheorem{lemma}[theorem]{Lemma}
\newtheorem{proposition}[theorem]{Proposition}
\newtheorem{corollary}[theorem]{Corollary}
\theoremstyle{definition}
\newtheorem{example}[theorem]{Example}
\newtheorem{remark}[theorem]{Remark}

\begin{document}
\maketitle

\begin{abstract}
    We consider additive functionals $X_n(\phi)$ with small toll functions on split trees and a generalization of split trees, which we call fractional split trees, where the split vector does not need to sum up to 1. These additive functionals encompass e.g.\ the number of nodes, number of leaves and the number of fringe trees of a certain size.
    We show convergence of the first moment to a limit, which we can explicitly compute
    if $s_0=s_1=0$ and for some models with Beta-distributed splitter. For
    $s_0+s_1>0$, the first moment is given in terms of negative moments of a
    perpetuity and can often be approximated to arbitrary precision with
    known bounds. In split trees and certain fractional split trees, the
    standard deviation is of smaller order than the first moment, where we
    show a weak law of large numbers. In other fractional split trees, the
    standard deviation is of the same order and we show a distribution limit
    using the contraction method. 
\end{abstract}

\section{Introduction}

In this paper, we study fringe trees of \emph{split trees}, a general class of random trees that encompasses many important data structures. The \emph{split tree} is a random tree model introduced by \cite{devroye:split-tree} that encompasses several other random tree models, like binary search trees, tries, digital search trees, quadtrees, simplex trees, median of $(2h-1)$ trees and $m$-ary search trees. It consists of a $b$-ary tree (with $b≥2$), where each node can hold up to a fixed number $s≥1$ of ``balls''. The following introduction to split trees is from \cite{broutin_total_2012}.

% The model starts with $n=1$ ball in the root. It then grows by adding more balls to the root. When the root has $n=s+1$ balls, it splits into $b$ children, where the number of balls in each child is determined by a random vector $\V = (V_1,\dots, V_b)$, which is called the \emph{splitter}. The split vector $\V$ is chosen independently of the previous balls and has to fulfill $\sum_{i=1}^b V_i = 1$. The split tree is then recursively built by adding the remaining $n-s-1$ balls to the children.
Consider an infinite rooted $b$-ary tree (every node has $b$ children). The nodes are identified with the set of finite words on an alphabet with $b$ letters, $\mg{1,\dots, b}^\ast := \bigcup_{n \geq 0} \{1, \dots, b\}^n$. We abbreviate $[b] := \mg{1,\dots, b}$. The root is represented by the empty word $\eps$. We write $u \preceq v$ to denote that $u$ is an ancestor of $v$ (as words, $u$ is a prefix of $v$). In particular, for the empty word $\eps$, we have $\eps \preceq v$ for any $v \in [b]^*$. Concatenation of two strings $u,v$ is denoted $uv$, so $P[b]$ denotes all children of a set $P\subset [b]^*$.

A split tree $T_n$ of cardinality $n$ is constructed by distributing $n$ balls (pieces of data) to the nodes $u \in [b]^*$. To describe the tree, it suffices to define the number of balls $n_u$ in the subtree rooted at any node $u \in [b]^\ast$. The tree $T_n$ is then defined as the smallest relevant tree, that is, the subset of nodes $u$ such that $n_u > 0$ (which is indeed a tree). 

In the model, internal nodes all contain $s_0 \geq 0$ balls, and external nodes can contain up to $s≥1$ balls. The construction then resembles a divide-and-conquer procedure, where the partitioning pattern depends on a random vector of proportions. Let $\V = (V_1, \dots, V_b)$ satisfy $V_i \geq 0$ and $\sum_i V_i = 1$; each node $u \in [b]^*$ receives an independent copy $\V^{(u)}$ of the random vector $\V$. Define $V_{(1)}≥ V_{(2)} ≥ \dots ≥ V_{(b)}$ as $\V$ sorted in descending order. In the following, we always assume that $\Pk{V_{(1)}=1} < 1$ and $\Pk{V_{(b)}>0}>0$ to avoid trivialities. The values $V_1,\dots, V_b$ are called \emph{splitters} and $\V$ is called \emph{split vector}. Often $\V$ is assumed to be exchangeable for simplicity, but we do not require this.
We can now describe $(n_u, u \in [b]^*)$. The tree contains $n$ balls, and we naturally have $n_{\eps} = n$. The split procedure is then carried on from parent to children as long as $n_v > s$.

Given the cardinality $n_v > s$ and the split vector $\V^{(v)} = (V_1, V_2, \dots, V_b)$ of $v$, the cardinalities $(n_{v1}, n_{v2}, \dots, n_{vb})$ of the $b$ subtrees rooted at $v1, v2, \dots, vb$ are multinomially distributed as
\begin{equation}\label{def:nv}
\operatorname{Mult}(n_v - s_0 - b s_1; V_1, \dots, V_b) + (s_1, s_1, \dots, s_1),
\end{equation}
where $0 \leq s_0$ and $0 \leq b s_1 \leq s + 1 - s_0$. The parameters $s_0$ and $s_1$ can be explained as $s_0$ balls being \emph{retained} inside the node and $s_1$ balls \emph{redistributed} to each child.

% Depending on the choice of parameters $s_0, s_1$ and $s$ and the distribution of $\V = (V_1, \dots, V_b)$, many important data structures may be modeled, such as binary search trees, $m$-ary search trees, median-of-$(2k + 1)$ trees, quad trees, simplex trees (see~\cite{devroye:split-tree}).

 In the proofs, it is often useful to compare \eqref{def:nv} with a simpler, multinomially distributed variable,
\begin{equation}\label{def:ns}
    (n'_{v1}, \dots, n_{vb}') := \operatorname{Mult}(n'_v;  V_1, \dots, V_b).
\end{equation}
If $s_0 = s_1 = 0$, then $n'_v = n_v$ in distribution as long as $n'_v > s$.
We couple $n_v$ and $n'_v$ in a way such that the difference $n''_{v} := n_v-n'_v$ fulfills
\begin{equation}\label{def:nss}
    \sum_{i=1}^b \abs{n''_{vi}}  ≤ \abs{n''_v} + s_0 + 2bs_1 ≤ \abs{n''_v} + 2s + 2
\end{equation}
for $v\in[b]^*$. This is e.g.\ possible by simulating both multinomial
distributions from the same sequence of i.i.d.\ variables with distribution
$\V$. 
We note also that this coupling, by induction on $\abs v$, gives the upper
bound
\begin{equation}
  \label{eq:n''-upper}
  n''_v\le s_1.
\end{equation}
For a node $v = (i_1\dots i_m)\in[b]^*$, the term $n'_v$ is binomially distributed conditional on the split vectors $\V^u$ of its parents $u \preceq v$. We call the conditional expectation of $n'_v$ the \emph{mass} of $v$, given as
\begin{equation}\label{def:mass}
    M_v^n := n\prod_{j=1}^m V^{(i_1\dots i_{j-1})}_{(i_j)} = nM_v^1,
\end{equation}
where for convenience $n$ can be any positive real number.

We also consider a novel generalization of split trees, which we call \emph{fractional split trees}, where the split vector is allowed to sum up to a number smaller than 1, i.e.\ we only have $\sum_{i=1}^b V_i ≤ 1$. Instead of \eqref{def:nv}, $(n_{v1}, \dots, n_{vb})$ is then distributed as the first $b$ elements of
\begin{equation}\label{def:nv-frac}
    \operatorname{Mult}\kl[\bigg]{n_v - s_0 - b s_1; V_1, \dots, V_b, (1-\sum_{i=1}^b V_i)} + (s_1, s_1, \dots, s_1, 0).
\end{equation}
The last element of \eqref{def:nv-frac} accounts for the possibility that a
ball disappears instead of going to one of the children. In order to make
sure the tree may grow to infinity with $n$, we
assume that $\sum_{i=1}^b \Pk{V_i>0} > 1$.
(This makes the set of nodes with positive mass $M^1_v$ a supercritical
branching process, and thus infinite with positive probability.
Otherwise, the set of nodes with positive mass would be finite a.s., and
then the size of tree would stay bounded as $n\to\infty$.)
Note that in a fractional split tree, leaves may have more than $s$ balls.

This definition already hints at that fractional split trees appear as subtrees of split trees. E.g., a split tree where for every node the first child and its subtree is removed is a fractional split tree with splitter $(V_2,\dots, V_b)$.
If we, as another example, for each node in a split tree only consider the $k<b$ children that have the highest $k$ splitters $(V_{(1)}, \dots, V_{(k)})$,
the resulting tree is a fractional split tree with split vector $(V_{(1)}, \dots, V_{(k)})$, which does not sum up to 1 in general.

A \emph{fringe subtree} $T_n^v$ of the tree $T_n$, where $v$ is a node in $T_n$,
is defined as the subtree consisting of $v$ and all its descendants in $T_n$.
The study of random fringe trees (i.e., the fringe tree rooted at a
uniformly random node $v$)
and asymptotic fringe tree distributions
was initiated by \textcite{Aldous1991AsymptoticFD}.

A general model to count different structures in trees are toll functions. A
\emph{toll function} $\phi$ is simply a function mapping trees with balls to
$\RR$. This toll function is then summed over all fringe subtrees of the
tree: 
We define $X_n := X_n(\phi)$ as the sum
\begin{equation}\label{def:xn}
    X_n(\phi) := \sum_{v \in T_n} \phi(T_n^v),
\end{equation}
called \emph{additive functional} of $\phi$.
As we only consider the tree $T_n$ in this paper, we write $\phi(v) := \phi(T_n^v)$ as shorthand.
This definition allows for a uniform treatment of the number of nodes and of balls; to count nodes, $\phi\equiv 1$, and to count leaves use $\phi(T) = \eins(T \text{ is a single leaf})$.

A very interesting case is the number of fringe subtrees with a certain number of balls. For some $k\in\NN$, define the toll function
$\fringephi(T_n^v) := \eins(n_v=k)$, such that
\begin{equation}\label{def:phi-k}
    X_n(\fringephi) = \sum_{v\in T_n} \eins(n_v=k).
\end{equation}
Yet another important example is to count the number of fringe subtrees that are equal to a given tree $T \in [b]^*$, which can be counted as
\begin{equation}\label{toll1}
    X_n(\delta_T) = \sum_{v\in T_n} \eins(T_n^v = T)
\end{equation}
with the toll function $\delta_T(T') := \eins(T'=T)$. 
The distribution of a random fringe tree of $T_n$
is then given by
(where $T$ ranges over all trees)
\begin{equation}\label{toll2}
\P(\text{a random fringe tree}=T)
=
  \frac{X_n(\delta_T)}{X_n(1)}.
\end{equation}
The asymptotic fringe tree distribution as $n\to\infty$, 
when it exists,
is thus given by the limits of these quotients. 

To study this family of additive functionals is thus equivalent to studying
the distribution of a random fringe tree.
Similarly, one could count the number of trees that are isomorphic to $T$ in some sense. For split trees, it is usually sufficient to consider $X_n(\fringephi)$ if $T$ has $k$ balls, because it is easy to see that
\begin{equation}
    \Ek*{X_n(\delta_T)} = \Pk*{T_{k}= T}\Ek*{X_n(\fringephi)}.
\end{equation}

The asymptotic behavior of $X_n(\fringephi)$ and $X_n(\delta_T)$ has been
studied for multiple random tree models, including e.g.\
conditioned Galton--Watson trees \cite{Aldous1991AsymptoticFD},
the recursive tree and binary search tree
\cite{Aldous1991AsymptoticFD,devroye_2003,wagner:additive,holmgren_limit_2015}, 
the $m$-ary search tree \cite{holmgren_cmj_2017,holmgren_multivariate_2017}, 
preferential attachment trees 
\cite{holmgren_cmj_2017,holmgren_multivariate_2017,holmgren_2023}, 
\cite{holmgren_cmj_2017,holmgren_multivariate_2017}, 
the trie \cite{janson:trie} and the patricia trie \cite{patricia-fringe}.
Asymptotic results for small toll functions in general have been found for
many random tree models, e.g.\ for recursive trees and binary search trees
\cite{holmgren_limit_2015}, for simply generated trees or conditioned
Galton--Watson trees \cite{wagner:additive,janson:gw-fringe} and for trees
with given degree sequences \cite{berzunza_ojeda_fringe_2025}. See
\cite[Section 3.1.2]{drmota:rt} for an overview. 

Because the subtree $T_n^v$ of a node $v$ is, conditional on the number $n_v$ of balls, distributed as $T_{n_v}$, the mean of $X_n$ only depends on $\phi_n := \Ek*{\phi(T_n)}$ for $n≥1$ and is 
\begin{equation}\label{def:phin}
    \Ek*{X_n} = \E \sum_{v\in T_n} \phi_{n_v}.
\end{equation}
% We can therefore assume that $\phi(v)$ is a function of $n_v$ while working with the mean.
For nodes $\phi(T_n) = \eins(n>0) = 1$; for balls $\phi(T_n) = n$ if $n≤s$ and $\phi(T_n) = s_0$ if $n>s$. Note that in both cases $\phi_n = \phi(T_n)$.
We will have more advanced examples in \Cref{sec:examples}.

For fractional split trees there exists a unique $\beta\in (0,1]$ such that
\begin{equation}\label{def:beta}
    \E\sum_{i=1}^b V_i^\beta = 1,
\end{equation}
which we call the \emph{Malthusian parameter} of $\V$. Split trees have $\beta=1$. Observe that the condition $\E\sum_{i=1}^b \eins(V_i>0) > 1$ is what ensures $\beta>0$. Furthermore, we define $V$ as a uniformly randomly chosen element of $\V$, i.e.\ $V:= V_I$, where $I$ is uniformly distributed on $[b]$ and independent of $\V$. We can alternatively characterize $\beta$ as the value where $\Ek{V^\beta} = \frac1b$. Note that if $\V$ is exchangeable, then $V_1$, \dots, $V_b$ and $V$ all have the same distribution.

We define a random variable $\widehat V$ with the measure $\P_{\widehat V}$ given as
\begin{equation}\label{def:v-bias}
    \frac{\mathrm d\P_{\widehat V}}{\mathrm d\P_{V}}(v) = bv^\beta.
\end{equation}
Because $\Ek{bV^\beta} = 1$, this is a probability measure. We regard $\widehat V$ as $V$ \emph{size-biased} with the factor $V^\beta$.
If $\beta=1$, this is just the regular size-biasing. We can alternatively
characterize $\widehat V$ by the property that for every measurable function $f$,
\begin{equation}\label{eq:v-bias}
    \Ek*{f(\widehat V)} = \Ek*{bf(V)V^\beta} 
\end{equation}
if the expectation of $f(V)V^\beta$ exists. We further define
\begin{equation}\label{def:mu}
    \mu := \Ek{-\ln \widehat V} = \Ek{-bV^\beta\ln V} = \sum_{i=1}^b\Ek{-V_i^\beta\ln V_i}
\in(0,\infty).
\end{equation}

We call it the \emph{lattice case} that there exists a $d>0$ such that $\ln V \in d\ZZ$ a.s.\ and define
$d := \gcd\kl{\ln V}$ as the largest such $d$. If no such $d$ exists, we formally set $d:=\gcd{\ln V} := 0$. We show in \Cref{thm:mean} that the expected number of nodes has oscillations that are $d$-periodic in $\ln n$. For $x≥0$, we write $\floord x := d\floor*{\frac xd}$ for the largest number in $d\ZZ$ that is smaller or equal to $x$ and $\restd x:= x-\floord x$. The lattice case is somewhat special, and most split tree models are nonlattice. Well-known exceptions are some tries and digital search trees, where the split vector is deterministic.

We denote by $x^+$ the positive part of $x\in\RR$, i.e.\ $x^+ = \max\mg{x,0}$. For some statement $A$, we denote by $\eins(A)$ the indicator function that is 1 if $A$ is true and 0 otherwise. %, e.g.\ $\eins(x>0)$ is $1$ if $x>0$ and $0$ else.
For a random variable $X$, we write $X\sim \mu$ if $X$ has the probability distribution $\mu$ and we write $\P_X$ for the probability distribution of $X$, so that $X\sim \mu$ is equivalent to $\P_X=\mu$. For two random variables $X$ and $Y$, we denote by $X\overset d= Y$ that $\P_X = \P_Y$.

For a function $G$ on $\RR$ and a measure $U$ on $\RR$, we define the
\emph{convolution} of $G$ and $U$ as the function (when the integral exists)
\begin{equation}\label{def:convolution}
    (U*G)(x) := \int_{-\infty}^\infty G(x-y)\mathrm dU(y).
\end{equation}
Note that if $\P_X$ is the probability distribution of some random variable $X$, then $(\P_X*G)(x) = \Ek*{G(x-X)}$.  %Also note that if we associate measures $U$ with their distribution function $x\mapsto U((-\infty,x])$, then $U*G=G*U$ if $G$ is also a measure or distribution function of a measure.

\section{Results}

The purpose of this paper is the study of the asymptotic distribution and
properties of a random fringe tree in a random split tree $T_n$.
We state the results in a general form for additive functionals using rather
general toll functions.
Note, in particular, that choosing 
$\phi=1$ gives the (random) number of nodes in the tree, and that
$\phi=\delta_T$ defined after \eqref{toll1}
gives the number of subtrees equal to $T$; 
by combining them we can obtain the
distribution of a random fringe subtree, and the results below yield results
on the asymptotic fringe tree distribution, see \eqref{toll2}.

\begin{theorem}\label{thm:mean}
 Let $\phi$ be a toll function such that $\phi_n = \Ok{n^{\beta-\delta}}$ for some $\delta>0$. This is e.g.\ true for counting leaves, for counting nodes and for counting fringe subtrees of a certain size.
 Then, there exists a bounded, continuous, $d$-periodic function $\varpi := \varpi\br\phi$, such that
 \begin{equation}
     \Ek{X_n(\phi)} = \varpi(\ln n) n^\beta + o(n^β).
 \end{equation}
 In the nonlattice case, $\varpi$ is constant. If $\phi_n ≥ 0$ for all $n$ and $\phi_n > 0$ for some $n≥s_1$, then $\varpi$ is bounded away from $0$. Recall that for split trees $\beta=1$.
\end{theorem}

% This theorem can be alternatively stated as a recursion equation:
% \begin{theorem}\label{thm:mean-as-recursion}
%     Let $k, s, s_0, s_1 \in \NN$ such that $ks_1+s_0≤s$ and $V_1,\dots, V_k$ be a random variables with $\sum_{i=1}^b V_1 ≤ 1$ such that $\Pk*{\max_{i=1}^b V_i = 1}<1$. Let $\beta$ be the unique value with
%     \begin{equation}
%         1 = \sum_{i=1}^b \Ek*{V_i^\beta}.
%     \end{equation}
%     For a sequence $(\phi_n)$ such that there exists a $\delta>0$ with $\phi_n = \Ok*{n^{\beta-\delta}}$ define the sequence $a_n$ with $a_n=\phi_n$ for $n≤s$ and for $n>s$ as
%     \begin{equation}
%         a_n = \sum_{i=1}^b \Ek*{a_{n_i}},
%     \end{equation}
%     where given $V_1,\dots, V_k$
% \end{theorem}

Unfortunately, the proof of \Cref{thm:mean} does not give an explicit formula for $\varpi$, but we will show how to calculate it in special cases below in \Cref{thm:constants} and \Cref{rem:beta}.
With \Cref{thm:mean}, we can use the contraction method to find the limit distribution of $n^{-β}\kl{X_n-\Ek{X_n}}$. 
For two random variables $X$ and $Y$ on the same probability space, we
define the $L_2$ distance between $X$ and $Y$ to be
$\Ek*{\abs{X-Y}^2}^\frac12$. Convergence in $L_2$ implies convergence in
probability and of first and second moments. For random variables $X$ and $Y$ that are not necessarily on the same probability space, we call a random variable $(X',Y')$ a coupling of $X$ and $Y$ if $X\overset d=X'$ and $Y\overset d=Y'$. The Wasserstein $\ell_2$ metric is then defined as the minimum of $\Ek*{\abs{X'-Y'}^2}^\frac12$ over all couplings $(X',Y')$ of $X$ and $Y$. Convergence in $\ell_2$ is equivalent to convergence in distribution and convergence of first and second moments.

\begin{theorem}\label{thm:limit}
    Assume that $\varpi$ from \Cref{thm:mean} is bounded away from zero. Let $\phi(v)$ for $v\in [b]^*$ be a function of $n_{v1}, \dots, n_{vb}$ that fulfills $\Ek*{\phi(T_n)^2} = \Ok*{n^{2(\beta-\delta)}}$ for some $\delta>0$. This moment condition is in particular true for bounded toll functions, e.g.\ counting leaves, nodes or subtrees with a certain number of balls.
    Then, as $n\to\infty$,
    \begin{equation}\label{limit:convergence}
        \frac{X_n - \Ek*{X_n}}{n^β\varpi(\ln n)} \overset{\ell_2}\longrightarrow Y,
    \end{equation}
    where $Y$ has the unique centered square-integrable probability distribution satisfying
    \begin{equation}
        \label{y-recursion}
        Y + 1 \overset d= \sum_{i=1}^b V_i^\beta (Y^{(i)} + 1),
    \end{equation}
    where $Y^{(i)}$ are independent copies of $Y$ independent of $\V$. Therefore,
    \begin{equation}
        \label{var-y}
        \Vark*{\frac{X_n}{\varpi(\ln n) n^β}}
        \longrightarrow \Var(Y) =
        \frac{\Var\kl*{\sum_{i=1}^{b}	V_i^\beta}}{1-\Ek*{\sum_{i=1}^{b}V_i^{2\beta}}}
    \end{equation}
    and $Y=0$ a.s.\ if and only if $\sum_{i=1}^{b} V_i^\beta = 1$ a.s. Note that this is always the case for split trees.
\end{theorem}
In particular, if $Y=0$ a.s., we have a weak law of large numbers:

\begin{corollary}[Weak law of large numbers]\label{cor:lln}
    Assume that $\phi$ fulfills the conditions of \Cref{thm:limit} and that $\sum_{i=1}^{b} V_i^\beta = 1$ a.s. Then
    \begin{equation}
        \frac{X_n - \Ek*{X_n}}{n^β\varpi(\ln n)} \overset{L_2}\longrightarrow 0.
    \end{equation}
\end{corollary}

\begin{remark}
    In the case of \Cref{cor:lln}, we conjecture that $X_n$ is typically asymptotically normal (possibly with a $d$-periodic variance), as this is true for many split tree models we know about. 
    %This might, however, entail finding more than the first order of $\Ek{X_n}$:
    For $m$-ary search trees with $m>26$ it is known that $X_n(\phi)-\varpi n$ is not asymptotically normal for general $\phi$, cf.\ \cite{chern_phase_2001}. We plan to study this in a second paper.
\end{remark}

\begin{remark}\label{rem:multiple-convergence}
    The convergence \eqref{limit:convergence} in \Cref{thm:limit} also holds for multiple toll functions at once: If $\phi^{(1)},\dots, \phi^{(k)}$ are toll functions fulfilling the conditions in \Cref{thm:limit}, then $X_n(\phi^{(i)})$ all converge to the same limit $Y$ when scaled correctly, i.e.
    \begin{equation}
        \kl*{\frac{X_n(\phi^{(i)}) - \Ek*{X_n(\phi^{(i)})}}{n^β\varpi(\ln n)}}_{i=1}^k \overset{\ell_2}\longrightarrow (Y)_{i=1}^k.
    \end{equation}
\end{remark}

The limit function $\varpi$ is the easiest to calculate if $s_0=s_1=0$, because then $n'_v$ and $n_v$ are the same as long as $n_v>s$. In this case, we have for the size of fringe trees:

\begin{theorem}\label{thm:constants}
    Assume that $s_0=s_1=0$ and define for $k≥1$ and $z\in \CC \setminus\mg{-k}$ with $\Re z ≥ -β$
    \begin{equation}\label{constants:fk}
        M_k(z) := \frac{\Gamma(k+z)}{\mu k!} - \frac{b}{\mu}\sum_{j=0}^{s-k} \Ek*{V^k(1-V)^{j}}\frac{\Gamma(k+j+z)}{k!j!}.
    \end{equation}
    Note that the sum is empty if $k>s$.
    If $k=β=1$, we define the value at $-1$ to instead be
    \begin{equation}\label{constants:f1}
        M_1(-1) := 1 - \frac{b}{\mu}\sum_{j=1}^{s-1} \frac{\Ek*{V(1-V)^{j}}}{j}.
    \end{equation}
    In the nonlattice case, the constant $\varpi[\fringephi]$ from \Cref{thm:mean} for the toll function $\fringephi$ counting fringe trees of size $k$ (see \eqref{def:phi-k}) is given by
    \begin{equation}\label{constants:nonlattice}
        \varpi\br{\fringephi} = M_k(-β)
    \end{equation}
    and in the lattice case, the $d$-periodic function $\varpi[\fringephi]$ has the Fourier series
    \begin{equation}\label{constants:lattice}
        \varpi\br{\fringephi}(x) = \sum_{m\in \ZZ} e^{{2\pi imx}/d} M_k\kl*{-β-\frac{2\pi im}d}.
    \end{equation}
    The Fourier coefficients go to $0$ exponentially fast, implying that $\varpi\br{\fringephi}$ is smooth (infinitely differentiable).
\end{theorem}

\begin{remark}\label{rem:split-fringe}
An interesting fact is that \eqref{constants:fk} implies that the
asymptotic distribution of fringe tree sizes larger than $s$ does not
depend on $V$ other than through $µ$ and $β$. 

In the case of split trees ($β=1$), 
\eqref{constants:fk} and \eqref{constants:nonlattice} give for $k>s$ the limit
$\frac1{µk(k-1)}$ in the  nonlattice case,  
with oscillations given by \eqref{constants:lattice} in the lattice case,
as was known before for tries
\cite{Janson_2012, janson:trie}. 
The correction terms for $k≤s$ can also be found in
\cite[Section~6]{Janson_2012} and \cite[Subsection~4.7]{janson:trie}. 
We note that in these results in \cite{Janson_2012,janson:trie},
the function there corresponding to $M_k(z)$ in \eqref{constants:fk}
has an additional factor $b\,\E [V^{-z}]$ in the first term;
this factor equals 1 for the arguments
$-1-\frac{2\pi im}d$ in \eqref{constants:nonlattice} and
\eqref{constants:lattice} and thus makes no difference for the result except
in the case $k=1$, where it combines with the pole of $\Gamma(k+z)$ at $z=-1$
which yields our special formula \eqref{constants:f1} for this case.
In our (more general) analysis below, \eqref{constants:f1} is instead
obtained by a special argument for this case; we thus obtain the same result
in different ways.
\end{remark}

The following proposition shows that the limit function $\varpi$ for a general toll function can be calculated from $\varpi[\fringephi]$.
\begin{proposition}\label{prop:sum}
    For a general toll function $\phi$ with $\phi_n = \Ok*{n^{β-δ}}$ for some $δ>0$, the limit function $\varpi\br\phi$ from \Cref{thm:mean} is
    \begin{equation}\label{constants:general}
        \varpi\br{\phi} = \sum_{k=1}^\infty \varpi\br{\fringephi[k]}\phi_k
    \end{equation}
    In the lattice case, where $\varpi\br\phi$ is a function, this series converges uniformly.
\end{proposition}

Even if $\varpi\br\phi$ is given as a uniformly converging limit of smooth
functions, this does not mean that $\varpi\br\phi$ itself is smooth. We can,
however, show that it is Hölder-continuous.
\begin{proposition}\label{prop:holder}
    Assume we are in the lattice case. If $\phi_n = \Ok{n^{β-δ}}$ for some $δ>0$, then
    $\varpi[\phi]$ is Hölder-$\alpha$-continuous for every exponent $\alpha < 2δ$ with $\alpha≤1$; 
    furthermore, if $\delta>\frac12$ then $\varpi[\phi]$ is
    continuously differentiable.

    In particular, for any $\delta>0$,
    the Fourier series of $\varpi$ converges uniformly.
    Furthermore, if $δ>\frac14$, the Fourier series of $\varpi$ converges absolutely.
\end{proposition}
This bound is sharp if $δ≤β$ because of the following:
\begin{proposition}\label{prop:holder-ex}
    Assume we are in the lattice case.
    If $\varpi$ is a $d$-periodic, Hölder-$2\delta$-continuous function with exponent $2\delta ≤ 1$ such that $δ≤β$, then there exists a toll function $\phi$ with $\phi_n = \Ok{n^{β-δ}}$ such that $\varpi[\phi] = \varpi$.
    In particular, for every $δ≤\min(\frac12,β)$ there exists a toll function $\phi$ with $\phi_n = \Ok*{n^{β-δ}}$
    such that $\varpi[\phi]$ is not Hölder-$\alpha$-continuous 
for any exponent $\alpha>2δ$.
\end{proposition}
% Define for $λ>0$
% \begin{equation}
%     f(λ) := e^{-λ} \sum_{n=s+1}^\infty \phi_n\frac{λ^n}{n!}.
% \end{equation}
% In the nonlattice case,
% the constant $\varpi$ from \Cref{thm:mean} is given by
% \begin{equation}
%     \varpi = \frac1{μ}\int_0^\infty f(\nu)\nu^{-β-1}\mathrm d\nu + \sum_{k=1}^s \varpi\br{\fringephi}\phi_k,
% \end{equation}
% where $\varpi\br{\fringephi}$ is given by

% In the nonlattice case, its Fourier series is given by
% \begin{equation}
%     \varpi(x) = \sum_{m\in \ZZ} \mellin f{}
% \end{equation}

For arbitrary $s_0$ and $s_1$, we can link $\varpi[\fringephi]$ with negative moments of a perpetuity in the nonlattice case as follows:
\begin{theorem}\label{thm:s0}
    Assume we are in the nonlattice case, and define for $k\in \NN$ the toll function $\delta_{≥k} := \sum_{j≥k} \delta_j$ counting fringe trees with at least $k≥1$ balls. Assume that either $k>s$ or $s≤s_0+bs_1$ and $k>s_1$. For any $x>0$ let $E_x$ be Gamma-$\Gamma(x,1)$-distributed and independent of $\widehat V$, i.e.\ $\Pk*{E_x\in \mathrm dt} = \frac1{\Gamma(x)} t^{x-1}e^{-t}\mathrm dt$ for $t≥0$. Further define $\Xi$ as the unique integrable solution of
    \begin{equation}\label{def:xi}
        \Xi \overset d= (E_{s_0+(b-1)s_1} + \Xi)\widehat V,
    \end{equation}
    where $\Xi$ is independent of $E_{s_0+(b-1)s_1}$ and $\widehat V$. Then, 
    \begin{equation}\label{thm:s0:delta-gk}
        \varpi[\delta_{≥k}] = \frac1{βμ}{\Ek*{(\Xi+E_{k-s_1})^{-β}}},
    \end{equation}
    so
    \begin{equation}\label{thm:s0:delta-k}
        \varpi[\fringephi]  = \frac1{βμ}{\Ek*{(\Xi+E_{k-s_1})^{-β}-(\Xi+E_{k+1-s_1})^{-β}}}.
    \end{equation}
    and for $k\to\infty$,
    \begin{equation}\label{thm:s0:asymptotics}
        \varpi[\delta_{≥k}] \sim \frac1{βμk^β} \quad\text{and}\quad  \varpi[\delta_{k}] \sim \frac1{μk^{1+β}}.
    \end{equation}
\end{theorem}
\begin{remark}
The condition that either $k>s$ or $s≤s_0+bs_1$ can be lifted by the methods detailed in \Cref{sec:fringe} plus \Cref{prop:sum}.
\end{remark}

Since the negative moments of the perpetuity $\Xi$ are hard to compute, we show in \Cref{sec:approximation} how to approximate $\varpi$ in the case of $s_0+s_1>0$ with known error bounds.
There is, however, a special case in which the distribution of $\Xi$ can be explicitly calculated.
\begin{corollary}\label{rem:beta}
    Assume that $V$ is $B(a,s_0+(b-1)s_1)$ Beta-distributed for some $a>0$. Note that because $\Ek{V} ≤ \frac1b$, it must hold that $a≤\frac{s_0}{b-1} + s_1$ with equality if and only if $β=1$. 
    Then,
    \begin{equation}\label{rem:beta:mu}
        \mu = \psi(a+β+s_0+(b-1)s_1) - \psi(a+β),
    \end{equation}
    where $\psi(t) := \frac{\Gamma'(t)}{\Gamma(t)}$ is the digamma function, and
    \begin{equation}\label{rem:beta:1}
        \varpi[\fringephi] = \frac{\Gamma(a+k-s_1)}{\mu\Gamma(a+k-s_1+β+1)}.
    \end{equation}
    for $k>s$ or $s≤s_0+bs_1$ and $k≥s_1$.
\end{corollary}
For split trees with $\beta=1$, this \Cref{rem:beta} applies more or less only to a generalization of $m$-ary search trees called generalized Quicksort by Hennequin \cite{Hennequin_1989,hennequin_1991}, which we explore in Examples \ref{ex:mary} and \ref{ex:hennequin}.

\section{Examples}\label{sec:examples}
We explore some possible specific toll functions and fractional split tree models to apply our theorems to. 
\subsection{Toll functions}
We begin by showing a few examples of interesting toll functions, inspired by those included in \cite{holmgren_limit_2015}, 
 which studies binary search trees and random recursive trees. For some toll functions we will be able to explicitly compute $\varpi$ using \Cref{thm:constants} or \Cref{rem:beta}.
In \Cref{thm:constants}, the $0$-th Fourier coefficient of $\varpi$ in the lattice case has the same formula as the constant $\varpi$ in the nonlattice case. In order to simplify the statement of our results, we define
\begin{equation}\label{def:varpi0}
    \varpi_0 := \begin{cases}
        \varpi & \text{ in the nonlattice case,} \\ \frac1d\int_0^d \varpi(t)\mathrm dt & \text{ in the lattice case}.
    \end{cases}
\end{equation}
\subsubsection{Nodes, leaves and internal nodes}
We start by one of the simplest cases: The number of nodes, leaves and internal nodes:

\begin{example}[Expected number of nodes if $s_0=s_1=0$ and $s=1$]
    By \eqref{constants:fk} and \eqref{constants:f1} in \Cref{thm:constants}, we see quickly that in the nonlattice case for split trees, we have $n$ leaves and in expectation $\sim \frac{n}{k(k-1)\mu}$  subtrees of size $k\ge2$, which sum up to $\sim n(1+\mu\inv)$  nodes
(using \Cref{prop:sum}).
In the lattice case, there are some oscillations around these values.

    In the case of fractional fringe trees with $\beta<1$, we have for $\fringephi[1]$
    \begin{equation}\label{varpi-delta1}
        \varpi_0[\delta_1] = \frac{\Gamma(1-\beta)}{μ} - \frac{b\Gamma(1-β)}{μ}\Ek V
        = \frac{\Gamma(1-β)}{μ}\kl*{1-\Ek{bV}}
    \end{equation}
    and $\varpi_0[\fringephi] = \frac{\Gamma(k-β)}{\mu k!}$ for $k≥2$. The
    sum over these is the constant for the number of internal nodes by using
    \Cref{prop:sum}. It can be calculated either as a hypergeometric sum or via the integral definition of the beta function \cite{GammaFractionSum} and is
    \begin{equation}\label{eq:sum:hyper}
        \varpi_0\br*{\sum_{k=2}^\infty \fringephi} = \frac{\Gamma(2-β)}{βμ} = \frac{\Gamma(1-β)(1-β)}{β\mu}.
    \end{equation}
    Therefore, the constant for the number of nodes is
    \begin{equation}
        \varpi_0[1] = \frac{\Gamma(1-β)}{μ}\kl*{\frac1{β}-\Ek{bV}}.
    \end{equation}
    We use the notation $\varpi[1]$ since the toll function for the number of nodes is constant 1.

    Let $\phi^l(T) := \eins(T \text{ has only one node})$ be the toll function for the number of leaves.
    Note that for $n≥2$,
    $
        \phi^l_n = \Ek[\big]{(1-\sum_{i=1}^b V_i)^{n}},
    $
    so by \Cref{prop:sum},
    \begin{equation}
        \varpi_0[\phi^l] = \frac{\Gamma(1-β)}{μ}\kl*{1-\Ek{bV}} + \sum_{k=2}^\infty \frac{\Gamma(k-β)}{\mu k!}\El{\kl*{1-\sum_{i=1}^b V_i}^{k}}.
    \end{equation}
\end{example}

\subsubsection{Balls remaining in the tree}
In fractional split trees, balls can disappear instead of going to one of the children, so it is an interesting question to consider how many balls remain in the tree. Define the toll function
\begin{equation}\label{def:phi-b}
    \phi_b(T_n) := s_0\eins\kl{n > s} + n\eins\kl{n≤s},
\end{equation}
so that $X_n(\phi_b)$ counts the number of balls in the nodes of $T_n$. Since this toll function is bounded by $s+1$ and only depends on $n$, we can apply Theorems \ref{thm:mean} and \ref{thm:limit}.

For \Cref{thm:constants}, we again assume $s_0=s_1=0$, which means that all balls are in leaves and $\phi_b = \sum_{k=1}^s k\fringephi[k]$. If $β<1$, we have by \eqref{constants:general}
 and \eqref{constants:fk},
\begin{align}
    \varpi_0[\phi_b] &= \sum_{k=1}^s 
\biggl(
\frac{\Gamma(k-β)}{\mu(k-1)!} - \frac{b}{\mu}\sum_{j=0}^{s-k}\Ek*{V^k(1-V)^j}\frac{\Gamma(k+j-β)}{(k-1)!j!}
\biggr)
    \nonumber \\&= \sum_{k=1}^s 
\biggl
(\frac{\Gamma(k-β)}{\mu(k-1)!} - \frac{b}{\mu}\sum_{\ell=0}^{k}\Ek*{\ell V^\ell(1-V)^{k-\ell}}\frac{\Gamma(k-β)}{\ell!(k-\ell)!}
\biggr)
    \nonumber \\&= \sum_{k=1}^s \frac{\Gamma(k-β)}{\mu(k-1)!} \kl*{1-\Ek*{bV}}.
\end{align}
The first equality is by noting the double sum runs over all $k, j$ with
$k+j≤s$ and changing indices, the second is by interpreting 
 the inner
sum as the expectation of a binomial $\Bin(k,V)$ random variable. In the lattice case, this is the 0-th Fourier coefficient of $\varpi[\phi_b]$, implying $\varpi[\phi_b]$ is oscillating around this value.

\subsubsection{\texorpdfstring{$k$}{k}-protected nodes} A node is called $k$-protected ($k≥1$) if the shortest distance to a descendant that is not a leaf is at least $k$. Hence, $1$-protected nodes are internal nodes and $2$-protected nodes are internal nodes that have no children that are leaves. 
The 2-protected nodes are the most studied case.

The number $X_n(\phiprot)$ of $k$-protected nodes has the toll function
\begin{equation}
    \phiprot(T) := \eins\kl*{\forall v\in T, \abs v<k: \deg(v)\ne 0},
    \label{def:phiprot}
\end{equation}
where $\deg$ is the outdegree, and $\abs v$ denotes the length of $v$ as a string, which equals the distance of $v$ to the root.
If $\beta=1$ or $s_1>0$, a node $v$ is a leaf if and only if $1≤n_v≤s$, so
\begin{equation}
    \phiprot(T_n) = \eins\kl*{\forall v\in [b]^\ast, \abs v<k: n_v=0 \text{ or } n_v>s},
    \label{def:phiprot-2}
\end{equation}
Else, i.e.\ if $\beta<1$ and $s_1=0$, a node could be a leaf just by all balls vanishing; this happens with probability $\Ek[\big]{(1-\sum_{i=1}^b V_i)^{n-s_0}}$ for $n>s$.

Since $\phiprot$ is a bounded toll function, \Cref{thm:mean} applies, and we have 
\begin{equation}
    \Ek*{X_n(\phiprot)} = \varpiprot(\ln n)n^β+o(n^β)
\end{equation}
for a $d$-periodic function $\varpi$ bounded away from zero.
\Cref{thm:limit} applies if $k=1,2$ and $\beta=1$, because then $\phiprot$ only depends on the numbers $n_1, \dots, n_b$ of balls that go to the children of the root, as can be seen from \eqref{def:phiprot-2}.
%\Cref{thm:constants} can also be applied if $s_0=s_1=0$ to calculate $\varpiprot$.

Assume $s_0=s_1=0$. As outlined in \Cref{sec:poisson}, $\varpiprot$ can be calculated from the Mellin transform of 
$\fprot(ł) := \mu\inv\E[\phiprot(\poisson T_ł)]$, 
where $\poisson T_ł$ is $T_n$ for $n$ being Poisson$(ł)$ distributed and independent. We have for $\beta=1$
\begin{align}
    \mu\fprot(ł) &= \Pk*{ \forall {v\in [b]^{k-1}}: n_v' = 0 \text{ or } n_v'>s} - \P(n_{\eps} = 0)
    \nonumber \\ &= \Ek*{ \prod_{v\in [b]^{k-1}} \kl*{1 - \sum_{j=1}^s \frac{(M^1_vł)^j}{j!}e^{-M^1_vł}}} - e^{-ł}.
    \label{ex:prot-f}
\end{align}
This product can be rewritten into a sum 
of $1-e^{-\lambda}$ (with Mellin transform $-\Gamma(z)$)
and a number of  terms of the form $\E[ł^a Xe^{-Ył}]$ for
some integers $a\ge1$ and random variables $X$ and $Y$ (depending on the
split vectors  $\V^{(v)}$ for $\abs v<k-1$); 
these terms give (using Fubini's theorem) Mellin transforms of the form
$\Gamma(z+a)\E[X Y^{-z-a}]$, yielding a result similar to 
\cite[(4.66) and (4.68)]{janson:trie}. 
Note that $\fprot$ is bounded, and also $\fprot(\lambda)=\Ok*{\lambda^2}$
since $\phiprot_1=0$; hence the Mellin transform 
$\mellin{\fprot}(z)$ exists at least for $-2<\Re(z)<0 $
(the apparent singularity at $z=-1$ is thus removable).
%They also show how to deal with the removable
%singularity at $z=-1$ for the case $β=1$. 

\begin{example}
    In the particular case $β=1$, $b=2$, $s=1$, $s_0=s_1=0$ and $k=2$,
    the formula \eqref{ex:prot-f} simplifies to just, using $V_1+V_2=1$,
    \begin{equation}
        \mu\fprot[2](ł) = 1 - e^{-ł} - 2\Ek*{Vłe^{-Vł}} + \Ek*{V_1V_2ł^2\exp\kl{-ł}},
    \end{equation}
    which has Mellin transform
    \begin{equation}
        \mu\mellin{\fprot[2]}(z) = -\Gamma(z) - 2\Ek*{V^{-z}}\Gamma(z+1) + \Ek*{V_1V_2}\Gamma(z+2).
    \end{equation}
    % If we are in the nonlattice case and $\beta<1$, then by \eqref{nonlattice-mellin}
    % \begin{equation}
    %     \varpiprot[2] = \frac1{\mu}\kl*{-\Gamma(-β) - \Gamma(1-β) + \Ek*{V_1V_2(V_1+V_2)^{β-2}}\Gamma(2-β)}.
    % \end{equation}
    If we are in the nonlattice case, we need to find the value of $\mellin{\fprot[2]}$ at the removable singularity $z=-1$:
    \begin{align}
        \mu\mellin{\fprot[2]}(-1) &= \lim_{z\to -1}-\frac{\Gamma(z+2)}z\cdot \frac{z\Ek*{2V^{-z}} + 1}{z+1} + \Ek*{V_1V_2}
        \nonumber \\ &= \left. \dd z z\Ek*{2V^{-z}}\right|_{\mathrlap{z=-1}} + \Ek*{V_1V_2}
        \nonumber \\ &= 
\Bigl(2\Ek*{V^{-z}} + z\Ek*{-2\ln(V)V^{-z}} \Bigr)
\biggr|_{\mathrlap{z=-1}} + \Ek*{V_1V_2}
        \nonumber \\ &= 1 - \mu  + \Ek*{V_1V_2},
        \label{ex:prot:singularity}
    \end{align}
as found in a special case in 
     \cite[Example~4.13]{janson:trie}.
\end{example}

Similar and stronger results have been shown in e.g.\ \cite{Devroye_Janson_2014, janson:trie,holmgren_cmj_2017,holmgren_multivariate_2017, Bóna_2014, Cheon_Shapiro_2008, Fuchs_Lee_Yu_2016, Du_Prodinger_2012, Gaither_Homma_Sellke_Ward_2012, GaitherWard2013}.

\subsubsection{Outdegree}\label{sec:degree}
The number $X_n(\phi\degk)$ of nodes with a certain outdegree $k≥1$ has the bounded toll function
\begin{equation}\label{def:degk}
    \phi\degk(T_n) := \eins(\deg(\eps) = k) = \eins\kl*{\abs*{\mg{i: n_i>0}}=k}
\end{equation}
only depending on $n_1,\dots, n_b$, so Theorems \ref{thm:mean} and \ref{thm:limit} both apply. Therefore,
\begin{equation}
    \frac{X_n(\phi\degk)}{n^β\varpi[\phi\degk]} - 1 \overset{\ell_2}\to Y,
\end{equation}
with $\varpi[\phi\degk]$ from \Cref{thm:mean} and $Y$ from \Cref{thm:limit}.

If $s_0=s_1=0$, we can use \Cref{thm:constants} to calculate the limit
function $\varpi[\phi\degk]$. 
First, we note that \eqref{def:degk} trivially yields $\phi\degk_n=0$ for $n<k$ or $n≤s$. 
% that thus \eqref{deg-2} holds also for $0≤n<k$ with $ \phi\degk_0:=0$.
By \Cref{prop:sum} and \Cref{thm:constants}, we have
\begin{equation}
  \label{deg-3}
\varpi_0[\phi\degk]
= \sum_{n=1}^\infty  \phi\degk_n\varpi_0[\fringephi[n]]
= \sum_{n=j}^\infty  \phi\degk_n\varpi_0[\fringephi[n]]
= \sum_{n=j}^\infty  \phi\degk_n\frac{\Gamma(n-β)}{μn!}
\end{equation}
for any $j=0,\dots, s+1$ with $\phi\degk_0:=0$.

In order to keep the terms simple, we restrict ourselves to the case $k>s$.
It then follows from \eqref{def:degk} and the inclusion-exclusion formula that, for every $n≥1$,
\begin{equation}\label{deg-2}
    \phi\degk_n %= \Ek*{\phi\degk(T_n)}
    = \sum_{l=b-k}^b (-1)^{l+k-b} \binom l{b-k} \sum_{\substack{I\subseteq [b] \\ \abs{I} = l}} \Ek*{\kl*{1-\sum_{i\in I}V_i}^n}.
\end{equation}
Write as a shorthand $V_I := \sum_{i\in I} V_i$. 
% By \Cref{thm:constants}, we have
% $\varpi_0[\fringephi[n]]=\frac{\Gamma(n-β)}{μn!}$ for $n > s$ and thus
% \eqref{deg-3} yields
% \begin{equation}
%   \label{deg-4}
% \varpi_0[\phi\degk] 
% = \sum_{n=j}^\infty  \phi\degk_n\frac{\Gamma(n-\beta)}{\mu n!},
% \end{equation}
% for any $j=0,\dots,k$.  so that \eqref{deg-2} holds for all $n$.
Then, if $\beta<1$, 
the sum in \eqref{deg-3} %\eqref{constants:general} 
can be written as a linear combination of sums having the form
\begin{equation}\label{ex:deg:1}
    \frac{1}{μ}\sum_{n=0}^\infty \frac{\Gamma(n-\beta)}{n!}
    \Ek*{\kl*{1-V_I}^n}
=
    \frac{1}{μ}\sum_{n=0}^\infty \Gamma(-\beta)\binom{\beta}n(-1)^n \Ek*{\kl*{1-V_I}^n}
    = \frac{\Gamma(-\beta)}{\mu} \Ek*{V_I^{\beta}}
\end{equation}
by Fubini's theorem and summing the binomial series 
(a Taylor series)
for $(1-t)^{\beta}$, evaluated at $t=1-V_I$.
If $\beta=1$, we take instead $j=2$ in \eqref{deg-3}.
Since the resulting sum is a continuous function of $\beta>0$, we can 
omit calculations and obtain the result directly by taking the limit as
$\beta\to1$ of the resulting formula for $\varpi_0[\phi\degk]$.
(The individual terms  \eqref{ex:deg:1} diverge, but not their linear
combination.) 
We thus arrive at the following result:

\begin{example}\label{ex:degree}
    Assume that $s_0=s_1=0$ and $k>s$, and consider the toll function $\phi\degk$ counting
    nodes with degree $k$ defined in \eqref{def:degk}. Then, 
for $\beta<1$,
    \begin{equation}\label{ex:deg:varpi}
        \varpi_0[\phi\degk] = \frac{\Gamma(-β)}{μ}\sum_{l=b-k}^b (-1)^{l+k-b} \binom l{b-k} \sum_{\mathclap{\substack{I\subseteq [b] \\ \abs{I} = l}}} \Ek*{V_I^β}
.    \end{equation}
If $β=1$, we instead have, by letting $\beta\to1$ in \eqref{ex:deg:varpi}
and using the Taylor expansion
$\Ek{V_I^β}=\Ek{V_I}+(\beta-1)\Ek{V_I\ln(V_I)}+\Ok*{(1-\beta)^2}$,
    \begin{equation}
        \varpi_0[\phi\degk] = \frac{1}{μ}\sum_{l=b-k}^b (-1)^{l+k-b} \binom l{b-k} \sum_{\substack{I\subseteq [b] \\ \abs{I} = l}} \Ek*{V_I\ln(V_I)}.
    \end{equation}
    Replacing $β$ with $β+2\pi im/d$ in \eqref{ex:deg:varpi} gives the
    $m$-th Fourier coefficient of $\varpi$ in the lattice case (also in
    the case $\beta=1$).
\end{example}

\subsubsection{Shape functional}

The shape functional is an example of a toll function that is unbounded. It is defined as the logarithm of the product of all subtree sizes,
thus by
\begin{equation}\label{def:shape}
    X_n(\phi\shape) = \ln \prod_{v\in T_n} n_v = \sum_{v\in T_n} \ln \kl{n_v}.
\end{equation}
The toll function is $\phi\shape(T_n) = \ln n$, and we can therefore apply Theorems \ref{thm:mean} and \ref{thm:limit}. The shape functional is closely linked with binary search trees: For a binary split tree with $s=s_0=1$ and $s_1=0$, the shape functional is given as $-\ln\Pk*{T^{\mathrm{bin}}_n = T_n\given T_n}$, where $T^{\mathrm{bin}}_n$ is a binary split tree independent of $T_n$. So if $T_n$ is itself a binary split tree, the shape functional is its information content, and the expectation of the shape functional is its (Shannon) entropy. 

\begin{example}
Theorems \ref{thm:mean} and \ref{thm:limit} apply to $X_n(\phi\shape)$, and it follows that
    \begin{equation}
        X_n(\phi\shape) = \varpi(\ln n)n^β + o(n^β),
    \end{equation}
    with a continuous, $d$-periodic function $\varpi$ bounded away from zero. If $s_0=s_1=0$,
    \Cref{thm:constants} and \Cref{prop:sum} yield
   \begin{equation}\label{ex:shape:varpi}
        \varpi_0[\phi\shape] = \sum_{k=2}^\infty 
\biggl(
\frac{\Gamma(k-β)\ln k}{\mu k!} - \frac{b}{\mu}\sum_{j=0}^{s-k}\Ek*{V^k(1-V)^j}\frac{\Gamma(k+j-β)\ln k}{k!j!}
\biggr).
    \end{equation}
    In the nonlattice case, $\varpi$ is constant; in the lattice case, its $m$-th Fourier coefficient is given by replacing $β$ in \eqref{ex:shape:varpi} with $β+2\pi im/d$. Note that the second sum in \eqref{ex:shape:varpi} vanishes if $s=1$. In the particular case of fringe trees with $s=1$ and $\beta=1$, we have
    \begin{equation}
        \varpi_0[\phi\shape] = \frac1\mu\sum_{k=2}^\infty \frac{\ln k}{k(k-1)} \approx 1.25775\mu\inv.
    \end{equation}
\end{example}
See similar results for tries in \cite{janson:trie}, simply generated trees in \cite{meir_log-product_1998}, \cite{fill_limiting_2004} and various models in \cite{wagner:additive}.

Note that when generalizing the shape functional to fractional split trees, we could have also used the number of nodes $X_n(1)$ or the number of remaining balls $X_n(\phi_b)$ from \eqref{def:phi-b} instead of $n$ in the definition of $\phi\shape$ in \eqref{def:shape}. These also allow us to apply \Cref{thm:mean}, but we can neither use \Cref{thm:limit} nor can we calculate $\varpi$ accurately using \Cref{thm:constants}. However, an approximation based on \Cref{prop:sum} is possible.

\subsubsection{Number of subtrees}

Let $\unrootedsubtrees(T_n)$ be the number of subtrees of $T_n$ and
$\rootedsubtrees(T_n)$ the number of subtrees of $T_n$ containing the
root. 
The latter are in bijection to antichains in $T_n$, as noted in
\cite{wagner:additive,wagner:additive-aofa} where these are studied for
some models of random trees. Note that a subtree $t$
containing the root consists of the root and for every $i\in [b]$, its
fringe subtree $t^i$ is either a subtree of $T_n^i$ containing the root or
empty. 
As noted by \cite{wagner:additive, wagner:additive-aofa}, 
it follows from this fact that $\ln(1+\rootedsubtrees(T_n))$ has the bounded toll function
\begin{equation}
    \phi(T_n) = \ln\kl*{1+\frac1{\rootedsubtrees(T_n)}}.
\end{equation}
Since $\unrootedsubtrees(T_n)$ is at most $X_n(1)$ (the number of nodes in $T_n$) times larger than $\rootedsubtrees(T_n)$, their logarithms have the same order.
\begin{example} As $n\to\infty$,
    \begin{equation}\label{ex:nsub:1}
        \El{\ln \unrootedsubtrees(T_n)} \sim \El{\ln \rootedsubtrees(T_n)} \sim n^β \varpi(\ln n),
    \end{equation}
    where $\varpi$ is a continuous $d$-periodic function bounded away from zero in the lattice case, and a positive constant in the nonlattice case. For split trees, i.e.\ $\beta=1$, 
and for simplicity assuming $s_0=s_1=0$ and $s=1$,
the average $\varpi_0$ (cf.\ \eqref{def:varpi0}) of $\varpi$ is bounded by
    \begin{equation}\label{ex:nsub:2}
\ln 2 ≤  
\varpi_0 ≤ \ln 2 + \sum_{k=2}^\infty \frac{\ln\kl{1+2^{-k}}}{k(k-1)μ} \approx 0.6931 + 0.1386\mu\inv;
    \end{equation}
\end{example}
\begin{proof}
    Equation \eqref{ex:nsub:1} is a direct consequence of \Cref{thm:mean}. 
For the upper bound in \eqref{ex:nsub:2},
the number of leaves in $T_n$ is exactly $n$, and 
every
subset of leaves maps injectively to a subtree containing the root and
exactly the leaves in the subset, so $\rootedsubtrees(T_n) ≥ 2^{n}$
and thus 
$\phi(T_n) \le \ln(1+2^{-n})$
\cite{wagner:additive}.
%The leaves each contribute $\ln(2)$, and the fringe subtrees of size $k$
%contribute at most $\ln(1+2^{-k})$. 
The statement \eqref{ex:nsub:2} then follows from \eqref{constants:general} 
and \Cref{thm:constants}.
\end{proof}

It follows from \cite[Theorem 4.2]{wagner:additive}  
or \cite[Theorem 1.14]{holmgren_limit_2015}
that 
for binary search trees, 
$(\ln\unrootedsubtrees(T_n)-\E[\ln\unrootedsubtrees(T_n)])/\sqrt n$
converges in distribution to $N(0,\sigma^2)$ for some $\sigma^2$ (given by a
rather complicated expression); this means asymptotic normality provided
$\sigma^2>0$, which seems to be likely but not yet shown.

\subsubsection{Independence number, domination number and similar functionals}\label{sec:indnum}
For a graph $G = (V,E)$, an \emph{independent set} is a set $U\subseteq V$, so that no nodes in $U$
are adjacent to each other.
The \emph{independence number} $\alpha(G)$ of a graph $G=(V,E)$ is then the maximal size of an independent set of $G$, i.e.
\begin{equation}\label{def:indnum}
    \alpha(G) := \max\mg{\abs{U}\mid U\subseteq V, e\not\subset U \,\forall e\in E}.
\end{equation}
The computation of $\alpha(G)$ for general graphs $G$ is NP-hard \cite{Karp_1972,Xiao_Nagamochi_2017}, but for a tree $T$ there is a simple algorithm based on the inequality
\begin{equation}
    \alpha(T) - 1 ≤ \max\mg{\abs{U}\mid \eps\notin U\subseteq T, e\not\subset U \,\forall e\in E} ≤ \alpha(T).
\end{equation}
Exactly one of the inequalities is sharp. If the left one is sharp, then the root $\eps$ is contained in all maximal independent sets of $T$ and $T$ is called \emph{essential}. A node $v$ is called essential if $T^v$ is essential. With a little bit of work (cf.\ \cite[Lemma 2.1]{Fuchs_Holmgren_Mitsche_Neininger_2021}), one can show the following:

\begin{lemma}
    A node is essential if and only if all of its children are not essential.
    The essential nodes of $T$ form an independent set $U$, with $\abs U = \alpha(T)$.  The toll function of $\alpha$ is therefore $\phi^\alpha(T) := \eins(\eps \text{ is essential})$.
\end{lemma}
The independence number thus has a bounded toll function, and we can apply \Cref{thm:mean}, \Cref{thm:constants} and \Cref{prop:sum}.

\begin{example}[Independence number]
    There exists a continuous, $d$-periodic function $\varpi[\phi^\alpha]$ bounded away from zero, so that
    as $n\to\infty$,
    \begin{equation}
        \Ek*{\alpha(T_n)} = \varpi[\phi^\alpha](\ln n)n^β + o(n^β).
    \end{equation}
    If $s_0=s_1=0$ and $s=1$, the mean $\varpi_0[\phi^\alpha]$ of $\varpi[\phi^\alpha]$ is given by
    \begin{equation}
        \varpi_0[\phi^\alpha] = \varpi_0[\delta_1] + \frac1{μ}\sum_{k=2}^\infty \frac{\Gamma(k-β)}{k!}\Ek*{\phi^\alpha(T_k)}.
    \end{equation}
    For split trees ($β=1$), $\varpi_0[\delta_1] = 1$, and for $β<1$ it is given in \eqref{varpi-delta1}.
\end{example}

In \cite{Fuchs_Holmgren_Mitsche_Neininger_2021}, $\varpi[\phi^\alpha]$ is calculated explicitly for binary search trees and random recursive trees,
and it is shown that $\alpha(T_n)$ is asymptotically normal for these models. 
Earlier, \cite{janson:indnum} showed that for Crump--Mode--Jagers processes $\alpha(T_n) \sim \varpi[\phi^\alpha]$ a.s. Crump--Mode--Jagers processes encompass binary search trees, $m$-ary search trees, random recursive trees and preferential attachment trees. 
There are a few more graph parameters which are for trees affine functions of the independence number, e.g.\ vertex cover number, matching number or clique cover number, see \cite[Remark 1.7]{Fuchs_Holmgren_Mitsche_Neininger_2021}.

For a graph $G=(V,E)$ and a subset $D\subseteq V$ of nodes, we write $N(D)$ for the set of neighbors of $D$ (including $D$).
The set $D$ is called a \emph{dominating set} if $N(D) = V$. The minimal size of a dominating set is called the \emph{domination number} $D(G)$, which is also NP-hard on general graphs, see \cite[Theorem A.1]{kann} and \cite{Karp_1972}.
The dominating number of trees can be computed similarly to the independence number by assigning every subtree a `type', based on which of the inequalities in
\begin{align}
    D(T) &≤ \min\mg*{\abs{D} \mid \eps \in D \subseteq T,  N(D) = T}
    \notag \\&≤ \min\mg*{\abs D\mid D \subseteq T, N(D)\cup\mg{\eps} = T} + 1 ≤ D(T) + 1
\end{align}
is sharp.
If the first inequality is sharp (type F), a minimal dominating set cannot contain the root.
If the third inequality is sharp (type B), it is easier to dominate $T\ohne\mg\eps$ than to dominate $T$.
(It is therefore better to dominate the root from its parent.)
If the second inequality is sharp (type R), neither is the case, and we can choose a dominating set containing the root without penalties, see \cite[Proof of Theorem 1.4]{Fuchs_Holmgren_Mitsche_Neininger_2021}.

Cockayne, Goodman and Hedetniemi \cite{Cockayne_Goodman_Hedetniemi_1975} showed that the type of a node is determined by the types of its children: Parents of B have type R. If a node has at least one child of type R, but none of type B, it has type F. Else, all children are F and the node has type B.

Furthermore, \cite{Cockayne_Goodman_Hedetniemi_1975} show that nodes of type R form a minimal set dominating every node except possibly the root.
So for the bounded toll function $\phi^D(T) = \eins(T\text{ is of type R})$, we have $D(T_n) - X_n(\phi^D) \in\mg{0,1}$ and thus we can apply our results.

\begin{example}[Domination number]
    There exists a continuous, $d$-periodic function $\varpi[\phi^\alpha]$ bounded away from zero, so that
    as $n\to\infty$,
    \begin{equation}
        \Ek*{D(T_n)} = \varpi[\phi^D](\ln n)n^β + o(n^β).
    \end{equation}
     If $s_0=s_1=0$ and $s=1$, the mean $\varpi_0[\phi^D]$ of $\varpi[\phi^D]$ is given by
    \begin{equation}
        \varpi_0[\phi^D] = \frac1{μ}\sum_{k=2}^\infty \frac{\Gamma(k-β)}{k!}\Ek*{\phi^D(T_k)}.
    \end{equation}
\end{example}

For binary search trees and random recursive trees, \cite{Fuchs_Holmgren_Mitsche_Neininger_2021} show that the domination number is asymptotically normal. 
% The constant $\varpi[\phi^D]$ for random recursive trees was approximated for random recursive trees to be $\approx 0.3745$ \cite{Cooper_Zito_2009}.
% independent dominating set.

\subsection{Generalized Quicksort}\label{sec:quicksort}
As noted before, the fringe trees where \Cref{rem:beta} applies are just a slight generalization of $m$-ary search trees. The constants $\varpi[\fringephi]$ for $m$-ary search trees are well-known, in fact, it is shown in \cite{holmgren_multivariate_2017} that the number $X_n(\fringephi)$ of fringe trees of size $k$ in $m$-ary search trees is asymptotically normal for $m≤26$. 

The $m$-ary search tree is a model for a sorting algorithm, where $m-1$ elements are chosen at random and the rest of the elements is divided into $m$ sets depending on how they compare to these $m-1$ elements. The algorithm then recurses on each of the $m$ sets. Here, $m$ corresponds to $b$. If the elements to sort are chosen independently uniformly on $[0,1]$, this is a split tree with $s_0=m-1, s_1=0$ and the splitter is Dirichlet-$(1,\dots, 1)$-distributed, so that $V\sim B(1,m-1)$. For simplicity we assume $s=m-1$.

\begin{example}[$m$-ary search trees]\label{ex:mary}
    If in \Cref{rem:beta} $s_0=b-1$ and $s_1=0$ and thus $a=β=1$, we get the special case of $m$-ary search trees. Here, $\mu = \psi(b+1) - \psi(2) = H_b - 1$, where $H_b := \sum_{k=1}^b \frac1k$ is the $b$-th harmonic number, and \eqref{rem:beta:1} evaluates to
   \begin{equation}
        \label{rem:beta:2}
        \varpi[\fringephi] = 
\frac1{H_b-1}\cdot\frac{k!}{(k+2)!} = \frac1{(H_b-1)(k+1)(k+2)}.
    \end{equation}
\end{example}

In the generalized Quicksort of Hennequin \cite{Hennequin_1989,hennequin_1991}, instead of selecting just $m-1$ elements, $ma-1$ (where $a≥1$) elements are selected and sorted. Out of these elements, the $a$-th, $2a$-th, \dots, $(m-1)a$-th largest element are selected to compare all other elements with. This additional step makes sure that the sizes of the sets are more balanced and decreases $\mu$, since the splitter is instead $B(a,(m-1)a)$ Beta-distributed. All the selected elements have to be distributed using $s_1$ and $s_0$, hence $s_0+ms_1=ma-1=s$. % If the elements that were first selected are then distributed evenly into the subsets or retained in the root only affects $s_0$ and $s_1$.
If $b=2$, then this is known as median-of-$(2a-1)$ Quicksort.

\begin{example}[Generalized Quicksort of Hennequin \cite{Hennequin_1989,hennequin_1991}]\label{ex:hennequin}
    If in \Cref{rem:beta} $β=1$ and $a = \frac{s_0}{b-1}+s_1$ is an integer, then the tree corresponds to the generalized Quicksort of Hennequin.
    Here $\mu = H_{ba} - H_a$ and \eqref{rem:beta:1} evaluates to
    \begin{equation}
        \varpi[\fringephi] = \frac1{(H_{ba}-H_a)(a+k-s_1)(a+k-s_1+1)}
    \end{equation}
    for $k>s$ or $s\in\mg{ba-1,ba}$ and $k>s_1$.
\end{example}

\subsection{Examples of fractional fringe trees}

We close this example section with a few motivating examples for the study of fractional fringe trees.

\subsubsection{Number of maxima}\label{sec:maxima}
A point $(x_1,\dots, x_d)$ in $\RR^d$, $d≥1$, is said to \emph{dominate} another point $(y_1,\dots, y_d)$ if $x_i≥y_i$ for every $i=1,\dots, d$.
For a set of points $p_1,\dots, p_n$ on $\RR^d$, a point is called a \emph{(coordinatewise) maximum} if no other point in the set dominates it.
The study of these maxima has many applications, e.g.\ in graph drawing or heuristic search in artificial intelligence. See \cite{chen_efficient_2003} for a comprehensive overview over applications and algorithms.

A special case is the number of coordinatewise maxima if $p_1,\dots, p_n$ are independently uniformly distributed on the triangle which is the convex hull of $(0,0), (1,0), (0,1)$ in $\RR^2$.
This can be described with a fractional fringe tree:
Consider the point $p_0 = (x_0, y_0)$ where $m = x_0+y_0$ is maximal. This point must be a coordinate maximum, and all other points will have $x+y<m$. Using $p_0$, we can divide the triangle into four smaller areas as in \Cref{fig:coord-max}:
A strip $x+y>m$, in which no points will be found, two triangles, one with $x>x_0, x+y<m$ and one with $y>y_0, x+y<m$, in which further maxima will be, and a rectangle $x≤x_0, y≤y_0$, in which the points are dominated by $p_0$.

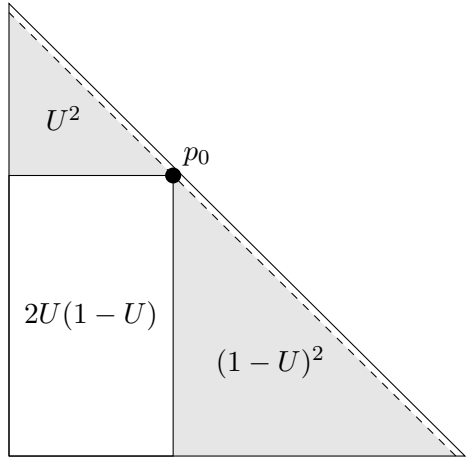
\begin{figure}
    \centering
    \begin{tikzpicture}[scale=6]
        \coordinate (p0) at (0.36,0.62);
        \path[fill=gray!20] (p0|-0,0) -- (p0) -- (0.98,0);
        \path[fill=gray!20] (0,0|-p0) -- (p0) -- (0,0.98);
        \draw (0,0)--(1,0)--(0,1)--(0,0);
        \draw[dashed] (0.98,0)--(0,0.98);
        \draw[draw=black] (0,0) rectangle (p0);
        \fill (p0) circle (0.5pt) node[above right] {$p_0$};
        \node  at ($0.5*(p0)$) {$2U(1-U)$};
        \node  at ($1/3*(p0) + 1/3*(0,0|-p0) + 1/3*(0, 1)$) {$U^2$};
        \node  at ($1/3*(p0) + 1/3*(p0|-0,0) + 1/3*(1, 0)$) {$(1-U)^2$};
    \end{tikzpicture}
    \caption{The three areas in which the point with maximal sum splits up the triangle.}\label{fig:coord-max}
\end{figure}
The important observation now is that the two triangles are scaled-down versions of the original triangle, and a point is a maximum if and only if it is either $p_0$ or a maximum in either of the triangles. Given $m$, the value $U:=\frac{x_0}m$ is uniformly distributed on $[0,1]$ because of symmetry and the other points are uniformly distributed on the convex hull of $(0,0), (0,m), (m,0)$. The sizes of the two smaller triangles are $U^2$ and $(1-U)^2$ of this triangle, so the maxima form a fractional split tree with $s=s_0=1$, $s_1=0$ and splitter $(U^2, (1-U)^2)$. Since $U^{2\cdot\frac12} + (1-U)^{2\cdot\frac12} = 1$ is non-random, $\beta=\frac12$ (see \eqref{def:beta}) and the variance is of order smaller than $\sqrt n$ by \Cref{cor:lln}.

This exponent of $\frac12$ is well known, in fact for all two-dimensional polygons, the order of the number of maxima can be either constant, $\ln n$ or $\sqrt n$ depending on whether there are edges in the upper-right corner, see \cite{golin_maxima_1993}. The distribution is fully described in \cite{bai_limit_2001}, where the $\sqrt n$ term can be reduced to this problem on the triangle.

The fringe trees in the tree have an interesting interpretation:
Order all the maxima by their $x$-value, forming a sequence $m_1,\dots, m_k$. This is also the in-order traversal of the tree.
Parent nodes have higher sums than child nodes, so a maximum $m_i$ is an ancestor of $m_j$ for $i<j$ if and only if $m_i ≥ \max\mg{m_i,m_{i+1}\dots, m_j}$ 
(analogously for $j<i$) and the fringe tree of $m_i$ is the unique longest subsequence in which $m_i$ is the maximum. The size of the fringe tree can thus be considered the `influence' of the maximum.
\begin{theorem}
    Let $X_n(\fringephi)$ be the number of fringe trees of size $k$ in 
the tree of the coordinatewise maxima. Then, for $k≥1$,
    \begin{align}
        X_n(\fringephi) &= \frac{\Gamma\kl*{k+\frac12}}{(k+1)!}\sqrt n + o\kl*{\sqrt n} \quad\text{and} \label{ex:maxima:fringe} \\
        X_n(1) &= \sqrt{\pi n} + o\kl*{\sqrt n} \label{ex:maxima:nodes}
    \end{align}
    in probability, where $X_n(1)$ is the number of nodes, i.e., the number of coordinatewise maxima.
\end{theorem}
\begin{proof}
    Here, $V\sim B\kl*{\frac12,1}$, so we can apply \Cref{rem:beta} to calculate $\varpi[\fringephi]$. Note that $\mu = \psi(2) - \psi(1) = 1$
    and \eqref{ex:maxima:fringe} follows from \eqref{rem:beta:1}. For \eqref{ex:maxima:nodes}, $\varpi[1]$ can be immediately obtained from \Cref{thm:s0} plus \eqref{rem:beta:moment} in the proof of \Cref{rem:beta} or by summing up \eqref{ex:maxima:fringe}
as in \eqref{eq:sum:hyper}.
\end{proof}
Note that \eqref{ex:maxima:nodes} is known much more precisely: In \cite{bai_limit_2001} it is shown that
\begin{equation}
    \Ek*{X_n(1)} = \frac{\sqrt{\pi}n!}{\Gamma\kl*{n+\frac12}} - 1,
\end{equation}
the variance is estimated to $\Ok*{n^{-K}}$ precision for any $K>0$ and asymptotic normality of $X_n(1)$ is shown.

\subsubsection{Partial Search in Quadtrees}
A (point) quadtree (short for quadrant tree) is a data structure to store a set of points $\mg{x_1,\dots, x_n} \subset \RR^d$ in a way that allows for efficient searches in space.
The structure is build recursively as follows: The first element $x_1$ is chosen as the \emph{pivot element} and stored in the root of the tree. All other elements are compared coordinate-wise with this pivot and sorted into $2^d$ sets, labelled by $\mg{0,1}^d$. An element $y\in \RR^d$ is sorted into the set
with label $\kl[\big]{\eins(y_1>(x_1)_1),\dots, \eins(y_d>(x_1)_d)}$. The procedure is repeated on the $2^d$ sets, which form the subtrees of the root, until no elements are left.

If we choose $x_1,\dots, x_n$ i.i.d.\ uniformly on $[0,1]^d$, this forms a split tree with $b=2^d$, $s=s_0=1$, $s_1=0$ and splitters
$V_I = \prod_{i=1}^d \kl*{U_i I_i + (1-U_i)(1-I_i)}$ for $I\in\mg{0,1}^d$, with $U_1,\dots, U_d$ i.i.d.\ distributed on $[0,1]^d$. If one now wants to find an element $y$, one can compare the pivot with the element and then continue to the set with label $\kl[\big]{\eins(y_1>(x_1)_1),\dots, \eins(y_d>(x_1)_d)}$ until the element is found.

Something interesting happens if we only know a few of the coordinates of the element $y$; without loss of generality assume that we know the first $s<d$ coordinates $y_1,\dots, y_s$. This is called a partial search, and we have to continue to the $2^{d-s}$ sets with labels $I$ for which $I_i = \eins(y_i>(x_1)_i)$ for $i≤s$. Because the splitting depends on $y$, this is not a fractional split tree, and even if we choose $y$ to be uniformly random, it still would not be one because of the branches then depend on each other. However, since the independence of the branches does not affect the mean, we still get the result that
\begin{example}
    Let $T_n$ be the nodes visited on a partial search on a $d$-dimensional quadtree, with $s<d$ coordinates given and chosen uniformly at random. Let $\phi_n$ be a toll function with $\phi_n = \Ok*{n^{β-δ}}$ for some $δ>0$,
    where $0 < \beta < 1$ is the solution to the equation
    \begin{equation}\label{ex:partial-search:beta}
        2^d = (β+1)^{d-s}(β+2)^s.
    \end{equation}
    Then there exist a constant $\varpi_{d,s}[\phi] > 0$
    \begin{equation}
        \Ek*{X_n(\phi)} = \varpi_{d,s}[\phi]n^β + o(n^β).
    \end{equation}
\end{example}
\begin{proof}
    This is a simple application of \Cref{thm:mean}, the only thing left to explain is the formula \eqref{ex:partial-search:beta}. Because we chose the coordinate uniformly, the splitter of a branch is of the form $\prod_{i=1}^{s} \widehat U_i \prod_{i=s+1}^d U_i$, where $\widehat U_i$ is the size-biased version of $U_i$. The reason for the size biasing is that the probability that a uniformly chosen coordinate falls into an interval scales linearly with the size of the interval. When the interval has size $U_i \sim B(1,1)$, that means that the interval a random coordinate is in has size $\widehat U_i \sim B(2,1)$. The equation for $\beta$ in \eqref{def:beta} then becomes
    \begin{equation}
        2^{s-d} = \Ek*{V^β} = \Ek*{U^β}^{d-s}\Ek*{\widehat U^β}^{s} =  \kl*{β+1}^{s-d}\fracc2{β+2}^s.
    \end{equation}
    Dividing both sides by $2^s$ and inverting gives \eqref{ex:partial-search:beta}.
\end{proof}
In the case where $\phi_n = 1$ is the toll function counting nodes, this has
been shown by \cite{Chern_Hwang_2003}, with the constant $\varpi_{d,s}[1]$
explicitly given and the error term improved to $\Ok*{1+n^{\Re β_2}}$, where
$β_2$ is the complex solution to \eqref{ex:partial-search:beta} with the
second-largest real part ($β$ itself having the largest real part).

For partial search, $\sum_{I\in \mg{0,1}^{d-s}} V_I^β$ is random, so we expect that $X_n(\phi)$ has variance of order $n^{2β}$. However, we cannot apply \Cref{thm:limit} because the branches are dependent if the coordinates $y_1,\dots, y_s$ are chosen at random.
An alternative approach to taking random coordinates $y_1,\dots, y_s$ is to consider instead the function mapping $(y_1,\dots, y_s)$ to the number of nodes. This was done by Broutin, Neininger and Sulzbach \cite{Broutin_Neininger_Sulzbach_2013, broutin_partial_2012}, who show that the variance is indeed of order $n^{2β}$ and this process converges when scaled down by $n^{β}$.

\subsubsection{Random Disk Triangulation}
\textcite{Curien_LeGall_2011} studied a random recursive triangulation of the closed unit disk $\disk := \mg{z\in \CC\mid \abs z≤1}$ as follows:
\begin{itemize}
    \item Start with a empty set $L_0 := \emptyset$
    \item In the $n$-th step, draw two random elements uniformly $x,y$ from the unit circle $S^1 := \mg{z\in \CC\mid \abs z=1}$ independently.
    \item If the chord $\overline{xy} := \mg{λx+(1-λ)y \mid λ\in[0,1]}$ does not intersect with any chord already in $L_{n-1}$, then $L_n :=L_{n-1} \cup \overline{xy}$, else $L_n := L_{n-1}$.
\end{itemize}
If we continue this procedure infinitely to some closed set $L_\infty := \overline{\cup_{n\in\NN} L_n}$, we obtain an infinite triangulation of $\disk$, which is where the name of this construction comes from.

The chords in $L_n$ can also be described using a split tree. Let $\overline{x_1y_1}, \dots, \overline{x_n y_n}$ be the chords that are randomly drawn in each step.
 When the first chord $\overline{x_1y_1}$ is added, it splits $S^1$ into two segments,
with relative length $U$ and $1-U$, where $U\sim\unif[0,1]$, and every subsequent chord must have both end points in the same segment, or it would intersect with $\overline{x_1y_1}$. Since the points are drawn uniformly, each chord has a $U^2$ chance to have both points in the first segment and a $(1-U)^2$ chance to have both in the second segment. All other chords are discarded. The chords on the two segments now behave independently and in the same way as on $S^1$, so we have a fractional split tree with splitters $(U^2, (1-U)^2)$, $s_0=s=1$ and $s_1=0$. This is the same fractional split tree as in the number of maxima from \Cref{sec:maxima}.

\section{Asymptotic mean}\label{sec:proof}

In this section, we prove \Cref{thm:mean}. The outline of the proof is as follows: We start by restricting ourselves to toll functions $\phi$ where $\phi_n$ is constant for $n>s$. In \Cref{sec:upper-bound} we derive an upper bounds on $\Ek*{X_n}$. Then, we prove \Cref{thm:mean} for those toll functions in \Cref{sec:constant-phi}. In the following section \Cref{sec:other-toll}, we extend this to general toll functions. In the last Section \ref{sec:lower-bound}, we show that the constants $\varpi$ in \Cref{thm:mean} are bounded away from zero.
As we will be only be concerned with the mean in this section, we do not need to care for the difference between $\phi(T_n)$ and its expectation $\phi_n$; and might implicitly assume that $\phi(T_n) = \phi_n$ only depends on the number of balls.

\subsection{Upper bound on the mean}\label{sec:upper-bound}
We start by proving that $\Ek*{X_n} = \Ok*{n}$. We later will sharpen this to $\Ek*{X_n} = \Ok*{n^β}$.

\begin{lemma}\label{lem:change}
    Assume that $\phi_n$ is constant for $n>s$.
    For $n,m \in \NN$,
    \begin{equation}
        \Ek*{X_{m+n}} = \Ek*{X_m} + \Ok*{n}.
    \end{equation}
\end{lemma}
\begin{proof}
    It suffices to prove the statement for $n=1$ and iterate.
    Adding a ball to an internal node $v$ either increases the number of balls in one of the subtrees of $v$ or the new ball disappears. Since $\phi_n$ is constant for $n>s$, the mean of $\phi(v)$ is unchanged. When a ball reaches a leaf $l$ with $m$ balls, the cost of the leaf changes from $\phi_m$ to $\phi_{m+1}$ if $m<s$, else the leaf reaches $s + 1$ balls and new nodes are created. If $s_0 > 0$ or $s_1>0$, at most $b$ new nodes are created. In the case of $s_0=s_1=0$, additionally a geometrically distributed number of empty internal nodes is created for every step that all $s + 1$ balls have in common. The number of new nodes created is bounded in expectation, and since $\phi$ is also bounded, the expected increase of $X_n$ is bounded by a constant.
\end{proof}

Let $B ≥ 1$ and define $P$ as the set of nodes $v \in [b]^*$ with $M^n_v≥B$. Then define $R := P[b] \setminus P$, i.e.\ the children of $P$ that are not in $P$ themselves. Therefore, $\abs R ≤ b\abs P$ and $M^n_v < B$ for all $v\in R$. Nodes in $R$ have mass lower than $B$, but their parents have mass at least $B$, and every node not in $P$ is the descendant of exactly one node in $R$. We will later vary $B$ and note that all estimates will hold uniformly in $B≥1$ except if explicitly stated.
The natural $\sigma$-algebra for $P$ and $R$ is 
\begin{equation}\label{def:gc}
    \Gc := \sigma(M^1_v \mid v\in [b]^*),
\end{equation}
the $\sigma$-algebra of all masses. When conditioning on $\Gc$, the term $n'_v$, defined in \eqref{def:ns} is binomially $\Bin(n, M^1_v)$ distributed as noted before, and we can consider the $\Gc$-measurable set $R$ as constant.
These two facts will be continuously used throughout the proofs of the following lemmata.

\begin{lemma}\label{lem:r}
     In the nonlattice case, as $n\to\infty$, 
    \begin{equation}\label{lem:r:statement}
        \Ek*{\abs P} \sim \frac1{\mu\beta}\fracc nB^\beta.
    \end{equation}
    In the lattice case with $d = \gcd\kl{\ln V}$,
    \begin{equation}\label{lem:r:recursive}
        \Ek*{\abs P} \sim \frac d{\mu\kl{1-e^{-\beta d}}}e^{\beta\floord{\ln\frac nB}}.
    \end{equation}
    Because $\abs R≤ b\abs P$, we have $\Ek*{\abs R} = \Ok*{\fracc nB^β}$ in both cases.
\end{lemma}

\begin{proof}
Define $U(t)$ as the expected number of nodes $v$ with $M_v^n ≥ ne^{-t}$ or equivalently $M_v^1 ≥ e^{-t}$. Thus 
    $\Ek{\abs P} = U\kl*{\ln\frac nB}$. Then, by conditioning on a random splitter $V$ of the root,
    \begin{equation}\label{lem:r:u}
        U(t) = \eins(t≥0) + b\Ek*{U(t+\ln V)} = \eins(t≥0) + (b\P_{-\ln V} * U)(t).
    \end{equation}
    This has the form of a renewal equation;  cf.\ \Cref{sec:rt}.
Because $b\P_{-\ln V}$ has total mass $b>1$, this is a so-called
\emph{excessive} renewal equation, which can be shifted to a \emph{proper}
renewal equation using a probability measure by size-biasing $V$ with
$e^{\beta\ln V} = V^\beta$, which gives us the probability measure $\P_{-\ln
  \widehat V}$, see \eqref{def:v-bias}
and \Cref{sec:rt}.
The function $e^{-\beta t}\eins(t\ge0)$ is directly Riemann integrable,
and thus,     
the key renewal theorem \Cref{thm:krt} applies and shows that $\widehat U(t)$ converges for $t\to \infty$ in the nonlattice case, such that
    \begin{equation}\label{ut-nonlattice}
        U(t) \sim e^{\beta t} \frac1\mu \int_0^\infty e^{-\beta u}\mathrm du = \frac{e^{\beta t}}{\mu\beta}.
    \end{equation}
    In the lattice case, we instead have
    \begin{equation}\label{ut-lattice}
        U(t) \sim e^{\beta t} \frac1\mu \sum_{u=0}^\infty d e^{-\beta (du+\restd t)} = \frac{de^{\beta\floord t}}{\mu(1-e^{-\beta d})}.
    \end{equation}Note that \eqref{ut-nonlattice} is the limit of \eqref{ut-lattice} for $d\to0$. Plugging in $t=\ln\frac nB$ proves \eqref{lem:r:statement} and \eqref{lem:r:recursive}.
\end{proof}

The term $\sum_{r\in R}\abs{n''_r}$ (defined before \eqref{def:nss}) accounts for the effects of balls retained via $s_0$ or redistributed via $s_1$ in $P$.
By recursively expanding \eqref{def:nss}, we obtain
\begin{equation}\label{eq:nss}
 \E \sum_{r\in R}\abs{n''_r} ≤ 2(s+1)\E \abs{P} = \Ok*{n^\beta B^{-\beta}}.
\end{equation}

For fractional split trees with $β<1$, \Cref{lem:change} is not strong enough, so we improve the bound:

\begin{lemma}\label{lem:order}
    Assume that $\phi_n$ is constant for $n>s$.
    For $n,m \in \NN$,
    \begin{equation}
        \Ek*{X_{m+n}} = \Ek*{X_m} + \Ok*{n^\beta}.
    \end{equation}
    In particular, $\Ek*{X_n} = \Ok*{n^\beta}$.
\end{lemma}

\begin{proof}
    Consider adding $n$ balls to an existing fractional split tree with $m$ balls. 
    Note that the number of additional balls that reach a fixed node $r$ is also made up by a binomial part $n'_r$ and a remainder term $n''_r$.
    Since $\phi_l$ is constant for $l>s$, the term $\norm\phi := \max_l \phi_l - \min_l \phi_l$ is finite. Define $P$ and $R$ as before in terms of the added balls, so e.g.\ $P := \mg{v\in [b]^* \mid M^n_v≥B}$.
    Because $\norm\phi$ is finite, the expected costs of nodes in $P$ changes at most by $\norm\phi\cdot\abs P$ when adding balls.
    By \Cref{lem:change} applied to the number of nodes $X_n(1)$, adding a ball to a subtree increases the number of nodes by at most a constant in expectation, so
    \begin{equation}
        \abs[\big]{\Ek{X_{m+n} - X_m}} \le \norm\phi\Ek[\big]{\abs P} + \Ok*{\E \sum_{r\in R} \kl{n_r' + \abs{n_r''}}}.
    \end{equation}
    By \Cref{lem:r}, both $\Ek{\abs P}$ and $\Ek{\abs R}$ are of order $\Ok*{n^\beta}$, and by \eqref{eq:nss},
    the sum of $\abs{n_r''}$ also has order $\Ok*{n^\beta}$. For the sum of $n_r'$, use the $\sigma$-algebra $\Gc$ defined before in \eqref{def:gc}. The set $R$ is $\Gc$-measurable and $n'_r$ is binomially $\Bin(n, M^1_r)$ distributed conditional on $\Gc$. Therefore,
    \begin{equation}\label{pf:ns-sum}
        \Ek*{\sum_{r \in R} n_r' \given \Gc} = \sum_{r \in R} \Ek*{n_r' \given \Gc} = \sum_{r \in R} M^n_r ≤ \abs{R} B.
    \end{equation}
    We take expectations and apply \Cref{lem:r} again to finish the proof.
\end{proof}

We will later use the argument in \eqref{pf:ns-sum} multiple times without explicit mention of $\Gc$.

\subsection{Proof of Theorem~\ref{thm:mean} for a bounded special case}\label{sec:constant-phi}

We continue to assume that $\phi_n$ is constant for $n>s$ and prove \Cref{thm:mean} for this case in this subsection.
By recursing until we hit a node in $R$, we can write $\E X_n$ as
\begin{equation}\label{r-recursion}
    \E X_n = \phi_{s+1}\E \abs{P} + \E \sum_{r\in R} X_r = \E \sum_{r\in R} X_r + \Ok*{n^\beta B^{-\beta}},
\end{equation}
where $X_r$ is the sum of costs of nodes that are descendants of $r$. Recall that $n_r$ denotes the number of balls in the subtree of $r$. Given $n_r$ for $r\in R$, the expectation of $X_r$ is by \Cref{lem:order} of order $\Ok*{n_r^\beta}$. This $n_r$ is typically close to $M^n_r$, and in the following lemma we show that the contribution of $r\in R$ where $n_r$ is far from $M^n_r$ is negligible.

\begin{lemma}\label{lem:wenige-extrem}
    Define $R' := \mg*{r\in R \mid \abs{n_r-M^n_r} < B^{0.6}}$, the nodes with a number of balls in the subtree close to their weight. This set
    contains most nodes in $R$, in the sense that
    \begin{equation}\label{eq:rs}
        \Ek*{\abs R-\abs{R'}} = \Ok*{n^\beta B^{-β-0.2}}.
    \end{equation}
    The contribution of subtrees of nodes in $R\setminus R'$ to $\Ek{X_n}$ is also negligible, in the sense that
    \begin{equation}\label{lem:wenige-extrem:eq}
        \Ek*{
            \sum_{r\in R\setminus R'} n_r^\beta
        } = \Ok*{n^\beta B^{-(0.2 \minv \beta)}}.
    \end{equation}
    Because of \Cref{lem:order} and \eqref{r-recursion}, this means that if $\phi_n$ is constant for $n>s$,
    \begin{equation}\label{rs-recursion}
        \Ek*{X_n} = \Ek*{\sum_{r\in R'} X_r} + \Ok*{n^\beta B^{-0.2 \minv \beta}}.
    \end{equation}
\end{lemma}
\begin{proof}
    We start by proving \eqref{eq:rs}. As before, we can write $n_r = n'_r + n''_r$, where $n'_r \sim \Bin(n, M^1_r)$ and $n''_r$ accounts for the effects of $s_0$ and $s_1$. If $r\in R\setminus R'$, then $\abs{n_r-M^n_r}≥B^{0.6}$. Recall $n_r=n_r'+n_r''$ by \eqref{def:nss}, so either $\abs{n_r'-M^n_r} ≥ B^{0.6}/2$ or $\abs{n_r''} ≥ B^{0.6}/2$. Using Markov bounds,
    \begin{align}
        \Ek*{\abs R-\abs{R'}} &= \Ek*{\sum_{r\in R} \eins\mg*{\abs{n_r-M^n_r} > B^{0.6}}}
        \nonumber \\&≤ \Ek*{\sum_{r\in R} \eins\mg*{\abs{n'_r-M^n_r} > \frac{B^{0.6}}2} + \eins\mg*{\abs{n''_r} > \frac{B^{0.6}}2}}
        \nonumber \\&≤ \Ek*{\sum_{r\in R} 4\kl*{n_r'-M^n_r}^2B^{-1.2} + 2\abs{n''_r} B^{-0.6}}.
    \end{align}
    Now, we will use the fact that $n_r'$, conditional on $\Gc$ (see \eqref{def:gc}), is binomial $\Bin(n, M^1_r)$-distributed and that its conditional variance $\Ek*{(n_r'-M^n_r)^2\given \Gc}$ is bounded by $M^n_r ≤ B$. Thus, arguing as in \eqref{pf:ns-sum}
    \begin{align}
        \Ek*{\abs R-\abs{R'}} &≤ \Ek[\big]{\abs{R}}\, 4B^{-0.2} + 2B^{-0.6}\Ek*{\sum_{r\in R}\abs{n''_r}}
        \nonumber \\&= \Ok*{n^\beta B^{-\beta-0.2}} + \Ok*{n^\beta B^{-\beta-0.6}},
        \label{eq:rs-proof}
    \end{align}
    where the last equality follows from \Cref{lem:r} and \eqref{eq:nss}.

    To show \eqref{lem:wenige-extrem:eq}, we see that $n_r^β ≤ B^β + \abs{n'_r-M^n_r}^β + \abs{n''_r}^β$ for $r\in R$. If $\abs{n_r-M^n_r}$ is larger than $B^{0.6}$, then $\abs{n_r'-M^n_r}$ or $\abs{n_r''}$ must be larger than $B^{0.6}/2$. If $\abs{n_r'-M^n_r}$ is larger than $\abs{n_r''}$, we bound $n_r^β ≤ B^β + 2\abs{n_r'-M^n_r}^β$, else we bound $n_r^β ≤ B^β + 2\abs{n_r''}^β$, so
    \begin{multline}\label{lem:wenige-extrem:aufteilung}
        \E \; {
            \sum_{\mathclap{r\in R\setminus R'}} n_r^\beta
        }
        ≤ 
        B^\beta \E\abs*{R\setminus R'}
        + 2\Ek*{
            \sum_{r\in R} \abs*{n''_r}^\beta \eins\mg*{\abs{n_r''} ≥ \frac{B^{0.6}}2}
        }
        \\
        + 2\Ek*{
            \sum_{r\in R} \abs*{n'_r-M_r^n}^\beta \eins\mg*{\abs{n_r'-M^n_r} ≥ \frac{B^{0.6}}2}
        }
        .
    \end{multline}
    The first term $B^\beta\E\abs*{R\setminus R'}$ is $\Ok*{n^β B^{-0.2}}$ because of \eqref{eq:rs}.
    For the second term in \eqref{lem:wenige-extrem:aufteilung}, we do not need the indicator and $\beta$, because
    \begin{equation}
        \Ek*{
            \sum_{r\in R} \abs*{n''_r}^\beta \eins\mg*{\abs{n_r''} ≥ \frac{B^{0.6}}2}
        }
        ≤ \Ek*{
            \sum_{r\in R} \abs*{n''_r}
        }
        % ≤ 2s\E\abs{P}
        = \Ok*{n^βB^{-β}},
    \end{equation}
    again using \eqref{eq:nss}.
    For the third term in \eqref{lem:wenige-extrem:aufteilung},
    we use the Markov inequality to be able to again bound in terms of the conditional variance $\Vark{n'_r\given \Gc}≤ M^n_r≤B$:
    \begin{align}
        \MoveEqLeft \Ek*{
            \sum_{r\in R} \abs*{n'_r-M_r^n}^\beta \eins\mg*{\abs{n_r'-M^n_r} ≥ \frac{B^{0.6}}2}
        }
        \nonumber \\
        &≤ \Ek*{ 
            \sum_{r\in R} \frac1{\abs*{n'_r-M_r^n}^{2-\beta}} \eins\mg*{\abs{n_r'-M^n_r} ≥ \frac{B^{0.6}}2} \abs*{n'_r-M_r^n}^2
        }
        \nonumber \\
        &≤ \Ek*{\sum_{r\in R}
        2^{2-β}B^{0.6\beta-1.2} \abs{n_r'-M^n_r}^2
        }
        \nonumber \\
        &≤ \Ek*{\sum_{r\in R}
        2^{2-β}B^{0.6\beta-0.2}
        } = \Ok*{n^\beta B^{-0.2-0.4\beta}}
    \end{align}
    using that $2-β>0$ and \Cref{lem:r}.
    We bounded all terms in \eqref{lem:wenige-extrem:aufteilung} by either $\Ok*{n^\beta B^{-0.2}}$ or $\Ok*{n^\beta B^{-\beta}}$, finishing the proof.
\end{proof}

\begin{lemma}\label{lem:rzeta}
    Choose $\eps>0$ such that $\eps\inv$ is an integer, let $B = \eps^{-30}$
    and define the set $R_{\zeta B}$ for $\zeta\in S := \mg{\eps^2,2\eps^2, \dots, 1}$ as
    \begin{equation}\label{def:rzeta}
        R_{\zeta B} := \mg[\big]{r\in R\mid M^n_r \in [\zeta B-\eps^2 B, \zeta B)}. % chktex 9
    \end{equation}
    Then, there exists for every $\zeta\in S$ a $d$-periodic (constant in the nonlattice case) function $\varpi_\zeta ≥ 0$, such that
    \begin{equation}
        \abs*{R_{\zeta B}} = \kl*{\varpi_\zeta\kl*{\ln\frac nB} + o_\eps(1)} \fracc nB^\beta \quad \text{for } \frac nB \to \infty.
    \end{equation}
    We use the notation $o_\eps(\cdot)$ instead of $o(\cdot)$ because the convergence does not hold uniformly in $\eps$.
    The sum of the functions fulfills 
    \begin{equation} \label{eq:rzeta-upper-bound}
        \sum_{\zeta\in S} \varpi_\zeta(q) ≤ \frac{be^{bd}}{\beta\mu}
    \end{equation} for all $q\in \RR$.
    Furthermore, \begin{equation} \label{eq:rzeta-lower-bound}
        \inf_q \sum_{\substack{\zeta\in S\\\zeta≥e^{-d}}}\varpi_\zeta(q) > 0.
    \end{equation}
\end{lemma}
\begin{proof}
    This proof will be a slight refinement of the proof of \Cref{lem:r}.
    For a uniformly random node $v\in[b]^k$ at depth $k$, the mass $M^1_v$ is the product of $k$ independent copies of the random variable $V$, a uniformly random element of $\V$. Calling this product $Y_k$, a uniformly chosen child $u$ of this node has mass $Y_k V$, with $V$ independent of $Y_k$. The child $u$
    is in $R_{\zeta B}$ if and only if $nY_k ≥ B$ (its parent is in $P$) and $nY_k V \in [\zeta B-\eps^2 B, \zeta B)$.
    Applying $-\ln(\cdot)$ and writing $q := \ln\frac nB$, this is equivalent to $-\ln Y_k -q ≤ 0$ and $-\ln Y_k-\ln V -q \in {(-\ln \zeta, -\ln(\zeta-\eps^2)]}$. In the case of $\zeta = \eps^2$, we set $\ln(\zeta-\eps^2) := -\infty$.
    Note that this only depends on $n$ and $B$ via $q = \ln \frac nB$, so we can define a function $Z(q)$ such that $Z\kl*{\ln\frac nB}=\E\abs[\big]{R_{\zeta B}}$ as follows:% chktex 9
    \begin{equation}\label{def:z}
        Z(q) := \sum_{k=0}^\infty b^{k+1} \Pk*{-\ln Y_k-q ≤ 0,\,-\ln Y_k -q - \ln V \in \left(-\ln \zeta, -\ln\kl*{\zeta - \eps^2}\right]}. % chktex 9
    \end{equation}
    Using the conditional expectation on the first independent copy of $V$ in $Y_k$, so that $-\ln Y_{k+1}-q$ has the same distribution as $-\ln Y_k - (q + \ln V)$, we can write the sum \eqref{def:z} as a recursion
    \begin{equation}
        Z(q) = b\Pk*{-q - \ln V \in \left(-\ln \zeta, -\ln\kl*{\zeta - \eps^2}\right]} % chktex 9
        \eins\kl*{q≥0}
         + b\Ek*{Z(q+\ln V)}.
    \end{equation}
    Defining
    \begin{equation}\label{def:g-zeta}
        G_\zeta(q) := b\Pk*{-q - \ln V \in \left(-\ln \zeta, -\ln\kl*{\zeta - \eps^2}\right]}\eins\kl*{q≥0}, % chktex 9
    \end{equation}
    we see that $Z(q)$ fulfills the recursion equation $Z = G + \kl*{b\P_{-\ln V} * Z}$.
    We will use the same size-biasing as in the proof of \Cref{lem:r}.
    Since $G_\zeta(q)$ is bounded and continuous almost everywhere, $G_\zeta(q)e^{-\beta q}$ is directly Riemann integrable, so the key renewal theorem
    \Cref{thm:krt} shows that in the nonlattice case, for $q=\ln\frac nB \to \infty$,
    \begin{equation}
        Z(q) \sim \frac {e^{\beta q}}\mu\int_{0}^\infty G_\zeta(u)e^{-\beta u}\mathrm du =: \fracc nB^\beta \varpi_\zeta \quad\text{ as }q\to\infty
    \end{equation}
    with $\varpi_\zeta$ constant.
    Note that $\varpi_1>0$ since $G_1(u) = b\Pk*{-\ln V\in (u,u-\ln(1-\eps^2)]} > 0$ % chktex 9
    for at least a small interval, showing \eqref{eq:rzeta-lower-bound}.

    In the lattice case, define the $d$-periodic function
    \begin{equation}\label{def:varpi-zeta}
        \varpi_\zeta\kl*{q} := \frac d\mu\sum_{u=0}^\infty G_\zeta(ud + \restd q)e^{-\beta(ud+\restd q)},
    \end{equation}
    then \Cref{thm:krt} implies that $\E\abs[\big]{R_{\zeta B}} = Z(q) = \fracc nB^{β}\kl*{\varpi_\zeta\fracc nB + o(1)}$ as $q = \ln\frac nB\to\infty$.
    On the other hand, summing $\varpi_\zeta$ up over all $\zeta\in S$ yields for $q≥0$ (i.e.\ $n≥B$)
    \begin{align}
       \sum_{\zeta\in S} \varpi_\zeta(q) &=   \frac1\mu\sum_{\zeta\in S}\sum_{u=0}^\infty G_\zeta(ud+\restd q)de^{-\beta (ud+\restd q)}
        \nonumber \\&= \frac d\mu \sum_{u=0}^\infty e^{-\beta ud-\beta\restd q}b\sum_{\zeta\in S}\Pk*{-\ln V -ud-\restd q\in \left(-\ln \zeta, -\ln\kl*{\zeta - \eps^2}\right]} % chktex 9
        \nonumber \\&≤ \frac {db}\mu \sum_{u=0}^\infty e^{-\beta ud}
        = \frac{db}{\mu(1-e^{-\beta d})}
        ≤ \frac{be^{βd}}{\beta\mu},
    \end{align}
    and an analogous calculation can be done in the nonlattice case as well, with an integral over $e^{-\beta u}$ in place of the sum, proving \eqref{eq:rzeta-upper-bound}.

    Back to the lattice case, summing instead over $\zeta\in S$ with $\zeta≥e^{-d}$, by \eqref{def:varpi-zeta} and \eqref{def:g-zeta},
    \begin{align}
        \sum_{\mathclap{\substack{\zeta\in S\\\zeta≥e^{-d}}}} \varpi_\zeta(q)
        &= \frac d\mu \sum_{u=0}^\infty e^{-\beta(ud+\restd q)}b\sum_{\mathclap{\substack{\zeta\in S\\\zeta≥e^{-d}}}}\Pk*{-\ln V-ud-\restd q \in \left(-\ln \zeta, -\ln\kl*{\zeta - \eps^2}\right]} % chktex 9
        \nonumber \\& ≥ \frac{db}\mu  \sum_{u=0}^\infty e^{-\beta (u+1)d} \Pk[\big]{-\ln V \in (ud+\restd q,(u+1)d+\restd q]} % chktex 9
        \nonumber \\& = \frac{db}{\mu} \sum_{u=0}^\infty e^{-\beta (u+1)d} \Pk[\big]{-\ln V = (u+1)d}
        \nonumber \\& = \frac{db}{\mu} \Ek*{V^β \eins(V \notin \mg{0,1})} > 0.
    \end{align}
    The inequality holds because the intervals $(-\ln\zeta, -\ln(\zeta-\eps^2)]$ fill out at least the interval $(0,d]$. % chktex 9
    The third line is because of $-\ln V \in d\ZZ$. In the last line, $V$ cannot be $0$ since then $-\ln V$ would be infinite, and it cannot be $1$ because the summation starts at $u=0$.
\end{proof}

We can now prove \Cref{thm:mean} in the case that $\phi_n$ is constant for $n>s$.

\begin{proposition}\label{lem:constant-phi}
    If $\phi_n$ is constant for $n>s$, there exists a bounded, continuous, $d$-periodic function $\varpi$, such that
 \begin{equation}\label{res:constant-phi}
     \Ek{X_n} = \varpi(\ln n) n^\beta + o(n^β).
 \end{equation}
 In the nonlattice case, $\varpi$ is a constant.
\end{proposition}
\begin{proof}
    We use the same notation as in \Cref{lem:rzeta}, in particular $B = \eps^{-30}$ and $R_{ζB}$. Note that O-notation in this proof is to be understood as $n\to\infty$ uniformly in $\eps$ and any fixed nodes.
    % In view of \Cref{lem:order}, we only need to show that $\Ek{X_n}n^{-\beta}$ is Cauchy sequence in the nonlattice case.
    Remember $R':= \mg*{r\in R \mid \abs{n_r-M^n_r} < B^{0.6}}$ from \Cref{lem:wenige-extrem}, which shows that $R'$ contains most of the nodes and balls in $R$. For $\zeta\in S$, a node $r$ in $R'\cap R_{\zeta B}$ has $\abs{n_r-M^n_r} < B^{0.6}$ and $\abs{M^n_r-\zeta B} ≤ \eps^2 B$, so $\abs{n_r - \zeta B} < B^{0.6} + \eps^2 B$ and $\Ek*{X_r\given r\in R'\cap R_{\zeta B}} = \Ek*{X_{\zeta B}} + \Ok*{B^\beta\eps^{2\beta} + B^{0.6\beta}}$ by \Cref{lem:order}. Because $\eps\inv$ is an integer, the term $\zeta B$ is an integer, too. Hence, using \eqref{rs-recursion} and \Cref{lem:r},
    \begin{align}
        \Ek*{X_n}
         &= \Ek*{\sum_{r\in R'} X_r} + \Ok*{n^\beta B^{-0.2\minv \beta}}
        \nonumber \\&= \Ek*{\sum_{\zeta\in S}\sum_{r\in R' \cap R_{\zeta B}} X_r} + \Ok*{n^\beta B^{-0.2\minv \beta}}
        \nonumber \\&= \Ek*{\sum_{\zeta\in S}\sum_{r\in R' \cap R_{\zeta B}} \Ek{X_{\zeta B}}} + \Ok*{\E\abs{R'}\kl*{\eps^{2\beta} B^{\beta} + B^{0.6\beta}}} + \Ok*{n^\beta B^{-0.2\minv \beta}}
        \nonumber \\&= \sum_{\zeta\in S}\E \, \abs*{R' \cap R_{\zeta B}}\,\Ek{X_{\zeta B}}
        + \Ok*{n^\beta\kl*{
        \eps^{2\beta}
        + B^{-0.4\beta}
        +B^{-0.2\minv \beta}}
        }
    \end{align}
    Remember that $B=\eps^{-30}$, so $B^{-0.2} = \eps^6$ and $B^{-0.4\beta} = \eps^{12\beta}$, so we can summarize the error terms as $\Ok*{\eps^{2\beta}n^\beta}$. By \eqref{eq:rs}, $\Ek*{\abs{R}-\abs{R'}} = \Ok*{n^βB^{-β-0.2}}$ and $\Ek{X_{\zeta B}} = \Ok*{B^\beta}$ by \Cref{lem:order}, so we can remove the intersection with $R'$:
    \begin{equation}
        \Ek*{X_n}
         = \sum_{\zeta\in S}\E \, \abs*{R_{\zeta B}}\,\Ek{X_{\zeta B}} + \Ok*{n^\beta B^{-0.2}} + \Ok*{\eps^{2\beta}n^\beta}.
    \end{equation}
    By \Cref{lem:rzeta}, $\abs*{R_{\zeta B}} = \kl{\varpi_\zeta\kl*{\ln\frac nB} + o_\eps(1)}\fracc nB^\beta$; remember that $o_\eps(1)$ stands for a term that converges to $0$ as $n\to\infty$ with a convergence speed depending on $\eps$; in contrast to $o(1)$, which is uniform in $\eps$. Hence,
    \begin{equation}
        \Ek*{X_n}
         = \sum_{\zeta\in S}\kl*{\varpi_\zeta\kl*{\ln\frac nB} + o_\eps(1)}\,\Ek{X_{\zeta B}}\fracc nB^\beta  + \Ok*{\eps^{2\beta}n^\beta}.
    \end{equation}
    For fixed $\eps$, this sum goes over finitely many terms, and $\Ek{X_{\zeta B}} = \Ok*{B^\beta}$ by \Cref{lem:order},
    \begin{equation}\label{nn-closeness}
        \Ek*{X_n}
         = \sum_{\zeta\in S}\varpi_\zeta\kl*{\ln\frac nB}\,\Ek{X_{\zeta B}}\fracc nB^\beta  + \Ok*{\eps^{2\beta}n^\beta} + o_\eps\kl*{n^\beta}.
    \end{equation}

    In the nonlattice case, the functions $\varpi_\zeta$ are constant. Thus, for $n'\in \NN$ with $n'>n$,
    \begin{equation}\label{pf:nonlattice-cauchy}
        \abs*{\Ek*{X_n}n^{-\beta} - \Ek*{X_{n'}}\kl{n'}^{-\beta}} = \Ok*{\eps^{2\beta}} + o_\eps\kl*{1}.
    \end{equation}
    For $\eps$ sufficiently small and $n$ large the right-hand side of \eqref{pf:nonlattice-cauchy} is small, so $\Ek{X_n}n^{-\beta}$ is a Cauchy sequence converging to some limit $\varpi$, showing the result \eqref{res:constant-phi}.

    In the lattice case, define the $d$-periodic function
    \begin{equation}
        \varpi_B(t) := \sum_{\zeta\in S}\varpi_\zeta\kl*{t-\ln B} \Ek{X_{\zeta B}}B^{-\beta},
    \end{equation}
    so that by \eqref{nn-closeness}
    \begin{equation}\label{varpi-b-closeness}
        \Ek*{X_n} = \varpi_B(\ln n)n^\beta + \Ok*{\eps^{2\beta}n^\beta} + o_\eps\kl*{n^\beta}.
    \end{equation}
    We will show that the functions $\varpi_B$ for $B\to\infty$ form a Cauchy sequence in the uniform metric. 
    Then define the limit as $\varpi := \lim_{B\to\infty}\varpi_B$ and \eqref{varpi-b-closeness} would conclude the proof of \eqref{res:constant-phi}.
    
    By the definition of $\varpi_\zeta$ in \eqref{def:varpi-zeta},
    \begin{equation}\label{def:varpi-b}
        \varpi_B(t) = \frac d\mu \sum_{\zeta\in S}\sum_{u=0}^\infty G_\zeta(ud + \restd{t-\ln B})e^{-\beta(ud+\restd{t-\ln B})}
        \Ek{X_{\zeta B}}B^{-\beta}.
    \end{equation}
    Recall the fact that $\ln V\in d\ZZ$ and the definition of $G_\zeta(t)$ in \eqref{def:g-zeta} as
    \begin{equation}\label{def:g-zeta:rep}
        G_\zeta(t) = b\Pk*{-t - \ln V \in \left(-\ln \zeta, -\ln\kl*{\zeta - \eps^2}\right]} % chktex 9
        \eins\kl*{t≥0},
    \end{equation}
    from which we see that $G_\zeta$ is a step function with $d$-periodic jumps at $\ln \zeta+d\ZZ$ and at $\ln(\zeta-\eps^2)+d\ZZ$ if $\zeta\ne\eps²$.
    Therefore, $\varpi_B(t)e^{\beta\restd{t}}$ is also a step function with jumps at $\ln(\zeta) + \ln B + d\ZZ$ for every $\zeta\in S$ and at $d\ZZ$.
    For $\eps'<\eps$ and $B':=(\eps')^{-30}$ it therefore suffices to check that $\varpi_B(t)$ and $\varpi_{B'}(t)$ are close on the dense set $\mg*{\restd{\ln n}\mid n≥N}$ for some $N$ sufficiently large. This follows from \eqref{varpi-b-closeness} because
    \begin{equation}
        \varpi_B(\ln n) = \Ek{X_n}n^{-\beta} + \Ok{\eps^{2\beta}} + o_\eps(1) = \varpi_{B'}(\ln n) + \Ok{\eps^{2\beta}} + o_\eps(1).
    \end{equation}
    To show that $\varpi$ is continuous, note that if $\varpi$ were discontinuous at some fixed point $t\in \RR$, then $\varpi_B$ for all large $B$ must also have a jump at $t$. Fix $\eps\inv\in\NN$ and $B=\eps^{-30}$. 
    %The functions $G_\zeta$ jump at $\ln \zeta +d\ZZ$ for $\zeta\in \mg{\eps^2,2\eps^2,\dots, 1}$ by \eqref{def:g-zeta:rep}, and the function $\varpi_B$ moves these jumps by $\ln B$, see \eqref{def:varpi-b}. Therefore, 
    As noted before, $\varpi_B$ jumps only at $\ln B + \ln(k\eps^2) + dj$ for $j\in\ZZ$ and $1≤k≤\eps^{-2}$.
    
    So if $t$ is a jump of $\varpi_B$, there must exist $k≤\eps^{-2}$ and $j\in \ZZ$ such that $t=\ln B + \ln(k\eps^2) + dj$.
    Therefore, $e^{t-dj} = kB\eps^2 = k\eps^{-28}$ is a natural number and a multiple of $\eps\inv$. %Hence, if a prime $p$ divides $\eps\inv$ (write $p\mid\eps\inv$) and $\varpi_B$ is discontinuous at $t$, then $p \mid e^{t-dj}$ for some $j$.
    If $\varpi$ were discontinuous at $t$, this would hold for all large $\eps\inv \in\NN$, so in particular for every prime $p$ there exist a $j\in\ZZ$ such that $p$ divides $e^{t-dj}$, which we write $p\mid e^{t-dj}$.%, or else we could choose $\eps\inv$ as multiples of $p$.

    These primes cannot all divide $e^{t-dj}$ for the same $j$, so assume that for $j_1\ne j_2$ both $e^{t-dj_1}$ and $e^{t-dj_2}$ are natural numbers.
    Any other $j\in\ZZ$ can be written as $j = łj_1+(1-ł)j_2$ with $ł\in\QQ$, so that $e^{t-dj} = \kl*{e^{t-dj_1}}^ł\kl*{e^{t-dj_2}}^{1-ł}$. This implies that $e ^{t-dj}$, if integer, can only have prime factors of $e^{t-dj_1}$ and $e^{t-dj_2}$. Consequently, there are only finitely many primes $p$ where $p\mid e^{t-dj}$ for some $j\in \NN$, contradicting the previous paragraph and showing that $\varpi$ is continuous.
    %and thus $\varpi_B$ is continuous at $t$ for all $\eps\inv$ with $p \mid \eps\inv$, implying that $\varpi$ is also continuous at $t$.
    % It seems easy to show that no $t$ is equal to $\restd{\ln(i\eps^2 B)} = \restd{\ln i-28\ln \eps}$ for some $i≤\eps^{-2}$ for all large integer $\eps\inv$, but we instead show that the jumps in $\varpi_B$ are of order $\eps^{2\beta}$. By careful analysis of \eqref{def:varpi-b}, we see that the jumps result from atoms of $-\ln V$ moving
\end{proof}
Note that we still did not show that $\varpi$ is bounded away from zero
when $\phi_n\ge0$ with some $\phi_n>0$,
which is equivalent to showing $\Ek{X_n} = \Omega\kl*{n^β}$. This will be done in \Cref{lem:lower-bound}.

\subsection{Extending to other toll functions}\label{sec:other-toll}

We already know by \Cref{lem:constant-phi} that \Cref{thm:mean} (except for the lower bound in \Cref{lem:lower-bound}) holds if $\phi_n$ is constant for $n>s$. We can first extend this to $\phi_n$ being constant for $n>B$ for some $B$. This is true because we can define a new `bucket' split tree $\widetilde T$ with the same $\V$, $s_0$ and $s_1$ as the original tree, but instead set $s$ to $B$ and define the toll function $\widetilde \phi$ by $\widetilde \phi(T_n) := \phi_{n} = \phi_{B+1}$ for $n>B$ and $\widetilde \phi(T_n) := \Ek*{X_n}$ for $n≤B$. Define $\widetilde X_n$ as the sum of the toll functions on the bucket split tree. Because the split tree can be generated by growing small split trees out of the buckets of the bucket split tree, $\Ek*{X_n} = \Ek{\widetilde X_n}$ for all $n$, and we can use \Cref{lem:constant-phi} to show \Cref{thm:mean} for this $\phi$. This is perhaps a bit clearer when looking at the recursion equation for $X_n$, which can be found later at \eqref{eq:xn-recursion}, where increasing $s$ corresponds to specifying more of the first values.

Furthermore, $X_n(\phi)$ is obviously linear in $\phi$. In the next step, we will show that the contribution from $\phi_n$ for $n$ large is negligible if $\phi_n = \Ok*{n^{\beta-\delta}}$ for $\beta≥\delta>0$.

\begin{lemma}\label{lem:tail}
    Recall $P$ from \Cref{sec:upper-bound}. For $\beta≥\delta>0$,
    \begin{equation}\label{eq:lem:tail:1}
        \E \sum_{v \in P} (M_v^n)^{\beta-\delta} = \Ok*{n^{\beta}B^{-\delta}}.
    \end{equation}
    Furthermore,
    \begin{equation}\label{eq:lem:tail:2}
        \E \sum_{v\in [b]^*} n_v^{\beta-\delta}\eins(n_v>2B+s_1) = \Ok*{n^{\beta}B^{-\delta}}.
    \end{equation}
\end{lemma}
\begin{proof}
    The first equation \eqref{eq:lem:tail:1} can be shown similar to \Cref{lem:r}, by (re)defining for $q:=\ln\frac nB$ % chktex 36
    \begin{equation}\label{def:lem:tail:U}
        U(q) := \sum_{v \in [b]^\ast} \eins(M_v^1 ≥ e^{-q})(M_v^1)^{β-δ}e^{q(β-δ)} = 
        \E \sum_{v \in P} \fracc{M_v^n}B^{\mathrlap{\beta-\delta}}.
    \end{equation}
    Note that if $q<0$, then $P$ is empty and thus $U(q):=0$. 
    By splitting according to the first step, we get
    \begin{equation}
        U(q) = \fracc nB^{\mathrlap{\beta-\delta}} \;\eins(q≥0) + b\E \sum_{v \in P} \fracc{M_v^{nV}}B^{\mathrlap{\beta-\delta}} \;
        = e^{(β-δ)q}\eins(q≥0) \; + \Ek*{U(q-\ln(1/V))}.
    \end{equation}
    This is a renewal equation and the term $e^{-q\beta}e^{q\kl{\beta-\delta}}\eins(q≥0) = e^{-q\delta}\eins(q≥0)$ is directly Riemann integrable, so we can again use \Cref{thm:krt} to deduce $U(q) = \Ok*{e^{\beta q}} = \Ok*{n^βB^{-β}}$. Multiplying \eqref{def:lem:tail:U} with $B^{\beta-\delta}$ gives \eqref{eq:lem:tail:1}.
    For the second equation \eqref{eq:lem:tail:2}, we first prove using \eqref{eq:lem:tail:1} that
    \begin{align}
        \Ek[\bigg]{\sum_{v\in P} n_v^{\beta-\delta}} &≤ 
        \Ek*{ \sum_{v\in P} \Ek*{(n'_v)^{\beta-\delta}\given \Gc} + (n_v'')^+}
        \nonumber \\
        &≤ \Ek*{\sum_{v\in P} (M_v^n)^{\beta-\delta}} + \Ok*{n^{\beta}B^{-β}} =
        \Ok*{n^{\beta}B^{-\delta}}.
        \label{eq:lem:tail:3a}
    \end{align}
    Here, we used \Cref{lem:r} and the bound $n_v''≤s_1$
from \eqref{eq:n''-upper}.
%, which can be shown by a slight refinement of the coupling argument in
%\eqref{def:nss}: 
%    We see that this holds for children of the root by comparison of
%    \eqref{def:nv} with \eqref{def:ns}. Then, we assume that the $s_1$
%    additional balls are included in the $s_0+bs_1$ balls missing, so that
%    there again can be at most $s_1$ balls more in $n_v$ than in $n'_v$. 
%    The statement follows by iterating.

    For the rest of the tree outside $P$, which consists of nodes in $R$ and their descendants, we use a similar argument as in \Cref{lem:wenige-extrem}. First note that the contribution of a subtree with $m≥1$ balls is crudely bounded by $\Ok*{m^{2\beta-\delta}}$, since the tree has $\Ok*{m^β}$ nodes by \Cref{lem:order},
    that each contribute at most $\Ok*{m^{\beta-\delta}}$. Hence,
    \begin{equation}
        \E{\sum_{v \in [b]^*} n_v^{\beta-\delta}\eins(n_v>2B+s_1, M^n_v<B)}
        = \Ok*{\El{ \sum_{v \in R} (n_v' + s_1)^{2\beta-\delta}\eins(n_v'>2B)}}.
    \end{equation}
    If $n'_v>2B$, then $n'_v+s_1 ≤ n'_v(s_1+1)$, so the term $s_1$ only contributes a constant factor. Since $2β-δ<2$, we can use $(n'_v)^{2β-δ-2} ≤ (2B)^{2β-δ-2}$ to bound
    \begin{align}
        \Ek*{ \sum_{v \in R} (n_v')^{2\beta-\delta}\eins(n_v'>2B)}
        &= \Ok*{\Ek*{ \sum_{v \in R} (n_v')^{2}(2B)^{-2+2\beta-\delta}\eins(n_v'>2B)}}
        \nonumber \\
        &=  \Ok*{\Ek*{ \sum_{v \in R} \Ek*{(2n_v'-2M^n_v)^2\given\Gc} B^{-2+2\beta-\delta}}}
        \nonumber \\
        &= \Ok*{\Ek*{ \sum_{v \in R} B \cdot B^{-2+2\beta-\delta}}} = \Ok*{n^{\beta}B^{-\delta}},
        \label{eq:lem:tail:3}
    \end{align}
    using first the fact that $n_v'$ has conditional expectation $M^n_v$ and conditional variance of at most $M^n_v≤ B$ and then \Cref{lem:r}. Combining \eqref{eq:lem:tail:3a} and \eqref{eq:lem:tail:3} gives \eqref{eq:lem:tail:2}.
\end{proof}

\begin{proof}[Proof of \Cref{thm:mean}]
For a toll function with $\abs{\phi_n} ≤  C{n^{\beta-\delta}}$ for some $\delta>0$ and $C>0$, we can for all $B$ write $\phi_n$ as
\begin{equation}
\phi_n = \phi_n\eins(n>2B+s_1) + \phi_n\eins(n≤2B+s_1)
\end{equation}
Because the second part is zero for all large $n$, we can apply \Cref{lem:constant-phi} to it, giving us a limit function $\varpi_B$. The
absolute value of the first part is bounded by $Cn^{\beta-\delta}\eins(n>2B+s_1)$, so its contribution to $\Ek{X_n}$ is by \Cref{lem:tail} of the order $\Ok*{n^{\beta}B^{-\delta}}$. Therefore,
\begin{equation}\label{ot:approx}
    \Ek{X_n} = \varpi_B(\ln n)n^\beta + \Ok*{n^{\beta}B^{-\delta}}
    + o_B(n^β),
\end{equation}
where $o_B(n^β)$ is a term that is $o(n^β)$, but not uniformly in $B$.
It follows as in the proof of \Cref{lem:constant-phi} that $\varpi_B$ is a Cauchy sequence in the uniform metric, so we can define $\varpi := \lim_{B\to\infty}\varpi_B$ and show \Cref{thm:mean} for this case. Since $\varpi_B$ is continuous and $d$-periodic for every $B$, the same holds for $\varpi$. The lower bound of $\varpi$ will be proven in the next section, \Cref{sec:lower-bound}.
\end{proof}

This argument also can be used to show \Cref{prop:sum}:
\begin{proof}[Proof of \Cref{prop:sum}]
    First, observe that $\varpi[\phi]$ is linear in $\phi$, since $X_n(\phi)$ is linear in $\phi$. Indeed, if we have two toll functions $\phi$ and $\phi'$ with $\phi_n, \phi_n' = \Ok*{n^{β-δ}}$ for some $δ>0$ and $\alpha\in \RR$, then by \Cref{thm:mean}
    \begin{equation}
        X_n(\phi_n + \alpha\phi_n') = X_n(\phi_n) + \alpha X_n(\phi_n') = \kl[\big]{\varpi[\phi](\ln n) + \alpha\varpi[\phi'](\ln n)}n^β + o(n^β)
    \end{equation} 
    and thus $\varpi[\phi + \alpha\phi'] = \varpi[\phi] + \alpha\varpi[\phi']$. Note that the function $\varpi_B$ from \eqref{ot:approx} above is hence given as
    \begin{equation}
        \varpi_B = \varpi[\phi_n\eins\kl{n≤2B+s_1}] = \sum_{n=1}^{2B+s_1} \phi_n \varpi[\fringephi[n]].
    \end{equation}
    And after \eqref{ot:approx} we showed that $\varpi_B \to \varpi[\phi]$ uniformly, showing \Cref{prop:sum}. 
\end{proof}

\subsection{Lower bound on the mean}\label{sec:lower-bound}

\begin{lemma}\label{lem:lower-bound}
    If $\phi_n ≥ 0$ for all $n$ and $\phi_n > 0$ for some $n≥s_1$, then
    \begin{equation}
        \Ek{X_n(\phi)} = \Omega\kl*{n^β}.
    \end{equation}
    In particular, $\varpi$ from \Cref{thm:mean} is bounded away from zero.
\end{lemma}
\begin{proof}
    First, we show that $F(n) := \Ek{X_n} = \Theta(1)$. From the assumption on $\phi$ we know that $\Ek{X_n} > 0$ for all large $n$, but $\Ek{X_n}$ could still be converging to zero. To show that this is not the case, we fix some large $N$ and consider only the nodes 
with at most $N$ balls; by grouping these nodes according to their
nearest ancestor with more than $N$ balls, we obtain for $n>N$
    \begin{equation}
        F(n) = \Ek{X_n} \ge \Ek[\bigg]{\sum_{v\in[b]^\ast} \sum_{i=1}^b \eins(n_v>N, n_{vi}≤N) F(n_{vi})}.
    \end{equation}
    By choosing $N$ big enough, $c := \min\mg{F(N-s),F(N-s+1),\dots, F(N)}$
    is positive, and we can bound $\Ek{X_n}$ by, for any $\gamma>0$,
    \begin{align}
        \Ek{X_n} &≥ \sum_{v\in[b]^\ast} \sum_{i=1}^b c\Pk{N-s≤n_{vi}≤N<n_v}
        \nonumber \\ &\ge \sum_{v\in[b]^\ast} \sum_{i=1}^b c\Pk{n_{vi}≤N<n_v, V^{(v)}_i > \gamma}
        \Pk{N-s≤n_{vi}\given n_{vi}≤N<n_v, V^{(v)}_i > \gamma},
        \label{lb:sum}
    \end{align}
    where $V^{(v)}_i$ is the $i$-th splitter of $v$, i.e.\ $M^1_{vi}/M^1_v$.
    We claim that in \eqref{lb:sum} both the sum over the unconditioned probabilities and the conditioned probabilities are $\Theta(1)$
if $\gamma$ is small enough.

    For the sum of probabilities note first that $\lim_{\gamma\to 0}
    \sum_{i=1}^b \Pk*{V_i>\gamma} = \sum_{i=1}^b \Pk*{V_i>0} > 1$ by
    continuity from below and our assumptions on $\V$. So we can choose
    $\gamma>0$ such that $\sum_{i=1}^b \Pk*{V_i>\gamma} > 1$. Consider the set
    $\mathcal A$ of nodes $vi$ for $v\in[b]^\ast$, $i\in[b]$ such that
    $V^{(v)}_i>\gamma$. Every node has by the assumption on $\gamma$ in
    expectation more than 1 child in $\mathcal A$, and the numbers of children in $\mathcal A$ are independent and have the same distribution for every node. Defining the root to be in $\mathcal A$, the connected component of the root is thus a supercritical Galton-Watson tree and there is a positive probability that $\mathcal A$ contains an infinitely long branch, on which there is always one node $vi$ with $n_{vi} ≤ N < n_v$. Since this node also has $V^{(v)}_i>\gamma$ by virtue of being in $\mathcal A$, this proves that 
the sum of the unconditional probabilities in \eqref{lb:sum} is $\Theta(1)$.

    For the conditional probabilities in \eqref{lb:sum}, observe that additionally conditioning on $n_v$ removes the dependency on the node $v$. It is therefore enough to
    show that for any $1≤i≤b$,
    \begin{align}
        & \inf_{n>N}\Pk{N-s≤n_{i}\given n_{i}≤N, V_i > \gamma}
        \nonumber \\
        ={}& \inf_{n>N}\frac{
            \sum_{j=N-s-s_1}^{N-s_1} \Ek*{\binom{n-s_0-bs_1}{j}
             V_i^j(1-V_i)^{n-s_0-bs_1-j} \given 1\ge V_i>\gamma}
        }{
            \sum_{j=0}^{N-s_1} \Ek*{\binom{n-s_0-bs_1}{j} V_i^j
             (1-V_i)^{n-s_0-bs_1-j} \given 1 \ge V_i>\gamma}
        },
    \end{align}
    is positive. The binomial weights follow from \eqref{def:nv}.
    The term in the infimum is positive for any finite $n$ (this requires
    the fact that $s\ge s_0+bs_1-1$ for $n=N+1$), so it suffices to show that its liminf is positive. Since $V_i>\gamma$ is bounded from below, the dominating term is the one with $j=N-s_1$, and the fraction converges to 1, concluding the proof that $\Ek{X_n} = \Theta(1)$.
    
    % Condition on $\Gc$ and consider the branch $v_0,v_1,v_2,\dots$ of the tree so that $v_0=\eps$ and $v_{i+1}$ is the child of $v_i$ with the highest mass $V_{(1)}M_{v_i}^n$. In the case of ties, any choice will do.
    % Let $n>N_0$ and consider the last node $v$ on this branch that has $n_v>N_0$ and condition on $n_v$.
    % By conditioning on $V_{(1)}$ of this node, we find the probability that its left-most child $v1$ has exactly $n_{v1} = N_0$ balls in its subtree to be
    % \begin{equation}
    %     \frac{
    %         \Ek*{\binom{n_v}{N_0} V_{(1)}^{N_0}(1-V_{(1)})^{n_v-N_0} \given n_v}
    %     }{
    %         \Ek*{\sum_{j=0}^{N_0} \binom{n_v}{j} V_{(1)}^j (1-V_{(1)})^{n_v-j} \given n_v}
    %     } \longrightarrow \Pk{V_{(1)}>0} \text{ as } n_v \to \infty
    % \end{equation}
    % and positive. Therefore, the probability that the left-most branch has at least one node with $N_0$ balls is positive, and therefore $\inf_{n≥N_0}\Ek{X_n}>0$.

    We can thus choose $c_0>0$ and $N_0≥1$ such that $\Ek{X_n}≥c_0$ for all $n≥N_0$, set $\eps = N_0\inv$ and use $B=\eps^{-30} = N_0^{30}$ and $R_{\zeta B}$ from \Cref{lem:rzeta}.  Define
    \begin{equation}
        R'' := R' \cap \bigcup_{\mathclap{\substack{\zeta\in S\\\zeta≥B^{-0.3}+\eps²}}} R_{\zeta B},
    \end{equation}
    i.e.\ nodes $v \in R$ which have $M_v^n ≥ B^{0.7}$ and $\abs{n_v-M_v^n} < B^{0.6}$ and therefore $n_v≥B^{0.7}-B^{0.6}>N_0$. 
    Note that $B^{-0.3}=\eps^{9}$ is much smaller than $\eps^2$, so we can instead sum over $\zeta≥2\eps^2$.
    The subtrees of the nodes in $R''$ are independent conditional on $\Gc$ and started with at least $N_0$ balls, so that
    \begin{align*}
        \Ek{X_n} &≥ c_0\E\abs{R''} \\
                &≥ c_0\sum_{\mathclap{\substack{\zeta\in S\\\zeta≥2\eps²}}}\E\abs{R_{\zeta B}} - c_0\Ek{\abs{R\ohne R'}} \\
                &= c_0\sum_{\mathclap{\substack{\zeta\in S\\\zeta≥2\eps²}}}\E\abs{R_{\zeta B}} + \Ok*{n^\beta B^{-\beta-0.2}} \\
                &= c_0n^βB^{-β} \sum_{\mathclap{\substack{\zeta\in S\\\zeta≥2\eps²}}} \varpi_\zeta\kl*{\ln \frac nB} + \Ok*{n^\beta B^{-\beta-0.2}}
                + o_\eps\kl*{n^\beta}
        \numberthis\label{eq:nn-lower-bound}
    \end{align*}
    by \eqref{eq:rs} and \Cref{lem:rzeta}, where $o_\eps$ again means that the convergence speed is depending on $\eps$. By \eqref{eq:rzeta-lower-bound}, the sum in \eqref{eq:nn-lower-bound} is bounded away from zero for $B$ large, so $\Ek{X_n} ≥ \Theta(n^\beta B^{-β}) + \Ok*{n^β B^{-β-0.2}} + o_\eps\kl*{n^\beta}$ and  the statement follows by choosing $B$ sufficiently large.
\end{proof}

\section{Contraction Method}\label{sec:contraction}
\begin{proof}[Proof of \Cref{thm:limit}]
We prove \Cref{thm:limit} with the contraction method in the form \cite{Wild_Nebel_Neininger_2015}, see also \cite{Roesler2001} and \cite{neininger_multivariate_2001} for a more general introduction. Define $X^{(i)}_n$ for $i=1,\dots, b$ as copies of $X_n$, independent of each other and the splitting $\V$ in the first step. Then $X_n$ fulfills the recursion
\begin{equation}
    \label{eq:xn-recursion}
    X_n \overset d= \sum_{i=1}^b X^{(i)}_{I_n^{(i)}} + \phi(T_n) \quad\text{for $n>s$},
\end{equation}
where $I_n^{(i)} = n_i$ is the number of balls that fall into the $i$-th subtree and $X^{(i)}_n$ are copies of $X_n$ that are independent of each other and of $I_n^{(i)}$. The precise distribution of $I_n^{(1)}, \dots, I_n^{(b)}$ is dependent on the specifics of the split tree, but is up to an error of $s$, the multinomial $(n;V_1,\dots, V_k)$ distribution given $\V$.

Rescaling \eqref{eq:xn-recursion} to $Y_n := (X_n-\Ek{X_n})/(\varpi(\ln n)n^{\beta})$, we obtain the recursion
\begin{align}
    Y_n &\overset d= \sum_{i=1}^b Y^{(i)}_{I_n^{(i)}}\fracc{I_n^{(i)}}{n}^\beta \frac{\varpi(\ln I_n^{(i)})}{\varpi(\ln n)}
    + \frac1{n^{\beta}\varpi(\ln n)}\kl*{\sum_{i=1}^b \Ek*{X_{I_n^{(i)}}\given I_n^{(i)}} - \Ek*{X_n} + \phi(T_n) -\Ek*{\phi(T_n)}}
    \nonumber \\&=: \sum_{i=1}^b Y^{(i)}_{I_n^{(i)}}\fracc{I_n^{(i)}}{n}^\beta\frac{\varpi(\ln I_n^{(i)})}{\varpi(\ln n)} + b_n
    \label{yn-recursion}
\end{align}
for $n>s$ and $Y^{(i)}_n$ are again independent copies of $Y_n$. In order to use the contraction method, we have to show
\begin{enumerate}[label= (\Alph*)]
    \item\label{cond:a} convergence $\fracc{I_n^{(i)}}{n}^\beta\frac{\varpi(\ln I_n^{(i)})}{\varpi(\ln n)} \overset{L_2}\to V_i^\beta $ and $b_n \overset{L_2}\to \kl*{\sum_{i=1}^b V_i^\beta}-1$ in $L_2$,
    \item\label{cond:b} that $\sum\limits_{i=1}^b \Ek*{V_i^{2\beta}} < 1$ for the limit $V_i^\beta$,
    \item\label{cond:c} and convergence $\sum\limits_{i=1}^b \Ek*{\fracc{I_n^{(i)}}{n}^{2\beta}\eins\mg{I_n^{(i)}≤ m \text{ or } I_n^{(i)}=n}} \to 0$ for all $m\in\NN$.
\end{enumerate}
Note that our Condition \ref{cond:c} is from \cite[Theorem 4.1]{neininger_multivariate_2001} instead of \cite{Wild_Nebel_Neininger_2015}, where $I_n^{(i)} < n$ is assumed.

For Condition \ref{cond:a}, note that we have almost sure convergence $I^{(i)}_n/n \to V_i$ due to the strong law of large numbers and that $I^{(i)}_n/n ≤ 1$ is bounded. Thus, we also have $L^2$ convergence
\begin{equation}\label{cm:v-convergence}
\fracc{I^{(i)}_n}{n}^β \to V_i^β.
\end{equation}

In the nonlattice case, $\varpi$ is just a constant, so $\varpi(\ln I^{(i)}_n)/\varpi(\ln n) = 1$. In the lattice case,
we apply the logarithm, and see that $\ln(I^{(i)}_n) - \ln(n) - \ln V_i$ converges to zero a.s. Since $\varpi$ is uniformly continuous, we know that \begin{equation}\label{cm:varpi-distance}
\varpi\kl*{\ln I^{(i)}_n} - \varpi\kl*{\ln(n) + \ln V_i} \longrightarrow 0 \quad \text{ a.s.}
\end{equation}
Because $\varpi$ is $d$-periodic and $\ln V_i \in d\ZZ$, we have $\varpi\kl*{\ln(n) + \ln V_i} = \varpi\kl*{\ln(n)}$ a.s.,
and we can use the fact that $\varpi$ is bounded away from zero to divide \eqref{cm:varpi-distance} by $\varpi(\ln n)$ to conclude that
\begin{equation}\label{cm:varpi-fraction}
    \frac{\varpi\kl[\big]{\ln I^{(i)}_n}}{\varpi\kl[\big]{\ln n}} \longrightarrow 1 \quad \text{ a.s.}
\end{equation}
Since $\varpi$ is also bounded, the convergence also holds in $L_2$. Equations \eqref{cm:v-convergence} and \eqref{cm:varpi-fraction} together show the first part of Condition \ref{cond:a}.

For $b_n$ note first that the term $\kl[\big]{\phi(T_n) - \Ek*{\phi(T_n)}}n^{-β}$ converges to zero in the $L^2$ norm because of the moment condition $\Ek{\phi(T_n)^2} = o(n^{2β})$. Applying \Cref{thm:mean} to \eqref{yn-recursion} yields
\begin{equation}
    b_n = \sum_{i=1}^b \frac{\varpi\kl[\big]{\ln I^{(i)}_n} \kl[\big]{ I^{(i)}_n}^β}{\varpi\kl[\big]{\ln n} n^β} - 1 + o(1)
    \longrightarrow \kl[\bigg]{\sum_{i=1}^b V_i^β} - 1
\end{equation}
in the $L_2$ norm using \eqref{cm:v-convergence} and \eqref{cm:varpi-fraction}.

Condition \ref{cond:b} follows from that $\sum_{i=1}^b \Ek*{V_i^{2\beta}} ≤ \sum_{i=1}^b\Ek*{V_i^{\beta}} =  1$ with equality if and only if $V_i\in (0,1)$ a.s.\ for all $i$. We have ruled this out in the definition of $\V$.

The value in the expectation of the technical Condition \ref{cond:c} is bounded by $\fracc mn^{2\beta} + \sum_{i=1}^b\Pk{I_n^{i} = n}$ and goes to zero if and only if $s_0+s_1>0$ or $\Pk{V=1}=0$. That leaves the special case $\Pk{V=1}>0$ and $s_0=s_1=0$, which we will show later in this section.

By the contraction method \cite[Section 4.1]{Wild_Nebel_Neininger_2015}, $Y_n$ converges in $\ell_2$ to the unique solution $Y$ among all centered random distributions with second moments of the recurrence
\begin{equation}
    Y \overset d= \sum_{i=1}^b V_i^\beta Y^{(i)} + \sum_{i=1}^b V_i^\beta - 1,
\end{equation}
where $Y^{(i)}$ are independent copies of $Y$. Adding 1 to both sides gives \eqref{y-recursion}:
\begin{equation}
    Y + 1 \overset d= \sum_{i=1}^b V_i^\beta (Y^{(i)} + 1).
\end{equation}
The variance of $Y$ can be calculated by conditioning \eqref{y-recursion} on $\V$:
\begin{align}
    \Var(Y) &= \Ek*{\Vark*{\sum_{i=1}^b V_i^\beta (Y^{(i)} + 1)\given \V}} + \Vark*{\Ek*{\sum_{i=1}^b V_i^\beta (Y^{(i)} + 1)\given \V}}
    \nonumber \\
    &= \Ek*{\sum_{i=1}^b V_i^{2\beta}}\Var(Y) + \Vark*{\sum_{i=1}^b V^\beta_i},
    \label{cm:var-y}
\end{align}
using independence of $Y^{(i)}$ and $\V$.
The expectation $\Ek*{\sum_{i=1}^b V_i^{2\beta}}$ is smaller than 1 because $\Pk*{V_{(1)} = 1} < 1$, and solving for $\Var(Y)$ gives \eqref{var-y}.

From \eqref{var-y} is it obvious that $\Var(Y) = 0$ if and only if $ \Vark*{\sum_{i=1}^b V^\beta_i} = 0$. Since we know the expectations of both variables, we can instead state that $Y=0$ a.s.\ if and only if $\sum_{i=1}^b V^\beta_i = 1$ a.s.

This completes the proof of \Cref{thm:limit} except in the special case
 $\Pk{V=1}>0$ and $s_0+s_1=0$ treated below.
\end{proof}

\begin{proof}[Proof of \Cref{thm:limit} if\/ $\Pk{V=1}>0$ and\/ $s_0+s_1=0$]
The case $\Pk{V=1}>0$ and $s_0=s_1=0$ can be resolved as follows: We condition on the first time where $V_{(1)} = \max_{i=1}^b V_i < 1$. Instead of \eqref{eq:xn-recursion}, we obtain the similar recursion
\begin{equation}\label{eq:x-recursion-alter}
    X_n \overset d= \sum_{i=1}^b X^{(i)}_{\widetilde I_n^{(i)}} + \widetilde\phi(T_n),
\end{equation}
where $\widetilde I_n^{(i)}$ are the number of balls that fall into the $i$-th subtree conditional on $V_{(1)}<1$, and $\widetilde\phi(T_n)$ is
\begin{equation}
    \widetilde\phi(T_n) := \sum_{v\in T_n} \eins(M_v^n = n)\phi(T_n^v).
\end{equation}
By the Minkowski inequality and the fact that $T_n^v$ is distributed as $T_n$ conditional on $M_v^n=n$,
\begin{align}
    \lnorm*2{\widetilde\phi(T_n)} &≤ \sum_{k=0}^\infty \lnorm[\bigg]2{\sum_{\substack{v\in T_n\\\abs{v}=k}} \eins(M_v^n=n)\phi(T_n^v)}
    \nonumber \\ &= \sum_{k=0}^\infty \lnorm[\bigg]2{\sum_{\substack{v\in T_n\\\abs{v}=k}} \eins(M_v^n=n)}\lnorm*2{\phi(T_n)}.
\end{align}
Since the probability of a node $v\in[b]^*$ having a child $i\in[b]$ with $M_{vi}^n = M_v^n$ is ${\Pk{V_{(1)}=1}} < 1$,
this yields
\begin{equation}\label{cm:widetilde-phi}
    \lnorm*2{\widetilde\phi(T_n)} ≤ \sum_{k=0}^\infty \Pk{V_{(1)}=1}^{k/2}\lnorm*2{\phi(T_n)} = \Ok{n^{β-δ}}.
\end{equation}
Define $(\widetilde V_1,\dots, \widetilde V_b)$ as $(V_1,\dots, V_b)$ conditioned on $V_{(1)}<1$. We can see that \eqref{eq:x-recursion-alter} is \eqref{eq:xn-recursion} with $(V_1,\dots, V_b)$ replaced by $(\widetilde V_1,\dots, \widetilde V_b)$ and $\phi$ replaced by $\widetilde\phi$.
Because of \eqref{cm:widetilde-phi}, 
and $\widetilde V_{(1)} < 1$ by the conditioning, 
we can therefore 
apply the previous case of the theorem,
which gives convergence of $\frac{X_n-\Ek{X_n}}{n^β\varpi(\ln n)}$ to a limit $\widetilde Y$ which fulfills the recursion 
\begin{equation}\label{y-recursion-alt}
    \widetilde Y+1 \overset d= \sum_{i=1}^b \widetilde V_i^β(\widetilde Y^{(i)} + 1),
\end{equation}
where $\widetilde Y^{(i)}$ are i.i.d.\ copies of $\widetilde Y$ independent
of $\widetilde V_1,\dots, \widetilde V_b$.
(Note that $\widetilde V_1,\dots, \widetilde V_b$ satisfy \eqref{def:beta}
with the same $\beta$ as $V_1,\dots,V_b$.)
However, since the recursion \eqref{y-recursion} reduces to just $Y + 1
\overset d= Y+1$ in the case of $V_{(1)} = 1$, one can see that the two
recursion equations \eqref{y-recursion}  and \eqref{y-recursion-alt} are equivalent and thus $Y \overset d= \widetilde Y$.
\end{proof}

\begin{proof}[Proof of \Cref{rem:multiple-convergence}]
We can repeat the same argument, replacing $\fracc{I_n^{(i)}}{n}^\beta\frac{\varpi(\ln I_n^{(i)})}{\varpi(\ln n)}$ with a diagonal matrix with $\fracc{I_n^{(i)}}{n}^\beta\frac{\varpi[\phi^{(i)}](\ln I_n^{(i)})}{\varpi[\phi^{(i)}](\ln n)}$ on the diagonal and $b_n$ with a vector, where $\varpi$ is replaced by $\varpi[\phi^{(i)}]$ for the $i$-th entry. Since the contribution of $\varpi$ cancels, we see that we get the vector with $Y$ on every entry as result. Alternatively, we can remark that the proof used in \cite{neininger_multivariate_2001} actually gives $L^2$ convergence to $Y$ and thus naturally extends to a multivariate version.
\end{proof}

\section{Calculation of the limit function \texorpdfstring{$\varpi$}{}}\label{sec:constants}

In this section we prove \Cref{thm:constants}. Again, we only care for the
mean and thus 
we can replace the possibly random $\phi(T_n)$ by $\phi_n$ and assume
that $\phi(T)$ depends only on $|T|$.
For a toll function $\phi_n$ with $\phi_n = \Ok*{n^{β-δ}}$ for some $δ>0$, we want to find the limit function $\varpi := \varpi[\phi]$ from \Cref{thm:mean}, which states that
\begin{equation}
    \Ek*{X_n(\phi)} = \varpi[\phi](\ln n)n^β + o(n^β).
\end{equation}

\subsection{Poissonization}\label{sec:poisson}

Consider for a Poisson variable $\poissonn \sim \operatorname{Poi}(λ)$ with $λ>0$ the split tree with $\poissonn$ many balls, which we call $\poisson T_λ := T_{\poissonn}$, where we assume that $\poissonn$ is independent of $T_n$ for all $n$.
From this tree, we can define a poissonized version of $X_n$, defined as $\poisson X_λ := X_{\poissonn}$. This has the same asymptotic mean as $X_n$:

\begin{lemma}\label{lem:poisson-mean}
    If $\phi_n = \Ok*{n^{β-δ}}$ for some $δ>0$, then \Cref{thm:mean} applies, and
    \begin{equation}
        \Ek*{\poisson X_{n}} = \Ek*{X_n} + o(n^β) = \varpi(\ln n)n^β + o(n^β).
    \end{equation}
\end{lemma}
\begin{proof}
    Condition $\poisson X_{n}$ on $\poissonn[n]$. Then, by \Cref{thm:mean},
    \begin{align}
        \Ek*{\poisson X_{n}} &= \Ek*{\poissonn[n]^β \varpi(\ln \poissonn[n]) + o(\poissonn[n]^β)}
        \nonumber \\&= \Ek{\poissonn[n]^β}\varpi(\ln n) + \Ek*{\poissonn[n]^β (\varpi(\ln \poissonn[n])-\varpi(\ln n))} + o\kl*{\Ek{\poissonn[n]}^β}
        \nonumber \\&= n^β\varpi(\ln n) + \Ok*{\Ek*{\abs{\poissonn[n]-n}^β}} + \Ek*{\poissonn[n]^β (\varpi(\ln \poissonn[n])-\varpi(\ln n))} + o(n^β)
        \nonumber \\&= n^β\varpi(\ln n) + \Ok*{\Ek*{\abs{\poissonn[n]-n}^2}^{β/2}} + \Ek*{\poissonn[n]^2}^{β/2} \Ek*{(\varpi(\ln \poissonn[n])-\varpi(\ln n))^2}^{1/2}
        \nonumber \\&\phantom= \qquad +o(n^β)
        \nonumber \\& = n^β \varpi(\ln n) + o(n^β)
        \nonumber \\& = \Ek*{X_n} + o(n^β).
    \end{align}
    We used that $\varpi$ is bounded and continuous and that $\poissonn[n] \sim \Poi(n)$ is concentrated around $n$; on the fourth line we used the Jensen and Cauchy--Schwarz inequalities. 
\end{proof}

Since $n = \poissonn$ is Poisson distributed, we have by a nice property of Poisson variables that if we condition on the splitter $\V$, then $n'_1,\dots, n'_b \sim \operatorname{Mult}(n;V_1,\dots, V_b)$ (see \eqref{def:ns}) are also Poisson distributed and independent, with parameters $λV_1,\dots, λV_b$.

In the rest of this subsection we assume that 
$s_0=s_1=0$. Then the random variables $n_1,\dots, n_b$ coincide with
$n'_1,\dots, n'_b$ for $n>s$, see \eqref{def:nv} and \eqref{def:ns}. 
We also assume that
\begin{equation}%\tag{A1} 
\label{cond:noleaves1}
    \phi_n = 0 \text{ for all } n≤s,
\end{equation}
allowing us to ignore the case $n≤s$, leading to the recursion
\begin{equation}
    \Ek*{\poisson X_λ} = \Ek*{\phi(\poisson T_λ)} + \sum_{i=1}^b \Ek*{\poisson X_{λV_i}},
\end{equation}
which we can alternatively write using the definition
\begin{equation}\label{def:fl}
    f(λ) := \mu\inv \Ek*{\phi(\poisson T_λ)} = \mu\inv \sum_{n=1}^\infty \phi_n e^{-λ}\frac{λ^n}{n!},
\end{equation}
as
\begin{equation}\label{const:xl-recursion}
    \Ek*{\poisson X_λ} = μf(λ) + b\Ek*{\poisson X_{λV}}.
\end{equation}
Once we define $Z(q) = \Ek*{\poisson X_{e^{q}}}$ and $G(q) =  μf(e^q)$, this becomes the renewal equation $Z = G + (b\P_{-\ln V} * Z)$. Note that $G$ and $Z$ are functions on $\RR$, not just on $[0,\infty)$. By the same technique as before in \Cref{lem:r}, we need to check that
\begin{equation}\label{const:g-bound}
    G(q)e^{-\beta q} = e^{-βq}μf(e^q)
\end{equation}
is continuous and directly Riemann integrable. 
Assume $\phi_n=\Ok*{n^{\beta-\delta}}$
with $0<\delta\le\beta$.
It is easy to see from
\eqref{def:fl} that $f$, when considered as a function on $\CC$, is well-defined and holomorphic at every $λ\in\CC$, and therefore continuous in particular.
For positive $q$, we know that there exist constants $C>0$ such that
\begin{equation}\label{const:g-bound-big}
    \Ek*{\phi(\poisson T_{e^{q}})} ≤ C\Ek*{\poissonn[e^{q}]^{β-δ}} 
≤ C\bigl(\Ek*{\poissonn[e^{q}]}\bigr)^{β-δ} = Ce^{q(β-δ)},
\end{equation}
and thus the term in \eqref{const:g-bound} is bounded by $Ce^{-qδ}$ for positive $q$. For negative $q$, we use Condition \eqref{cond:noleaves1} that $\phi_n = 0$ for $n≤s$ and therefore the terms with $n≤s$ in \eqref{def:fl} disappear. Hence, $f(λ) = \Ok*{λ^{s+1}}$ for $λ\in [0,1]$ and
\begin{equation}\label{const:g-bound-small}
    e^{-βq}μf(e^q) = \Ok*{e^{q(s+1-β)}}
\end{equation}
for negative $q$.

In the nonlattice case, the key renewal \Cref{thm:krt} implies 
\begin{align}
    \Ek*{\poisson X_{λ}} = Z(\ln λ) &= λ^β\frac{1}{µ}\int_{-\infty}^{\infty}e^{-βq}\mu f\kl{e^{q}}\mathrm dq + o(λ^β)
    \nonumber \\&= λ^β\int_{0}^{\infty}f(\nu)\nu^{-1-β}\mathrm d\nu + o(λ^β)
    \label{nonlattice-varpi}
\end{align}
with a change of variables. This integral is known as the Mellin transform
(see e.g.\ \cite{mellin} and \cite[Appendix B.7]{Flajolet_Sedgewick_2009}) of $f$, defined for $z\in \CC$ as
\begin{equation}\label{def:mellin}
    \mellin f(z) = \int_0^\infty f(x)x^{z-1} \mathrm dx
\end{equation}
if the integral is defined. Hence, by \Cref{lem:poisson-mean} and \eqref{nonlattice-varpi},
\begin{equation}\label{nonlattice-mellin}
    \varpi = \mellin f(-β).
\end{equation}
We will show later in \eqref{pf:mellin-dk} that $M_k = \mellin{f}$ if $\phi=\delta_k$, so that \eqref{constants:nonlattice} follows from \eqref{nonlattice-mellin}.

In the lattice case, \Cref{thm:krt} and \eqref{const:xl-recursion} yield
\begin{equation}
    \Ek*{\poisson X_{λ}} = Z(\ln λ) = λ^β\frac{d}{µ}\sum_{u\in\ZZ} e^{-β(du+\restd{\ln λ})}f\kl*{e^{du+\restd{\ln λ}}} + o(λ^β)
\end{equation}
and thus
\begin{equation}
    \varpi(x) = \frac{d}{µ}\sum_{u\in \ZZ} e^{-β(du+\restd{x})}f\kl*{e^{du+\restd{x}}}.
    \label{lattice-sum}
\end{equation}
% By shifting the summation index by $\floord{x}$, we see that we can replace $\restd x$ by $x$ in \eqref{lattice-sum}.
The $m$-th coefficient of the Fourier series of $\varpi$ is hence given by
\begin{align}
%    \MoveEqLeft 
\frac{1}{µ}\int_0^d  e^{-\frac{2\pi im}dx}\varpi(x) \mathrm dx
&=
\frac{1}{µ}\int_0^d \sum_{u\in \ZZ}  e^{-β(du+x)-\frac{2\pi im}dx}f\kl*{e^{du+x}} \mathrm dx
    \nonumber \\&= \frac{1}{µ}\sum_{u\in \ZZ}\int_{ud}^{(u+1)d}  e^{-βx-\frac{2\pi im}dx}f\kl*{e^{x}} \mathrm dx
    \nonumber \\&= \frac{1}{µ}\int_{-\infty}^\infty  e^{-βx-\frac{2\pi im}dx}f\kl*{e^{x}} \mathrm dx
    \nonumber \\&= \frac{1}{µ}\int_{0}^\infty  u^{-β-\frac{2\pi im}d-1}f\kl*{u} \mathrm du
    \nonumber \\&= \frac{1}{µ}\mellin f\kl*{-β-\frac{2\pi im}d}.
    \label{fourier-coeffs}
\end{align}
Under the assumption that $M_k = \mellin{f}$ if $\phi=\delta_k$
(shown below), 
this would show \eqref{constants:lattice}.

\subsection{Constants for fringe trees when \texorpdfstring{$s_0+s_1=0$}{s0+s1=0}}\label{sec:fringe}

The restriction that $\phi$ has no mass on leaves can be lifted in two ways. 
First, note that we can always move mass `up the tree' by defining
\begin{equation}\label{def:parent-phi}
    \phi''(T) = \sum_{i=1}^b \phi(T^i).
\end{equation}
The new toll function $\phi''$ has $\phi''_n = 0$ for $n≤s$ and the difference between $X(\phi)$ and $X(\phi'')$ is then just $\phi(T_n)$, which is just of order $n^{β-δ}$. This works for all toll functions, but $\varpi[\phi'']$ is harder to calculate because of the additional sum in \eqref{def:parent-phi}, cf.\ \cite[Example~3.17]{janson:trie}.

The second way is to consider the split tree $T'_n$ with the same splitter $\V$ and $s=s_0=s_1=0$. The equivalent of $n_v$ in the tree $T'_n$ is $n'_v$ as defined in \eqref{def:ns}, as can be seen from the fact that \eqref{def:ns} and \eqref{def:nv} coincide if $s=s_0=s_1=0$. 
Note that if $β=1$, this tree $T'_n$ has infinitely many nodes with exactly 1 ball, so we have to make sure that either $β<1$ or any toll function $\phi'$ we use on $T'_n$ has $\phi'_1 = 0$.
 (Note that elsewhere in the paper we assume $s\ge1$, to avoid this problem.)

Just like in \Cref{sec:other-toll}, $T_n$ is now the ``bucket'' version of $T'_n$ and can be obtained from $T'_n$ by removing all children of nodes $v$ with $n'_v≤s$. We denote this operation by $\bucket s{\cdot}$, so that $T_n = \bucket s{T'_n}$.

If a node has $s$ or fewer balls, we have to subtract the toll function of its children, so we define $\phi'$ as
\begin{equation}\label{def:leaves-phi}
    \phi'(T) = \phi(\bucket s{T}) - \eins(n_\eps≤s)\sum_{i=1}^b \phi(\bucket s{T^i}),
\end{equation}
where $n_\eps$ is the number of balls in the tree $T$. If either $\phi_1=0$ or $\beta<1$, the sum defining $X_n(\phi')$ on $T_n'$ has only finitely many nonzero terms and is equal to $X_n(\phi)$ on $T_n$, because
\begin{align}
    \sum_{v\in T'_n} \phi'((T'_n)^v) &= \sum_{v\in T'_n} \kl*{\phi(\bucket s{(T'_n)^v}) - \eins(n'_v≤s) \sum_{i=1}^b \phi\kl*{\bucket s{(T'_n)^{vi}}}}
    \nonumber \\ &
    = \sum_{v\in T'_n} \phi\kl*{\bucket s{(T'_n)^v}} -
    \sum_{\mathclap{\substack{\eps \ne v\in T'_n \\ n'_{\operatorname{parent}(v)}≤s}}} \phi\kl*{\bucket s{(T'_n)^{v}}}
    \label{pf:two-sum-row} \\ &
    = \phi\kl*{\bucket s{T'_n}} + \sum_{\mathclap{\substack{\eps \ne v\in T'_n \\ n'_{\operatorname{parent}(v)}>s}}} \phi\kl*{\bucket s{(T'_n)^v}}
    \nonumber \\ &
    = \sum_{v\in T_n} \phi(T_n^v).
    \label{pf:bucket-back}
\end{align}
In the last line \eqref{pf:bucket-back} we used the fact the fringe trees $\mg{T_n^v\::\:v\in T_n}$ are exactly given by $\bucket s{(T'_n)^v}$ for nodes $v$ that don't have a parent with less than $s$ nodes.

The calculation above fails if $\beta=1$ and $\phi_1>0$ because the two sums in \eqref{pf:two-sum-row} become infinite. If $\beta=1$ and the toll function is $\fringephi[1]$ defined in \eqref{def:phi-k}, something interesting happens instead: Since a node with one ball has exactly one child with one ball for $s=0$, we have $\fringephi[1]'(T'_1) = 0$. Thus, writing $n'_v$ for the equivalent of $n_v$ for $v \in T'_n$,
\begin{align}
    \sum_{v\in T'_n} \fringephi[1]'((T'_n)^v) &= \sum_{v\in T'_n} \eins(n'_v=1) - \eins(n'_v≤s) \sum_{i=1}^b \eins(n'_{vi} = 1)
    \nonumber \\ &
    = -\sum_{v\in T'_n}\eins(1<n'_v≤s)\sum_{i=1}^b \eins(n'_{vi} = 1).
    % \nonumber \\ &
    % = -\sum_{\mathclap{\eps \ne v\in T'_n}} \eins(n_v'=1, 1 < n'_{\operatorname{parent}(v)}≤s).
\end{align}
Hence $-\fringephi[1]'$ is counting the number of nodes $v$ with $n'_v=1$ whose parents have between $2$ and $s$ nodes. For $\beta=1$, every ball in $T_n'$ will eventually end up in an infinite chain of nodes with $n'_v=1$. The ball will form a single-ball leaf in $T_n$ if and only if the last node before the chain has more than $s$ balls, so that it is an internal node. As $-\fringephi[1]'$ is counting the number of balls that don't form their own single-ball leaves, we obtain
\begin{equation}\label{pf:leaf-correction}
    X_n(\fringephi[1]) = n + \sum_{v\in T_n'} \fringephi[1]'((T'_n)^v).
\end{equation}
This correction term of $n$ can alternatively be seen by writing the infinite sums in \eqref{pf:two-sum-row} as limits over the depth $\abs{v}$ of nodes, and seeing that the remainder term
\begin{equation}
    -\lim_{m\to\infty} \sum_{\abs v=m}\eins(n'_{\operatorname{parent}(v)}≤s, n_v'=1)
\end{equation} converges to $-n$.

% Compare this with $\fringephi[1]''$ as defined in \eqref{def:parent-phi},
% \begin{align}
%     X_n(\fringephi[1])
%     &= \sum_{v\in T_n} \fringephi[1]''(T_n) 
%     = \sum_{v\in T_n} \eins(n_v>s) \sum_{i=1}^b \eins(n_{vi}=1)
%     \nonumber \\ &= \sum_{v\in T_n'} \eins(n'_v>s) \sum_{i=1}^b \eins(n'_{vi}=1)
%     \nonumber \\ &= \sum_{v\in T_n'} \eins(n'_v>1) \sum_{i=1}^b \eins(n'_{vi}=1) + \sum_{v\in T_n'} \fringephi[1]'((T'_n)^v)
%     \nonumber \\ &= n + \sum_{v\in T'_n} \fringephi[1]'((T'_n)^v)
% \end{align}

\begin{proof}[Proof of \Cref{thm:constants}]
Remember $\fringephi$ from \eqref{def:phi-k} as the toll function counting fringe trees with exactly $k$ balls. 
For $k>s$, we can directly use the calculations in \Cref{sec:poisson},
noting that for $\phi=\fringephi$, the function $f$ defined 
in \eqref{def:fl} is given by
\begin{equation}
    f_k(\nu) := µ\inv\Ek*{\delta_k(\poisson T_{\nu})} = \mu\inv\Pk*{\poissonn[\nu] = k} = \mu\inv e^{-\nu} \frac{\nu^k}{k!}.
\end{equation}
Its Mellin transform (see \eqref{def:mellin}) is given by,
for $\Re z>-k$,
\begin{equation}
    \mellin{f_k}(z) = \frac{1}{\mu} \int_0^\infty e^{-\nu}\frac{\nu^{k-1+z}}{k!}\mathrm d\nu
    = \frac{\Gamma(k+z)}{\mu k!} = M_k(z),
\end{equation}
with $M_k(z)$ from \eqref{constants:fk}.
Together with \eqref{nonlattice-mellin} and \eqref{fourier-coeffs}, this shows \eqref{constants:nonlattice} and \eqref{constants:lattice} for $k>s$.

For $k≤s$ and $k>β$, we use \eqref{def:leaves-phi} to instead consider the toll function $\delta_k'$ on the tree $T'_n$ that has $s=0$.
As shown in \eqref{pf:bucket-back}, this has the same expectation and we can use \Cref{sec:poisson}. Note that that section also works for $s=0$, except for \eqref{const:g-bound-small}, 
where we (for $f_k$) replace $s$ by $k-1$.
The function $f$ from before \eqref{def:mellin} is for $\delta_k'$ given as
\begin{equation}
    f_k(λ) := \frac1{µ}\Ek*{\fringephi'(T'_{\poissonn})} = \frac1{µ}e^{-λ}\frac{λ^k}{k!} - \frac b{µ} \sum_{j=k}^s e^{-λ}\frac{λ^j}{k!(j-k)!}\Ek*{V^k(1-V)^{j-k}}.
\end{equation}
Therefore,
\begin{align}
    \mellin{f_k}(z) &= \frac{1}{\mu} \int_0^\infty
\biggl(
        e^{-\nu}\frac{\nu^{k-1+z}}{k!}
        - \frac b\mu \sum_{j=k}^s e^{-\nu}\frac{\nu^{j-1+z}}{k!(j-k)!}\Ek*{V^k(1-V)^{j-k}}
\biggr)    
\mathrm d\nu
    \nonumber \\&
    = \frac{\Gamma(k+z)}{\mu k!} - \frac{b}{\mu}\sum_{j=k}^s \Ek*{V^k(1-V)^{j-k}}\frac{\Gamma(j+z)}{k!(j-k)!},
    \label{pf:mellin-dk}
\end{align}
or with a change of summation index,
\begin{equation}\label{varpi-small-fringe}
    \mellin{f_k}(z) = \frac{\Gamma(k+z)}{\mu k!} - \frac{b}{\mu}\sum_{j=0}^{s-k} \Ek*{V^k(1-V)^{j}}\frac{\Gamma(k+j+z)}{k!j!}
    = M_k(z).
\end{equation}
Again, this shows \eqref{constants:nonlattice} and \eqref{constants:lattice} for $k>β$ using \eqref{nonlattice-mellin} and \eqref{fourier-coeffs}.

The remaining case for $M_k$ is $k=β=1$. By \eqref{pf:leaf-correction}, we know that we have to correct \eqref{varpi-small-fringe} by a fixed value $n$, which gives 
an additional constant 1 for $\varpi$. 
Note that the first term in \eqref{varpi-small-fringe} and the sum term for $j=0$ cancel, so
in the nonlattice case
\begin{equation}
    \varpi\br{\fringephi[1]} = 1 - \frac{b}{\mu}\sum_{j=1}^{s-1} \Ek*{V(1-V)^{j}}\frac{\Gamma(1+j-β)}{j!} = M_1(-1),
\end{equation}
as in \eqref{constants:f1}. This is also the 0-th Fourier coefficient in the lattice case.

Because the function $\Gamma(a+bi)$ goes to 0 exponentially fast for $a\in\RR$ fixed and $b \to \pm\infty$,
the Fourier coefficients go to 0 exponentially fast and thus $\varpi[\fringephi]$ is smooth for every $k$.
% Note that this is also the limit of \eqref{varpi-small-fringe} for $z=-1$ as $k\to 1$ because $-b\dd k\Ek{V^k} = -b\Ek{V\ln V} = \Ek{-\ln \widehat V} = \mu$. The fixed factor of $n$ does not influence the other Fourier coefficients, showing \eqref{constants:lattice} also for this special case.
% The constant for the number of leaves is now just the sum of \eqref{varpi-small-fringe} over $k=1,\dots, s$ and the number of balls remaining in the tree can be calculated in the same way.

% The only thing left to show for \Cref{thm:constants} is \eqref{constants:general}, which in the nonlattice case is
% \begin{equation}
%     \varpi\br\phi = \mellin*{\sum_{k=1}^\infty \phi_k\fringephi}(-β) = \sum_{k=1}^\infty\phi_k \mellin{\fringephi}(-β) = \sum_{k=1}^\infty\phi_k \varpi\br{\fringephi}.
% \end{equation}
% The sum can be exchanged with the Mellin transform because we know by \eqref{nonlattice-mellin} that the sums/integrals converge absolutely. The same can be done for the lattice case using
% \eqref{lattice-sum}.
\end{proof}

\subsection{Approximation if \texorpdfstring{$s_0+s_1>0$}{s0+s1>0}}\label{sec:approximation}

Even if $s_0+s_1>0$, there is a simple way to approximate $\varpi$: The recursion
\begin{equation}
    \Ek{X_n} = \phi_n + \sum_{i=1}^b \Ek{X_{n_i}} \quad \text{for }n>s
\end{equation}
can be rewritten to
\begin{equation}
    \Ek{X_n} = \phi_n + \sum_{i=1}^b \Ek{X_{n_i}-X_{n_i'}}  + \sum_{i=1}^b \Ek{X_{n_i'}}, \quad \text{for }n>s
\end{equation}
using \eqref{def:nv} and \eqref{def:ns}, which, defining
\begin{equation}\label{def:s0phi}
    \widehat \phi_n := \phi_n + \sum_{i=1}^b \Ek{X_{n_i}-X_{n_i'}}
\end{equation}
is the recursion of $X_n(\widehat\phi)$ for a split tree $T'_n$ with the same splitter $\V$ and $s$, but with $s_0=s_1=0$. So, if $\widehat\phi_n = \Ok*{n^{β-δ}}$, we can apply \Cref{thm:constants} and \Cref{prop:sum} and show that
\begin{equation}\label{approx:nonlattice}
    \varpi[\phi] = \sum_{n=1}^\infty M_n(-β)\widehat \phi_n
\end{equation}
in the nonlattice case and
\begin{equation}\label{approx:lattice}
    \varpi[\phi](x) = \sum_{n=1}^\infty \sum_{m=1}^\infty M_n\kl*{-β - \frac{2\pi imx}d} \widehat \phi_n e^{2\pi im x/d}
\end{equation}
in the lattice case, with the functions $M_n$ defined in \eqref{constants:fk} and \eqref{constants:f1}.

This is not very useful to determine $\varpi\br\phi$ accurately because having a formula for $\widehat\phi_n$ is essentially the same as knowing $\Ek{X_n}$. This is especially obvious when $s_1=0$, where
\begin{equation}
    \widehat\phi_n = \phi_{n+s_0} + \Ek{X_n} - \Ek*{X_{n+s_0}}.
\end{equation}
However, $\widehat\phi_n$ can usually be very easily bounded, e.g., for $\delta_k$ or the number of nodes it is known that adding a node changes $X_n(\phi)$ by at most 1. Therefore $\abs{\widehat \phi_n - \phi_n} ≤ \sum_{i=1}^b \Ek{\abs{n_i-n_i'}} ≤ s_0+2bs_1$.
Then, one can calculate $\widehat \phi_n$ explicitly for small $n≤N$, e.g.\ by listing all possible trees for $T_n$, and estimate \eqref{approx:nonlattice} in the nonlattice case by
\begin{equation}\label{approx:approximation}
   \varpi[\phi] = \sum_{n=1}^N M_n(-β)\widehat \phi_n + \sum_{n=N+1}^\infty M_n(-β)\phi_n \pm (s_0+2bs_1)\sum_{n=N+1}^\infty M_n(-β),
\end{equation}
where by $a = b \pm c$ we mean that $\abs{a-b}≤c$. The last two terms can be explicitly calculated, and the last term is of order $\Ok*{N^{-β}}$.
For nonlattice split trees ($β=1$) and $N>s$, the last term is
exactly $\frac{s_0+2bs_1}{N\mu}$, see \Cref{rem:split-fringe}. So in order to approximate $\varpi[\phi]$ up to an error of $0.01$, you would have to simulate $\widehat \phi_n$ up to $n=100(s_0+2bs_1)/\mu$.

For general $\phi_n$, we can at least say that $\widehat \phi_n$ is of the same order:
\begin{lemma}\label{lem:approx}
    Let $\phi$ be a toll function such that $\phi_n-\phi_{n-1} = \Ok{n^{β-δ}}$ for some $δ\in(0,1)$ and assume $s_0+s_1>0$. Then
    \begin{equation}\label{deltan-0}
        \Ek*{X_n(\phi)-X_{n-1}(\phi)} = \Ok*{n^{β-δ}}
    \end{equation}
    and thus also $\widehat\phi_n = \phi_n + \Ok*{n^{β-δ}}$.
\end{lemma}
\begin{proof}
    The difference $X_n(\phi)-X_{n-1}(\phi)$ is the effect of adding the $n$-th ball. If $n-1>s$, then this ball goes to the $i$-th child of the root with probability $V_i$. We hence have the recursion
    \begin{equation}\label{deltan-1}
        \Delta(n) := \abs[\big]{\Ek*{X_n(\phi)-X_{n-1}(\phi)}} ≤ \abs{\phi_n-\phi_{n-1}}
        + \sum_{i=1}^b \Ek*{\Delta(n_i)V_i} 
    \end{equation}
    for $n≥s+2$. This recursion is amenable to the contraction method: Scaling by $n^{β-δ}$,
    \begin{align}
        \frac{\Delta(n)}{n^{β-δ}} &≤ \sum_{i=1}^b \Ek*{\frac{\Delta(n_i)}{n_i^{β-δ}}\fracc{n_i}n^{β-δ}V_i} + \Ok1 
        \nonumber \\&≤ \Ok1 + \sup_{m<n}\frac{\Delta(m)}{m^{β-δ}} \sum_{i=1}^b \Ek*{\fracc{n_i}n^{β-δ}V_i}
        \nonumber \\&= \Ok1 + \sup_{m<n}\frac{\Delta(m)}{m^{β-δ}} \kl*{ \sum_{i=1}^b \Ek*{V_i^{1+β-δ}} + o(1)},
        \label{deltan-bounded}
    \end{align}
    where we used in the last line that $n_i/n\to V_i$ a.s.\ 
as $n\to\infty$
and that the terms in the expectation are bounded by 1. The supremum does not have to include $m=n$ because $n_i<n$ in the case of $s_0+s_1>0$. 
Since $\delta<1$ we
have that $\sum_{i=1}^b \Ek*{V_i^{1+β-δ}} < 1$ by \eqref{def:beta} and
because $\Pk*{V\in\mg{0,1}}<1$. Hence, \eqref{deltan-bounded} is a
contraction for large $n$, and it follows that $\Delta(n) = \Ok*{n^{β-δ}}$,
i.e., \eqref{deltan-0}. 

The statement on $\widehat \phi_n$ is a immediate consequence of
\eqref{deltan-0} in view of
\eqref{def:s0phi} and $\abs{n_i-n'_i}\le s_0+bs_1$.
\end{proof}
Since we can choose $δ$ arbitrarily close to 1 for the toll functions $\fringephi$ and $1$, we can improve the error bound in \eqref{approx:approximation} to $\Ok*{N^{-1+\epsilon}}$ for any $\epsilon>0$.
\begin{remark}
    \Cref{lem:approx} also holds for $s_0=s_1=0$, but we would have to include an extra argument for the case that $n_i=n$ for some $i\in[b]^*$. The case $s_0+s_1>0$ is sufficient for our purposes, since we only need the lemma to show that $\widehat\phi_n = \phi_n + \Ok*{n^{β-δ}}$, which is trivially true for $s_0=s_1=0$.
\end{remark}

\subsection{Hölder continuity of \texorpdfstring{$\varpi$}{the limit function}}
In this section we show how in the lattice case the order of $\phi_n$ and Hölder continuity of $\varpi[\phi]$ are linked. We will therefore assume for the whole section that we are in the lattice case.

\begin{proof}[Proof of \Cref{prop:holder}]
W.l.o.g.\ we can assume that $\phi_1=0$, since $\varpi\br{\delta_1}$ is
smooth. Furthermore, we can assume that $s_0+s_1=0$, since in the case of
$s_0+s_1>0$, we can use \Cref{lem:approx} 
and \eqref{approx:lattice}
to reduce the problem to $\widehat \phi$ on a tree with $s_0+s_1=0$.

Recall from \Cref{sec:poisson} the 
function $f$ defined in \eqref{def:fl}, and the
equation \eqref{lattice-sum},
which we write as
\begin{equation}
    \varpi = \frac{d}{µ}\sum_{u\in \ZZ} e^{-β(du+{x})}f\kl*{e^{du+ x}},
\end{equation}
since we can replace $\restd{x}$ by $x$ by shifting the summation index. By
using the definition of $f$ in \eqref{def:fl}
as a sum, we can extend $f$ to complex arguments; note that the sum
converges absolutely for all $\lambda$ and thus defines $f$ as an entire
function.

We will now separately treat the terms for $u≤0$ and for $u>0$.
Because $\phi_1=0$, we see from \eqref{def:fl} that $f(λ) = \Ok*{\abs{λ}^2}$
for $\abs{λ}<1$. 
Hence, for any $R>0$ there exist a constant $C_1$ such that if $u≤0$ and $z\in \CC$ with $\Re z≤R$, then
\begin{equation}
    \abs*{e^{-βdu}f\kl*{e^{du+z}}} ≤ C_1e^{-βdu}\abs*{e^{du+z}}^2 ≤\ C_1e^{(2-β)du + 2R}.
\end{equation}
Consequently,
\begin{equation}\label{hö:lower-terms}
    F_1(z) := \sum_{u≤0}e^{-βdu} f(e^{du+z})
\end{equation}
converges uniformly on compact sets in $\CC$, and is thus an entire analytic
function \cite[Theorem 10.27]{Rudin}.
In particular, $e^{-βx}F_1(x)$ is smooth on the real line.

For the sum with positive $u$, we note that by \eqref{const:g-bound-big} we have for real $λ≥1$ that
\begin{equation}\label{hö:f-bound-real}
    \abs{f(λ)} = \Ok{λ^{β-δ}}.
\end{equation}
More generally, for complex $λ\in\CC$ with $\abs{λ}≥1$, we have by \eqref{def:fl} together with \eqref{hö:f-bound-real} applied to the toll function $\abs\phi$,
\begin{align}
    \abs{f(λ)} &≤ \frac1{μ}\sum_{n=1}^\infty \abs{\phi_n}\frac{\abs{λ}^n}{n!} e^{-\Re λ}
    \nonumber \\ &= e^{\abs{λ}-\Re{λ}}\frac{1}{μ}\sum_{n=1}^\infty \abs{\phi_n}\frac{\abs{λ}^n}{n!} e^{-\abs{λ}}
    \nonumber \\ &= \Ok*{e^{\abs{λ}-\Re{λ}}\abs{λ}^{β-δ}}.
    \label{hö:f-bound-comp}
\end{align}
Furthermore, if $\Re λ>0$,
\begin{equation}
    \abs{λ}-\Re λ = \frac{\abs{λ}^2-(\Re λ)²}{\abs{λ}+\Re λ} ≤ \frac{(\Img{λ})^2}{2\Re(λ)}.
\end{equation}
Hence, if $x≥1$ is real and $\abs{λ-x} ≤ \frac12\sqrt x$, then $\abs{λ}-\Re λ ≤ \frac{x}{8(x-\frac12\sqrt x)}≤1$,
and \eqref{hö:f-bound-comp} yields 
\begin{equation}
    \abs{f(λ)} = \Ok*{\abs{λ}^{β-δ}} = \Ok*{x^{β-δ}} \qquad \text{for }\abs{x-λ} ≤ \frac12\sqrt x.
\end{equation}
Consequently, by Cauchy's estimate \cite[Theorem 10.25]{Rudin},
\begin{equation}\label{hö:fs-bound}
    \abs{f'(x)} = \Ok*{x^{β-δ-\frac12}} \qquad \text{for }x≥1.
\end{equation}

Now let
\begin{equation}
    g(q) := e^{-βq}\mu f(e^q).
\end{equation}
Then, \eqref{hö:f-bound-comp} and \eqref{hö:fs-bound} yield, for $q≥0$,
\begin{align}
    \label{hö:g-bound}
    g(q) &= \Ok{e^{-δq}}, \\
    \label{hö:gs-bound}
    g'(q) &= -βg(q) + e^{(1-β)q}μf'(e^q) = \Ok*{e^{(\frac12-δ)q}}.
\end{align}
For $0≤\alpha≤1$ and an interval $J$, let $\mathcal H_α(J)$ denote the Hölder space of functions $h$ on $J$ with finite Hölder norm
\begin{equation}
    \label{hö:hölder-norm}
    \norm{h}_{\mathcal H_α(J)} := \sup_{x\ne y\in J}\frac{\abs{h(x)-h(y)}}{\abs{x-y}^\alpha}.
\end{equation}
It is easy to see that
\begin{equation}
    \norm{h}_{\mathcal H_α(J)} ≤ 2\kl*{\sup_{x\in J} \abs{h(x)}}^{1-α}\kl*{\sup_{x\in J} \abs{h'(x)}}^α
\end{equation}
and thus \eqref{hö:g-bound} and \eqref{hö:gs-bound} imply for $u\ge0$ that
\begin{equation}
    \norm{g}_{\mathcal H_α([ud,(u+1)d])} = \Ok*{e^{\kl{\frac{\alpha}2-\delta}du}}
\end{equation}
Consequently, if $\alpha<2δ$, we have
\begin{equation}
    \norm*{\sum_{u=0}^\infty g(x+ud)}_{\mathrlap{\mathcal H_α([0,d])}}\qquad ≤ \sum_{u=0}^\infty \norm*{g(x)}_{\mathcal H_α([ud,(u+1)d])} < \infty.
\end{equation}
This, together with the fact that $F_1$ from \eqref{hö:lower-terms} is analytic, shows that $\varpi \in \mathcal H_α([0,d])$. Hence, the Fourier series of $\varpi$ converges uniformly by the Dini--Lipschitz test, see e.g.\ \cite[Theorems II.10.3 and II.10.8]{Zygmund}. Furthermore, if $\delta>\frac14$, then $\varpi \in \mathcal H_α([0,d])$ for some $α>\frac12$, and thus the Fourier series converges absolutely by Bernstein's theorem \cite[Theorem VI.3.1]{Zygmund}.
If $\delta>\frac12$, then \eqref{hö:gs-bound} similarly implies that $\varpi$ is
continuously differentiable.
\end{proof}

\begin{proof}[Proof of \Cref{prop:holder-ex}]
Choose $\phi$ such that $X_n(\phi)$ is exactly $n^β\varpi(\ln n)$; by
\eqref{eq:xn-recursion}, this is achieved by taking
\begin{equation}
    \phi(T_n) := n^β\varpi(\ln n) - \sum_{i=1}^{b} {n_i^β\varpi(\ln n_i)}
\end{equation}
and thus
\begin{equation}
    \phi_n = n^β\varpi(\ln n) - \sum_{i=1}^{b} \Ek*{n_i^β\varpi(\ln n_i)}.
\end{equation}
Using that $\varpi$ is $d$-periodic and \eqref{def:beta},
\begin{align}
    \phi_n &= \sum_{i=1}^b \Ek*{n_i^β\kl{\varpi(\ln n)-\varpi(\ln n_i)}} + \varpi(\ln n)\Ek*{n^β - \sum_{i=1}^b n_i^β}
    \nonumber \\ &= \sum_{i=1}^b \Ek*{n_i^β\kl{\varpi(\ln(nV_i))-\varpi(\ln n_i)}} + \varpi(\ln n)\sum_{i=1}^b \Ek*{(nV_i)^β-n_i^β}.
    \label{hc:phi-sum}
\end{align}

We first show that the second summand in \eqref{hc:phi-sum} is $\Ok1$. From the Taylor expansion of $x^β$ one sees that
\begin{equation}
    (nV_i)^β-n_i^β = β(nV_i-n_i)(nV_i)^{β-1} + \Ok*{(nV_i-n_i)^2(nV_i)^{β-2}},
\end{equation}
so, using $\Ek*{n_i\given V_i} = nV_i + \Ok 1$ and $\Ek*{(nV_i-n_i)^2\given V_i} = \Ok*{nV_i + 1}$
from \eqref{def:nss} and \eqref{def:ns}, we have on $nV_i≥1$
\begin{align}
    \Ek*{(nV_i)^β-n_i^β \given V_i} &= β\Ek*{nV_i-n_i\given V_i}(nV_i)^{β-1} + \Ok*{\Ek*{(nV_i-n_i)^2\given V_i}(nV_i)^{β-2}}
    \nonumber \\ &= \Ok1 + \Ok*{(nV_i)^{β-1}}.
    \label{hc:second}
\end{align}
If $nV_i≤1$, the left-hand side of \eqref{hc:second} is also $\Ok1$, showing that the second term in \eqref{hc:phi-sum} is indeed $\Ok1$.

We want to show that the first summand in \eqref{hc:phi-sum} is of order $n^{β-δ}$.
We condition on $V_1$, and note
first  that we can assume $nV_i≥1$, since otherwise the term would trivially
be $\Ok1$.
By the assumed Hölder property of $\varpi$,
\begin{equation}\label{hc:hused}
    \abs*{\varpi(\ln(nV_i)) - \varpi(\ln n_i)} = \Ok*{\abs{\ln(nV_i) - \ln n_i}^{2δ}\minv 1},
\end{equation}
so if $n_i≥e\inv nV_i$, we can use the mean value theorem to show that
\begin{equation}\label{hc:var}
    \abs*{\varpi(\ln(nV_i)) - \varpi(\ln n_i)} = \Ok*{\frac{\abs{n_i-nV_i}^{2δ}}{(nV_i)^{2δ}}}.
\end{equation}
If $n_i<e\inv nV_i$, then the right-hand side of \eqref{hc:hused} is bounded 
and the fraction in the right-hand side of \eqref{hc:var} is bigger than a
constant, so 
\eqref{hc:var} actually holds for all $n_i$ and $V_i≥\frac1n$. % Having V_i=0 would be confusing.

Therefore, using $δ≤β$,
\begin{equation}
    n_i^β\kl{\varpi(\ln(nV_i))-\varpi(\ln n_i)}
    = \Ok*{n_i^β\frac{\abs{n_i-nV_i}^{2δ}}{(nV_i)^{2δ}}}
    = \Ok*{n^{β-δ}\fracc{n_i\kl{n_i-nV_i}^{2}}{(nV_i)^{2}}^δ}.
    \label{hc:full}
\end{equation}
It hence suffices to show that $\Ek*{n_i(n_i-nV_i)^2\given V_i} =
\Ok*{(nV_i)²}$ when $nV_1\ge1$.
To that end, let $\nnn:=n-s_0-bs_1$ and note that we can write $n_i$
conditional on $V_i$ as $Z+s_1$ where
$Z\sim \Bin(\nnn,V_1)$;
hence $n_i-nV_i=Z-\nnn V_i + W$, where $W:=s_1-(s_0+bs_1)V_i=\Ok1$
is non-random given $V_1$.
Consequently, by standard calculations and recalling $nV_i\ge1$,
\begin{align}
  \E[(n_i-nV_i)^2\mid V_i]
= \E[(Z-\nnn V_i)^2\mid V_i]+\Ok1
=\nnn V_i(1-V_i)+\Ok1 
=\Ok{nV_i}
\end{align}
and
\begin{align}
  \E[(n_i-nV_i)^3\mid V_i]&
= \E[(Z-\nnn V_i)^3\mid V_i]+\Ok{\E[(Z-\nnn V_i)^2\mid V_i]}+\Ok1
\notag\\&
=\nnn \bigl(V_i(1-V_i)(1-2V_i)\bigr)+\Ok{nV_i}
=\Ok{nV_i}.
\end{align}
Consequently,
\begin{align}
  \E[n_i(n_i-nV_i)^2\mid V_i]
=   \E[(n_i-nV_i)^3\mid V_i]+
nV_i\E[(n_i-nV_i)^2\mid V_i]
=\Ok{(nV_i)^2},
\end{align}
as required. 
%the sum of $n-s_0-bs_1$ independent Bernoulli($V_i$) random variables $B_1,\dots, B_{n-s_0-bs_1}$ plus $s_1$:
%\begin{equation}
%    n_i(n_i-nV_i)^2 = \kl*{s_1 + \sum_{j=1}^{\mathclap{n-s_0-bs_1}}B_j}\kl*{s_1-V_i(s_0+bs_1) + \sum_{j=1}^{\mathclap{n-s_0-bs_1}}(B_j-V_i)}^2
%\end{equation}
%Writing this as a sum, we see that terms that don't choose the same $j$ in the second sum vanish in expectation because of the centering; and thus the dominating terms are of the form $\Ek*{B_k(B_j-V_i)^2\given V_i}$ for $j\ne k$. Hence,
%\begin{equation}
%    \Ek*{n_i(n_i-nV_i)^2\given V_i} = V_i²(1-V_i)n^2 + \Ok{nV_i} = \Ok*{(nV_i)^2}
%\end{equation}
Hence,
by Jensen's inequality, \eqref{hc:full} implies
\begin{equation}
  \Ek*{ n_i^β\kl{\varpi(\ln(nV_i))-\varpi(\ln n_i)}\mid V_i}  = \Ok*{n^{β-δ}},
\end{equation}
provided $nV_i\ge1$, and as said above this holds trivially also for $nV_i<1$.
Consequently, by \eqref{hc:phi-sum} and \eqref{hc:second},
\begin{equation}
    \phi_n = \sum_{i=1}^b\Ek*{ n_i^β\kl{\varpi(\ln(nV_i))-\varpi(\ln n_i)}} + \Ok1 = \Ok*{n^{β-δ}},
\end{equation}
which finishes the proof.
 
\end{proof}

\subsection{Integral for \texorpdfstring{$s_0+s_1>0$}{s0+s1>0}}\label{sec:integral}

An alternative approach would be to generalize the Poisson model to $s_0+s_1>0$. In this section we assume that we are in the nonlattice case and assume for simplicity that one of the two following holds:
\begin{enumerate}[label=(A\arabic*)]
    \item\label{cond:noleaves} $\phi_n = 0$ for $n\le s$, \quad or
    \item\label{cond:minimals} $s≤s_0+bs_1$, so either $s=s_0+bs_1$ or $s=s_0+bs_1-1$.
\end{enumerate}

Here, we model $\poissonn$ as the points in $[0,λ]$ of a standard Poisson point process on $\RR$. The time $E$ until the first $s_0 + (b-1)s_1$ points appear has the Gamma distribution $\Gamma(s_0 + (b-1)s_1,1)$ and conditional on $E$, the rest of the points is also a standard Poisson point process on $[E,λ]$, with total mass $(λ-E)^+$. The reason we chose $s_0 + (b-1)s_1$ is the following: Every node except possibly the root has at least $s_1$ balls. So it needs $s_0+(b-1)s_1$ more balls to distribute to itself and its children. After it reaches $s_0+bs_1$ balls, the rest is distributed according to the splitter.

Assume we start with $n = \poissonn + s_1$ balls and condition on $n>s$. The first $s_1b+s_0$ balls will be redistributed, leaving us with $\Poi((λ-E)^+)$ balls, which will be distributed multinomially to the $b$ children. The $i$-th child then has $n_i \sim \Poi((λ-E)^+V_i) + s_1$ children, which up to the changed parameter is the same as $n$. This does not hold in general if $n≤s$: If $s_0+bs_1<n≤s$, this model would distribute $n-s_0$ balls to its children, instead of holding all balls. This is why we need either condition \ref{cond:noleaves}, so that this difference does not count; or condition \ref{cond:minimals}, so that this case cannot happen. If neither is true, we would have to make a correction like in \Cref{sec:fringe}.

We now assume \ref{cond:noleaves} or \ref{cond:minimals}. 
We assume also, as usual, $\phi_n=\Ok{n^{\beta-\delta}}$ for some $\delta>0$,
and furthermore  (this is redundant if \ref{cond:noleaves} holds, and
trivial if $s_1=0$)
\begin{enumerate}[label=(A\arabic*)]
\setcounter{enumi}{2}
    \item\label{cond:noleaves-s1} $\phi_{s_1} = 0$. 
\end{enumerate}
We now define 
\begin{align}\label{def:s0-f}
    f(λ) &:= \mu\inv \Ek*{\phi_{\poissonn + s_1}}\eins(λ\ge0), & F(λ) &:= \mu\inv\Ek*{X_{\poissonn + s_1}}\eins(λ\ge0),
\end{align}
cf.\ \eqref{def:fl} where we had $s_1=0$.
Then the paragraph above and \eqref{eq:xn-recursion} show that
\begin{equation}\label{s0-recursion-unbiased}
    F(λ) = f(λ) + b\Ek*{F((ł-E)V)}.
\end{equation}
By the characterization of $\widehat V$ in \eqref{eq:v-bias}, we can replace $V$ by $\widehat V$ to obtain
\begin{equation}\label{s0-recursion}
%    λ^{-β}F(λ) = λ^{-β}f(λ) + \Ek*{(λ\widehat V)^{-β}F((ł-E)\widehat V)},
    F(λ) = f(λ) + \Ek*{\widehat V^{-β}F((ł-E)\widehat V)},
\end{equation}
which except for $E$ can be seen as a multiplicative renewal equation
by multiplying throughout by $λ^{-β}$.
Write
$E^{(1)},E^{(2)},\dots$ and $\widehat V^{(1)}, \widehat V^{(2)}\dots$ for
some independent copies of $E$ resp.\ $\widehat V$. Define $\widehat Y_k :=
\widehat V^{(1)}\cdots \widehat V^{(k)}$ similarly to the proof of \Cref{lem:rzeta} and expand \eqref{s0-recursion} a few steps:
\begin{align}
%    \frac{F(λ)}{λ^{β}} &= \!\begin{multlined}[t]
%    λ^{-β}f(λ) + \Ek*{(λ\widehat Y_1)^{-β}f(λ\widehat Y_1-E^{(1)}\widehat
%    Y_1)} \\+ \Ek*{(λ\widehat Y_2)^{-β}F\kl{λ\widehat Y_2-E^{(1)}\widehat
%    Y_2-E^{(2)}\widehat Y_1}} 
    F(λ) &= \!\begin{multlined}[t]
    f(λ) + \Ek*{\widehat Y_1^{-β}f(λ\widehat Y_1-E^{(1)}\widehat Y_1)} \\+ \Ek*{\widehat Y_2^{-β}F\kl{λ\widehat Y_2-E^{(1)}\widehat Y_2-E^{(2)}\widehat Y_1}}
    \end{multlined}
    \nonumber \\ &= \Ek[\Bigg]{\sum_{j=0}^{m-1} \widehat Y_j^{-β}f\kl[\bigg]{λ\widehat Y_j - \sum_{l=1}^j E^{(j)}\widehat Y_{j}}} + \Ek*{\widehat Y_m^{-β}F\kl[\bigg]{λ\widehat Y_m - \sum_{l=1}^m E^{(j)}\widehat Y_{j}}}.
    \label{s0-recursion-expanded}
\end{align}

As most mass comes from terms near the leaves of the tree, where in \eqref{s0-recursion-expanded} $j$ is large, we can reasonably assume that the sum $\sum_{l=1}^j E^{(j)}\widehat Y_{j}$ can be replaced by its limit for $j\to\infty$; which is the stationary distribution of the discrete-time Markov chain on $[0,\infty)$ with update step
\begin{equation}\label{def:xi-chain}
    x \mapsto \widehat V(x+E).
\end{equation}
(Remember here that $E$ is independent of $\widehat V$ and $E\sim \Gamma(s_0+(b-1)s_1, 1)$.)
Because $\Ek{\widehat V} < 1$, this update step is a contraction in the
$\ell_1$ metric, which implies that there exists a stationary distribution, 
i.e., a distribution such that a random variable $\Xi$ with this
distribution satisfies \eqref{def:xi};
moreover, it follows also that this stationary distribution is unique
among distributions with a finite mean. (In fact, it is easy to see that it
is unique among all distributions, for example by a coupling argument; we
omit the details since we do not need this.)

For a rigorous argument, we
use some Markov renewal theory, see Appendix \ref{sec:mrt}.
We define the functions, for $x,y\in\RR$,
\begin{align}\label{ma1}
   G(x,y) &= e^{-\beta y}f(e^y-x)
,\\\label{ma0}
  H(x,y) &= e^{-\beta y}F(e^y-x) 
%= \mu\inv\Ek*{X_{\poissonn[λ-x] + s_1}}\eins(λ>x)
.\end{align}
Then, letting $\lambda=e^y-x$,
we can rewrite \eqref{s0-recursion} as a Markov chain renewal equation
\begin{equation}\label{s0:z-recursion}
    H(x,y) 
= G(x,y)
+ \Ek*{H\kl[\big]{(x+E)\widehat V,y + \ln \widehat V}}.
\end{equation}
Let $(M_n)_{n=0}^\infty$ be the Markov chain 
with $M_0=x$ and the update step \eqref{def:xi-chain},
and note that $(M_n,-\ln\widehat V^{(n)})_{n=0}^\infty$
is a Markov chain on $\RR_+\times\RR$ of the type in \eqref{mrw9};
furthermore, in this Markov chain the distribution of $-\ln V^{(1)}$ does
not depend on $M_0$, and  the variables $-\ln V^{(i)}$ are i.i.d.
Let $S_n:=-\ln \widehat Y_n = \sum_{i=1}^n (-\ln\widehat V^{(i)})$.
Then we obtain by iterating \eqref{s0:z-recursion}, for any $N\ge0$,
\begin{align}\label{ma3}
H(x,y) 
=  \Ekv[\bigg]x{\sum_{n=0}^{N-1} G(M_n,y-S_n)} 
+\Ek*{H(M_N,y-S_N)}. 
\end{align}

We note that the proof of \Cref{lem:poisson-mean} extends to the present
situation, and thus, as $\lambda\to\infty$,
\begin{align}
  \label{ma4}
F(\lambda)=\mu\inv\varpi\lambda^\beta + o(\lambda^\beta).
\end{align}
Furthermore, \ref{cond:noleaves-s1} implies that for $\lambda\le1$, 
we have $F(\lambda)=\Ok{\lambda}$, which together with \eqref{ma4} implies the
uniform bound
\begin{align}  \label{ma5}
F(\lambda)=\Ok{\lambda^\beta},
\qquad \lambda\ge0.
\end{align}
Consequently, \eqref{ma0} yields, for some constant $C$,
\begin{align}\label{ma6}
  |H(x,y)|\le C,
\qquad x\ge0,\, y\in\RR.
\end{align}

Let $N\to\infty$ in \eqref{ma3}.
We may (by treating the positive and negative parts separately)
assume that $\phi_n\ge0$ for all $n$, and thus $f(\lambda)\ge0$ and
$G(x,y)\ge0$. For the first term on the right-hand side of \eqref{ma3}, we
then may use monotone convergence.
For the last term in \eqref{ma3}, we note that by \eqref{ma0} we have
$H(M_N,y-S_N)=0$ when $e^{y-S_N}<M_N$.
As $N\to\infty$, the law of large numbers shows that a.s.\ 
$S_N\to \infty$ and thus $e^{y- S_N}\to0$, and since $M_N$ 
converges in distribution to $\Xi$ with $\P(\Xi>0)=1$
(by \eqref{def:xi}), it follows that
$\P(e^{y-S_N}<M_N)\to1$, and thus $H(M_N,y-S_N)\to0$ in probability;
this implies $\E\bigl[{H(M_N,y-S_N)}\bigr]\to0$ 
by dominated convergence, justified by \eqref{ma6}.
Consequently, taking the limits yields the renewal equation in the form
\begin{align}\label{ma7}
H(x,y) 
=  \Ekv[\bigg]x{\sum_{n=0}^{\infty} G(M_n,y-S_n)} 
%=(G*U_y)(t),
.\end{align}

To apply the Markov key renewal \cref{thm:mrw-renewal}, we need to be in the
nonlattice case (which we have assumed)
and check that the Markov chain defined by
\eqref{def:xi-chain} is Harris-recurrent. For the recurrence set, one can
choose $\br*{\frac23,1}$. This is recurrent: For every $x$ there is a
positive probability to reach this set, so to not reach the set, the chain
has to diverge, but since the expectation stays bounded, this a.s.\ does not
happen. To check condition \eqref{harris-condition}, fix a small interval
$[y_1,y_2] \subseteq \br*{\frac23,1}$. Then 
\begin{equation}\label{harris-check-1}
    \Pk*{\widehat V(x+E)\in [y_1,y_2]} = \Pk*{E \in \br*{\frac{y_1}{\widehat V}-x, \frac{y_2}{\widehat V}-x}}.
\end{equation}
Because $\Pk{V_{(b)}>0}>0$, there exists some $\alpha>0$ such that $\Pk[\big]{\widehat V\in \br*{\alpha,\frac12}}>0$.
If $\widehat V$ is in this interval $\br[\big]{\alpha, \frac12}$ and $0<x≤1$,
then the interval in \eqref{harris-check-1} is 
at least $y_2-y_1$ long 
and lies inside $\br[\big]{\frac13,\alpha\inv}$. 
\eqref{harris-check-1} can thus be bounded from below by $\Pk*{\widehat V\in \br{\alpha, \frac12}}>0$
times the minimal density of $E$ in $\br*{\frac13,\alpha\inv}$ times the Lebesgue measure.

The condition \eqref{cond:mrw-continuous} holds because $f(λ)$ is continuous
at every $λ\neq0$. To show \eqref{cond:mrw-riemann}, 
which is essentially direct
Riemann integrability integrated over $x$, note that  
the definition \eqref{def:s0-f} implies,
similarly to
\eqref{const:g-bound-big} and \eqref{const:g-bound-small}
and using again \ref{cond:noleaves-s1},
\begin{equation}\label{ma8}
    {f(λ)} = \begin{cases}
        \Ok*{λ} & \text{for } 0\le λ≤1, \\
        \Ok*{λ^{β-\delta}} & \text{for } λ\ge1.
    \end{cases}
\end{equation}
In particular, $f(\lambda)=\Ok*{λ^{β-\delta}}$ for all $λ\ge0$.
Hence, 
%$f(\lambda)= \Ok*{λ^{β-\delta}}\eins(\lambda\ge0)$ for all $\lambda\in\RR$
by \eqref{ma1}, assuming as we may $\delta\le\beta$,
\begin{equation}
  G(x,y) = \begin{cases}
        \Ok*{\eins(y\ge \ln x)} & \text{for } y\le0, \\
        \Ok*{e^{-\delta y}} & \text{for } y\ge0.
    \end{cases}
\end{equation}
Consequently the integrand in \eqref{cond:mrw-riemann} (with $\rho=1$, say)
is $\Ok*{1+(-\ln x)^+}$, and the condition holds because
$\E[(-\ln\Xi)^+]<\infty$, 
which in turn follows becase \eqref{def:xi} (or \eqref{def:xi-chain}) 
shows the stochastic domination $\Xi\ge \widehat V E$ and thus, using
\eqref{def:mu}, 
\begin{equation}
  \label{eq:lnXi}
  \E[(-\ln\Xi)^+]
\le
\E[-\ln\widehat V] + \E[(-\ln E)^+]
<\infty.
\end{equation}
Finally, note that $\E_x [-\ln\widehat V^{(1)}]=\E [-\ln\widehat V]=\mu$
for every $x$, and thus $\mu$ in \eqref{mrw:mu} equals our $\mu$ in
\eqref{def:mu}.

We can thus apply \Cref{thm:mrw-renewal} to \eqref{ma7} to show that for $\P_{\Xi}$-almost all $x\in\RR$,
\begin{align}
    \lim_{λ\to\infty} \mu H(x,y) &= \E\int_{-\infty}^\infty e^{-yβ}f(e^y-\Xi)\,\mathrm dy
    \nonumber
    \\&= \E\int_{\Xi}^\infty λ^{-1-β}f(λ-\Xi)\,\mathrm dλ
    \nonumber
    \\&= \E\int_{0}^\infty (λ+\Xi)^{-1-β}f(λ)\,\mathrm dλ.
    \label{varpi-incgamma0}
\end{align}
Applying \eqref{s0:z-recursion} once more shows 
(using dominated convergence and \eqref{ma6})
that \eqref{varpi-incgamma} also holds for $x=0$, 
so, by \eqref{ma4} and \eqref{ma0}, 
\begin{align}
\varpi 
= \lim_{λ\to\infty}\lambda^{-\beta}F(λ)
= \lim_{λ\to\infty}\mu H(0,\ln λ)
= \E\int_{0}^\infty (λ+\Xi)^{-1-β}f(λ)\,\mathrm dλ.
\label{varpi-incgamma}
\end{align}
Unfortunately, the negative moments 
$\E (λ+\Xi)^{-1-β}$
in \eqref{varpi-incgamma} are hard to
compute in general. However, we can compute the integrals in the following case.

\begin{proof}[Proof of \Cref{thm:s0}]
For the toll function $\delta_{≥k}(T_n) := \eins(n≥k)$ for $k>s_1$, the
function $f$ from \eqref{def:s0-f} is given as $f(ł) = μ\inv\Pk*{\poissonn
  \ge k-s_1}\eins(ł>0)$. This is the probability that the Poisson
point process has  at least $k-s_1$ points before $ł$, and the time $E_{k-s_1}$ up to the $k-s_1$-th point is $\Gamma(k-s_1,1)$-distributed. So $f(ł) = μ\inv\Pk*{E_{k-s_1}≤ł}$, and thus by \eqref{varpi-incgamma}
\begin{align}
    \varpi[\delta_{≥k}] &= \frac1{μ}\int_0^\infty ł^{-1-β}\Pk*{\Xi+E_{k-s_1}≤ł}\mathrm dł
    \nonumber \\ &= \frac1{μ} \Ek[\bigg]{\int_{\Xi+E_{k-s_1}}^\infty λ^{-1-β} \mathrm dλ}
    \nonumber \\ &= \frac1{βμ}{\Ek*{(\Xi+E_{k-s_1})^{-β}}}
    \label{s0:delta-k}
\end{align}
using Fubini's theorem. Note that we assumed in \Cref{thm:s0} that Condition
\ref{cond:noleaves} or \ref{cond:minimals} holds, and also
\ref{cond:noleaves-s1}.
This shows the equality in \eqref{thm:s0:delta-gk}; to show the asymptotic behavior, note that $E_{k-s_1}/k \to 1$ a.s.\ by the law of large numbers (if we model $E_{k-s_1}$ as the sum of $k-s_1$ independent exponentials) and hence also $k^β(\Xi+E_{k-s_1})^{-β} \to 1$ a.s. To show uniform integrability, note that for $k≥s_1+2$
\begin{equation}
    \Ek*{k^{2β}(\Xi+E_{k-s_1})^{-2β}} ≤ \Ek*{k^{2β}E_{k-s_1}^{-2β}} = k^{2β}\frac{\Gamma(k-s_1-2β)}{\Gamma(k-s_1)} \to 1.
\end{equation}
It follows that \eqref{s0:delta-k} is asymptotically equal to $\frac1{βμk^β}$ as claimed.

For \eqref{thm:s0:delta-k}, define $E'_1\sim \Exp(1)$ to be independent of $E_{k-s_1}$, so that $E_{k-s_1+1} \overset d= E_{k-s_1} + E_1'$. We obtain from \eqref{thm:s0:delta-k}
\begin{equation}
    \varpi\br{\fringephi} = \frac1{βμ}\kl*{\Ek*{(\Xi+E_{k-s_1})^{-β}} - \Ek*{(\Xi + E_{k-s_1} + E_1')^{-β}}}.
\end{equation}
The mean value theorem shows that for some random $\Theta \in [0,1]$, 
\begin{equation}\label{s0:dk}
    \varpi[\fringephi] = \frac1{ μk^{β+1}}\Ek*{E_1'k^{β+1}(\Xi+E_{k-s_1}+\Theta E_1')^{-β-1}}
\end{equation}
The same argument as for \eqref{s0:delta-k} shows that the random variables in the expectation in \eqref{s0:dk} converge a.s.\ to $E_1'$ and are uniformly integrable. Hence, we can show \eqref{thm:s0:delta-k} by
\begin{equation}
    \varpi[\fringephi] \sim  \frac1{βμk^{β+1}}\Ek*{E_1'} = \frac1{βμk^{β+1}}.
\end{equation}
\end{proof}

Lastly, we prove \Cref{rem:beta}:
\begin{proof}
    Because $V\sim B(a,s_0+(b-1)s_1)$, we have $\widehat V\sim B(a+β, s_0+(b-1)s_1)$. The formula \eqref{rem:beta:mu} can be considered a basic fact of the beta distribution, obtained by
    \begin{align}
        \Ek*{\ln \widehat V} &= \ddat x0{\Ek*{\widehat V^x}} = \ddat x0{\frac{B(a+β+x,s_0+(b-1)s_1)}{B(a+β, s_0+(b-1)s_1)}}
        \nonumber \\ &= \ddat x0{\frac{\Gamma(a+β+x)\Gamma(a+β+s_0+(b-1)s_1)}{\Gamma(a+β+x+s_0+(b-1)s_1)\Gamma(a+β)}}
        \nonumber \\ &= \frac{\Gamma'(a+β)}{\Gamma(a+β)} - \frac{\Gamma'(a+β+s_0+(b-1)s_1)}{\Gamma(a+β+s_0+(b-1)s_1)}.
    \end{align}
    We claim that $\Xi \sim \Gamma(a+β, 1)$ solves \eqref{def:xi}: The sum $\Xi+E_{s_0+(b-1)s_1}$ is the sum of two independent Gamma-distributed variables, so it is $\Gamma(a+β+ s_0+(b-1)s_1, 1)$-distributed.
    Now we can use the fact that for any $a,b>0$, the product of a $\Gamma(a+b,1)$-distributed variable with an independent $B(a,b)$-distributed variable is $\Gamma(a,1)$-distributed, and that therefore $(\Xi+E_{s_0+(b-1)s_1})V$ is $\Gamma(a+β,1)$ distributed, as $\Xi$.
    Hence, $\Xi + E_{k-s_1}$ has distribution
    $\Gamma(a+β+k-s_1,1)$ and thus
    \begin{equation}\label{rem:beta:moment}
        \Ek*{(\Xi+E_{k-s_1})^{-β}} = \frac{\Gamma(a+k-s_1)}{\Gamma(a+β+k-s_1)}.
    \end{equation}
    Then, \Cref{thm:s0} shows that
    \begin{align}
        \varpi[\fringephi] &= \frac1{βμ}\kl*{\frac{\Gamma(a+k-s_1)}{\Gamma(a+k-s_1+β)} - \frac{\Gamma(a+k+1-s_1)}{\Gamma(a+k+1-s_1+β)}}
        \nonumber \\ &= \frac{\Gamma(a+k-s_1)}{βμ\Gamma(a+k-s_1+β)}\kl*{1-\frac{a+k-s_1}{a+k-s_1+β}}
        \nonumber \\ &= \frac{\Gamma(a+k-s_1)}{μ\Gamma(a+k-s_1+β)\kl{a+k-s_1+β}}
    \end{align}
    Combining the two factors in the denominator  into a single $\Gamma(a+k-s_1+β+1)$ gives \eqref{rem:beta:1}.
\end{proof}

% There are three (!) ways to calculate the mean if $s_0+s_1>0$}
% \begin{itemize}\color{magenta}
%     \item The number of balls missing in a node are part of a Harris recurrent Markov chain. Therefore, we can use a MC renewal theorem to get
%     \begin{equation}
%         \varpi = \frac1{μ}\int_m^\infty \int \Ek*{\phi(T_{\poissonn[\nu-m] + s_1})} \nu^{-β-1}\mathrm d\xi(m)\mathrm d\nu,
%     \end{equation}
%     where $\xi$ is the stationary distribution of the Markov chain on $[0,\infty)$ which in every step
%     \begin{itemize}
%         \item Multiply by $\widehat V$
%         \item Add an $\Exp(1)$ variable
%     \end{itemize}
%     \item Instead notice that the tree can be (randomly) constructed from a tree with $s_0=s_1=0$ by an operation called \newcommand\shrink{\operatorname{shrink}}$\shrink$. This leads to a small toll function
%     \begin{equation}
%         \phi'(T_n) = \phi(\shrink(T_n)) + \sum_{i=1}^b X(\phi)(\shrink(T_{n_v})) - X(\phi)(\shrink(T_{n_v'})).
%     \end{equation}
%     Unfortunately, the lower part is not easy to calculate at all.
%     \item Last, we could look at a geometrically distributed $n$, where we get the recursion
%     \begin{equation}
%         F(x) = f(x) + \fracc x{x+1}^{s_0+(b-1)s_1}b\Ek*{F(xV)},
%     \end{equation}
%     that might be solvable in special cases.
% \end{itemize}

\appendix
\section{Some renewal theory}\label{sec:rt}

In this paper, we repeatedly make use of following version of the key
renewal theorem for excessive renewal equations, adapted from
\cite{asmussen}. All renewal equations in the paper use the same
waiting-time distribution $b\P_{-\ln V}$, so we state the theorem using $V,
\widehat V, M_v^1, \mu$ and $\beta$ from before, see 
\eqref{def:beta}, \eqref{def:v-bias}, \eqref{def:mu} and \eqref{def:mass}. For this waiting-time distribution, the \emph{renewal function} $U(t)$ is defined as
\begin{equation}
    U(t) = \sum_{k=0}^\infty b^k\Pk{-\ln(Y_k)≤t},
\end{equation}
where $Y_k$ is the product of $k$ independent copies of $V$. Recalling the
definition of convolution in \eqref{def:convolution}, $U(t)$ fulfils the
convolution equation 
\begin{align}\label{rt:1}
U(t) = \eins(t≥0) + (b\P_{-\ln V}*U)(t).  
\end{align}
$U(t)$ is a monotonously growing function, and the associated measure $\mathrm dU(t)$ is called \emph{renewal measure}.

Because $b\P_{-\ln V}$ has total mass $b>1$, \eqref{rt:1} 
is a so-called \emph{excessive} renewal equation, which can be shifted to a \emph{proper} renewal equation using a probability measure by size-biasing $V$ with $e^{\beta\ln V} = V^\beta$, which gives us the probability measure $\P_{-\ln \widehat V}$, see \eqref{def:v-bias}. Then,
    $\widehat U(t) := e^{-\beta t}U(t)$ fulfills the proper renewal equation
    \begin{equation}\label{rt:2}
\widehat U(t) = e^{-\beta t}\eins(t≥0) + \kl*{\P_{-\ln \widehat V} * \widehat U}(t).
    \end{equation}
Using this shift to
$\P_{-\ln \widehat V}$ and the proper
renewal measure $\widehat U$, one can arrive at the following result. 

A function $g$ on $\RR$ is \emph{directly Riemann integrable} if it is
continuous almost everywhere, and 
    \begin{equation}\label{eq:dri}
        \sum_{n\in\mathbb{Z}}\sup_{n\rho\le q<(n+1)\rho} |g(q)| < \infty 
    \end{equation}
for some (and thus all) $\rho >0$.

\begin{theorem}[Key renewal theorem for excessive renewal equations]\label{thm:krt}
   Let $G$ be a function  on $\RR$ and let $Z=U\ast G$,
    so that $Z$ is the solution of the renewal equation
    \begin{equation}\label{rt:z-as-eq}
        Z = G + b\P_{-\ln V}*Z.
    \end{equation}
   If\/ $G(q)e^{-βq}$ 
%    \begin{equation}
%        \sum_{n\in\mathbb{Z}}\sup_{n\rho\le q<(n+1)\rho} |e^{-βq}G(q)| <
%        \infty \text{ for some } \rho >0, 
%    \end{equation}
    is directly Riemann integrable, then, for $q\to\infty$, in the nonlattice case 
    \begin{equation}
        Z(q) = \frac{e^{βq}}{µ}\int_{-\infty}^\infty G(u)e^{-βu}\mathrm du + o(e^{βq})
    \end{equation}
     and in the lattice case
    \begin{equation}
        Z(q) = \frac{de^{βq}}{µ}\sum_{u\in \ZZ} G(q+du)e^{-β(q+du)} + o(e^{βq}).
    \end{equation}
\end{theorem}
\begin{proof}
    % To get from \eqref{rt:z-as-sum} to \eqref{rt:z-as-eq}, note that the sum decomposes as
    % \begin{equation}
    %     Z(q) = \sum_{v\in [b]^\ast} \Ek*{G(q+\ln M^1_v)} = G(q) + \sum_{i=1}^b \sum_{v\in [b]^\ast} \Ek*{G(q+\ln M^1_{iv})}.
    % \end{equation}
    % Then, $\ln M^1_{iv}$ at a random $i$ is distributed as $\ln V + \ln M^1_v$ with $V$ and $M^1_v$ independent. Conditioning on $V$ gives
    % \eqref{rt:z-as-eq}.
    %
    See Theorem V.7.1 in \textcite{asmussen}. 
Note that $\beta$ in \cite{asmussen} corresponds to $-β$ here, given as the value such that $\int_0^\infty e^{-βq}\mathrm db\P_{-\ln V}(q) = b\Ek*{V^β} = 1$,
    and $\mu$ is given as $\int_0^\infty qe^{-βq}\mathrm db\P_{-\ln V}(q) = b\Ek*{-V^β\ln V} = \Ek*{-\ln\widehat V} = µ$.
    Theorem~V.7.1 in \cite{asmussen}
is only stated for the nonlattice case and for functions
    $G$ with $G(q) = 0$ if $q<0$, but it extends to our setting as follows: 
    
For the lattice case, we use the same tilting as in Theorem~V.7.1 and then use the key renewal theorem stated earlier in Proposition~V.4.8.

    The proof of the key renewal theorem V.4.7, which forms the base of Theorem V.7.1, also works for functions $G$ directly Riemann integrable and supported on the whole real line. The base idea of the proof is to write $Z$ as
    \begin{equation}
        Z(x) = \int_0^\infty G(x-y)\mathrm dU(y),
    \end{equation}
    and interpret this a Riemann-Stieltjes integral, with boxes of size $h>0$ centered around $x$. For any $y,h>0$, it is known that $U(y+h)-U(y)≤U(h)$ and $U(y+h)-U(y) \to \frac h{μ}$ (in the nonlattice case) for $y\to\infty$ by Theorem V.2.4(iii) and Blackwell's renewal theorem V.4.4 in \cite{asmussen}. The first fact guarantees a uniform upper bound and the second fact then shows convergence. The lattice case is similar.
\end{proof}

\subsection{Markov renewal theory}\label{sec:mrt}

A Markov random walk is a generalization of a renewal process where the independent renewal steps are replaced by a Markov chain. 
For the reader's convenience, we include its definition and the Markov renewal theorem from Alsmeyer \cite{alsmeyer_1994, alsmeyer_markov_1997}:

Let $(M_n,Y_n)$ be a temporally homogeneous Markov chain on a measurable space $(S, \mathcal S)$ with countably generated $\sigma$-field $\mathcal S$ and transition kernel $\mathbf P: \mathcal S \times (\mathcal S\otimes \mathfrak B) \to [0,1]$, $\mathfrak B$ being the Borel $\sigma$ field on $\RR$, such that
\begin{equation}\label{mrw9}
    \Pk{M_{n+1}\in A, Y_{n+1}\in B\given M_n, Y_n} = \mathbf P(M_n, A\times B) \quad \text{a.s.},
\end{equation}
so that $(M_{n+1}, Y_{n+1})$ only depends on $(M_n, Y_n)$. The corresponding Markov random walk (MRW) is $(M_n, S_n)$, where $S_n := \sum_{i=1}^n Y_i$. As usual, we denote the starting state / distribution of $(M_n)$ with a subscript, so $\P_x(\cdot)$ and $\E_x[\cdot]$ indicate that $M_0=x$.

$M_n$ is a \emph{Harris-recurrent chain}, 
%as defined in \cite{alsmeyer_markov_1994,alsmeyer_markov_1997},
if there exist some $α \in (0, 1], r ≥ 1$, a measurable \emph{recurrence set} $\mathcal R \subset \mathcal S$, i.e.\ $M_n ∈ \mathcal R$ a.s.\ infinitely often for any starting point $x ∈ \mathcal S$, and a probability measure $\phi$ on $\mathcal R$, such that 
\begin{equation}\label{harris-condition}
    \Pkv x{M_r \in\mathrm dy} ≥ α\phi(\mathrm dy)\text{ for all }x,y\in \mathcal R.
\end{equation}

\begin{theorem}[\textcite{alsmeyer_1994,alsmeyer_markov_1997}]
\label{thm:mrw-renewal}
Let \(\kl{(M_n,S_n)}_{n\ge 0}\) be a nonlattice MRW with Harris-recurrent driving chain \(\kl{M_n}_{n\ge0}\) having stationary measure \(\xi \) on \(\mathcal S\),
and let
\begin{align}\label{mrw:mu}
\mu:=\int_{\mathcal S} \E_x[Y_1]\,\xi(\mathrm d x)
= \E_\xi[Y_1]
\in(0,\infty]  
.\end{align}
If \(G:\:\mathcal S\times\mathbb{R}\to\mathbb{R}\) is a measurable function satisfying
\begin{align} \label{cond:mrw-continuous}
    \text{\(y\mapsto G(x,y)\) is Lebesgue-a.e.\ continuous for \(\xi\)-almost every \(x\in \mathcal S\), and} \\
    \label{cond:mrw-riemann}
    \int_{\mathcal S} \sum_{n\in\mathbb{Z}}\sup_{n\rho\le y<(n+1)\rho} |G(x,y)|\,\xi(\mathrm d x)<\infty \quad \text{ for some }\rho>0,
\end{align}
then
\begin{equation}
(G * U_x)(t) :=\Ekv[\bigg]x{\sum_{n\ge0} G(M_n,t-S_n)} \longrightarrow \frac{1}{\mu}\int_{\mathcal S}\!\int_{\mathbb{R}} G(u,v)\,\mathrm d v\,\xi(\mathrm d u)
\end{equation}
holds for $\xi$-almost all $x\in \mathcal S$ as $t\to\infty$ with the convention that the right-hand side is \(0\) if \(\mu=\infty \).
\end{theorem}

\printbibliography

\end{document}